\documentclass[a4paper,11pt]{amsart}

\usepackage[left=2.5cm,right=2.5cm,top=2.35cm,bottom=2.35cm]{geometry}
\usepackage[labelformat=simple]{subcaption}
\usepackage{amsmath,amssymb,amsthm,enumitem,tikz,hyperref,xcolor}
\usetikzlibrary{decorations.pathreplacing,decorations.pathmorphing,decorations.markings}

\theoremstyle{plain}
\newtheorem{thm}{Theorem}[section]
\newtheorem{prop}[thm]{Proposition}
\newtheorem{lem}[thm]{Lemma}
\newtheorem*{cm}{Claim}
\newtheorem*{obs}{Observation}
\theoremstyle{definition}
\newtheorem{defn}[thm]{Definition}
\newtheorem{ass}[thm]{Assumption}
\theoremstyle{remark}
\newtheorem{rmk}[thm]{Remark}

\newcommand{\XH}{X \cup \mathcal{H}}
\newcommand{\Cay}{\operatorname{Cay}}
\newcommand{\Stab}{\operatorname{Stab}}
\newcommand{\CayG}{\Cay(G,\XH)}
\newcommand{\acts}{\curvearrowright}
\newcommand{\N}{\mathbb{N}}
\newcommand{\Vtf}{\widetilde{V}_{\mathrm{free}}}
\newcommand{\Vtm}{\widetilde{V}_{\mathrm{med}}}
\newcommand{\Vti}{\widetilde{V}_{\mathrm{int}}}
\newcommand{\Vf}{V_{\mathrm{free}}}
\newcommand{\Vm}{V_{\mathrm{med}}}
\newcommand{\Vi}{V_{\mathrm{int}}}
\DeclareMathAlphabet{\mathsfit}{T1}{\sfdefault}{\mddefault}{\sldefault}
\newcommand{\pathXH}[1]{\mathsfit{#1}}
\newcommand{\wordXH}[1]{\mathcal{#1}}

\newenvironment{cmproof}[1][Proof of Claim]{%
	\begingroup\begin{proof}[#1]%
}{%
	\end{proof}\endgroup
}

\title{Helly groups, coarsely Helly groups, and relative hyperbolicity}
\newcommand{\wrocaddr}{Instytut Matematyczny, Uniwersytet Wroc{\l}awski, plac Grunwaldzki 2/4, 50-384 Wroc{\l}aw, Poland}

\author{Damian Osajda}
\address{\wrocaddr}
\address{Institute of Mathematics, Polish Academy of Sciences, \'Sniadeckich 8, 00-656 Warszawa, Poland}
\email{dosaj@math.uni.wroc.pl}
\thanks{The first author was partially supported by (Polish) Narodowe Centrum Nauki, UMO-2017/25/B/ST1/01335.}

\author{Motiejus Valiunas}
\address{\wrocaddr}
\email{valiunas@math.uni.wroc.pl}

\begin{document}

\begin{abstract}
A simplicial graph is said to be (\emph{coarsely}) \emph{Helly} 
if any collection of pairwise intersecting balls has non-empty (coarse) intersection. (\emph{Coarsely}) \emph{Helly groups} are groups acting geometrically on (coarsely) Helly graphs. Our main result is that finitely generated groups that are hyperbolic relative to (coarsely) Helly subgroups are themselves (coarsely) Helly. One important consequence is that various classical groups, including toral relatively hyperbolic groups, are equipped with a CAT(0)-like structure -- they act geometrically on spaces with convex geodesic bicombing. As a means of proving the main theorems we establish a result of independent interest concerning relatively hyperbolic groups: a `relatively hyperbolic' description of geodesics in a graph on which a relatively hyperbolic group acts geometrically. In the other direction, we show that for relatively hyperbolic (coarsely) Helly groups their parabolic subgroups are (coarsely) Helly as well. More generally, we show that `quasiconvex' subgroups of (coarsely) Helly groups are themselves (coarsely) Helly.
\end{abstract}
\maketitle
\setcounter{tocdepth}{1}
\tableofcontents

\section{Introduction} \label{sec:intro}

We consider any (simplicial) graph as its vertex set equipped with the combinatorial metric. We say a graph $\Gamma$ is \emph{Helly} (respectively, \emph{coarsely Helly}) if given any collection of balls $\{ B_{\rho_i}(x_i) \mid i \in \mathcal{I} \}$ in $\Gamma$ such that $B_{\rho_i}(x_i) \cap B_{\rho_j}(x_j) \neq \varnothing$ for any $i,j \in \mathcal{I}$, we have $\bigcap_{i \in \mathcal{I}} B_{\rho_i}(x_i) \neq \varnothing$ (respectively, $\bigcap_{i \in \mathcal{I}} B_{\rho_i+\xi}(x_i) \neq \varnothing$ for some universal constant $\xi \geq 0$). 
We say a group is \emph{Helly} (respectively, \emph{coarsely Helly}) if it acts geometrically on a Helly (respectively, coarsely Helly) graph. See Definition~\ref{defn:helly} for more details.

Intuitively, a group $G$ is said to be hyperbolic relative to subgroups $H_1,\ldots,H_m \leq G$ if the subgroups $H_j$ are some (possibly non-hyperbolic) subgroups ``arranged in $G$ in a hyperbolic way''. The reader may refer to Section~\ref{ssec:relhyp} for precise definitions, notation and terminology on relatively hyperbolic groups.

Based on a result of Lang \cite{Lang2013} it was shown in \cite{ccgho} that all (Gromov) hyperbolic groups are Helly. Our main results are as follows.

\begin{thm}[Hellyness of Relatively Hyperbolic Groups] \label{thm:helly}
Let $G$ be a finitely generated group that is hyperbolic relative to a collection of Helly subgroups. Then $G$ is Helly.
\end{thm}

As a step towards proving Theorem~\ref{thm:helly} we prove the following `coarse' version.

\begin{thm}[Coarse Hellyness of Relatively Hyperbolic Groups] \label{thm:coarsehelly}
	Let $G$ be a finitely generated group that is hyperbolic relative to a collection of coarsely Helly subgroups. Then $G$ is coarsely Helly.
\end{thm}

Helly graphs are classical objects that have been studied intensively within the metric and algorithmic graph theory for decades. They are also known as \emph{absolute retracts} (in the category of simplicial graphs with simplicial morphisms), and are universal, in the sense that every graph embeds isometrically into a Helly graph -- cf.\ e.g.\ the survey \cite{BanChe2008}. The study of groups acting on such graphs was initiated
recently in \cite{ccgho}, with non-positive aspects of flag completions of Helly graphs as a starting point.

Many (non-positive curvature)-like properties of Helly groups, including biautomaticity, the coarse Baum--Connes conjecture and the Farrell--Jones conjecture, were shown in \cite{ccgho}. Most of these properties, however, are known to hold for groups hyperbolic relative to subgroups satisfying these properties \cite{rebecchi,fukaya-oguni,bartels}. Nevertheless, Theorem~\ref{thm:helly} provides an alternative proof of these results for groups that are hyperbolic relative to Helly subgroups, for instance toral relatively hyperbolic groups.

However, our Theorem~\ref{thm:helly} provides new results for groups in question, e.g.\ equipping them with a fine local geometry, as follows. One remarkable feature of Helly groups proved in \cite{ccgho} is that they act geometrically on injective metric spaces -- injective hulls of corresponding Helly graphs -- and hence on \emph{spaces with convex geodesic bicombing}. The latter are geodesic spaces in which one has a chosen geodesic between any two points, and the family of such geodesics satisfies
some strong convexity properties reminiscent of CAT(0) geodesics. This allows one to obtain many CAT(0)-like results for groups acting geometrically on such spaces -- see e.g.\ \cite{descombes2016flats}. Theorem~\ref{thm:helly}
equips finitely generated groups hyperbolic relative to Helly groups with such a fine CAT(0)-like structure.
As an example, one refines this way the geometry of toral relatively hyperbolic groups -- a classical and widely studied class of groups hyperbolic relative to abelian subgroups (cf.\ e.g.\ \cite{DahGro2008}).
It is so because all finitely generated abelian groups and, more generally, all CAT(0) cubical groups are
Helly \cite{ccgho}. For further examples of Helly groups and group-theoretic constructions preserving Hellyness see \cite{ccgho,HO2019}.

The notion of coarsely Helly graphs has been introduced in \cite{ccgho}. The motivation was that -- as proved in \cite{ccgho} -- groups acting geometrically on coarsely Helly graphs satisfying an additional condition (of having stable intervals) are Helly -- see Theorem~\ref{thm:ccgho-helly}. This way the notion of coarse Hellyness may be seen as a means for
proving Hellyness, and this is exactly what is happening in this article -- see Section~\ref{sec:helly}.
On the other hand, the class of coarsely Helly groups seem to be of interest on its own.
Recently it has been shown in \cite{HHP2020} that mapping class groups of surfaces and, more generally, hierarchically hyperbolic groups act geometrically on coarsely injective spaces -- non-discrete analogues of coarsely Helly graphs. 

\medskip

In the direction converse to Theorems~\ref{thm:helly}~\&~\ref{thm:coarsehelly}, we prove the following.

\begin{thm}[Parabolic Subgroups of (Coarsely) Helly Groups] \label{thm:parabolics}
	Let $G$ be a group hyperbolic relative to $\{ H_1,\ldots,H_m\}$.
	If $G$ is Helly then $H_1,\ldots,H_m$ are Helly.
	If $G$ is coarsely Helly then $H_1,\ldots,H_m$ are coarsely Helly.	
\end{thm}

The parabolic subgroups $H_1,\ldots,H_m < G$ as above are known to be strongly quasiconvex \cite{drutu-sapir}. Roughly speaking, this means that if $G$ acts on a graph $\Gamma$ geometrically then for any $\lambda \geq 1$ and $c \geq 0$, all $(\lambda,c)$-quasi-geodesic paths with endpoints in an orbit $x \cdot H_i$ are bounded distance away from $x \cdot H_i$; see Definition~\ref{defn:quasiconvex}. Strongly quasiconvex subgroups (and subsets) are also known in the literature as Morse subgroups (and subsets), and have been extensively studied: see \cite{acgh,tran}, for instance. We also consider a weaker condition of being semi-strongly quasiconvex, where we replace $(\lambda,c)$-quasi-geodesics with $(1,c)$-quasi-geodesics.

Theorem~\ref{thm:parabolics} is then an immediate consequence of the following theorem concerning arbitrary (coarsely) Helly groups. We believe the result is of its own interest, providing further examples of such groups.

\begin{thm}[Quasiconvex Subgroups of (Coarsely) Helly Groups] \label{thm:quasiconvex}
	Let $\Gamma$ be a locally finite graph, let $G$ be a group acting on $\Gamma$ geometrically, and let $H \leq G$ be a subgroup. Then the following hold.
	\begin{enumerate}[label=\textup{(\roman*)}]
		\item \label{it:qconvex-helly} If $\Gamma$ is Helly and $H$ is strongly quasiconvex with respect to $\Gamma$, then $H$ is Helly.
		\item \label{it:qconvex-coarse} If $\Gamma$ is coarsely Helly and $H$ is semi-strongly quasiconvex with respect to $\Gamma$, then $H$ is coarsely Helly.
	\end{enumerate}
\end{thm}

We may use Theorem~\ref{thm:quasiconvex} to find new examples of groups that are not (coarsely) Helly. For instance, in \cite{hoda} Hoda characterised virtually nilpotent Helly groups, showing that they are all virtually abelian and act geometrically on $(\mathbb{R}^n,\|{-}\|_\infty)$ for some $n$. As a consequence, if $G$ is a group and $H < G$ is a strongly quasiconvex virtually nilpotent subgroup not satisfying these properties, then $G$ cannot be Helly.

\medskip

Our proof of Theorems~\ref{thm:helly}~\&~\ref{thm:coarsehelly} relies on the following result on relatively hyperbolic groups. We believe it is of independent interest from the point of view of the general theory of relative hyperbolicity. In what follows, we let $G$ be a finitely generated group hyperbolic relative to subgroups $H_1,\ldots,H_m$ and let $X$ be a finite generating set for $G$.
For $j \in \{1,\ldots,m\}$ and an integer $N \geq 1$, given a proper graph $\Gamma_j$ together with a geometric action $H_j \acts \Gamma_j$, we construct a (Vietoris--Rips) graph $\Gamma_{j,N}$ by adding edges to $\Gamma_j$ so that $v \sim w$ in $\Gamma_{j,N}$ if and only if $d_{\Gamma_j}(v,w) \leq N$. We then construct a graph $\Gamma(N)$ by, roughly speaking, taking a disjoint union of the barycentric subdivision of the Cayley graph $\Cay(G,X)$ and a copy of $\Gamma_{j,N}$ for each right coset of $H_j$ in $G$, and adding extra edges to make the graph connected. The graph $\Gamma(N)$ thus obtained comes equipped with a geometric action of $G$. The reader may refer to Section~\ref{ssec:GammaN} for precise definitions, and to Section~\ref{ssec:derived} for a construction of a `derived path' $\widehat{P} \subseteq \CayG$ given any path $P \subseteq \Gamma(N)$.

\begin{thm}[Geodesics to Quasi-geodesics] \label{thm:g->qg}
There exist constants $N \geq 1$, $\lambda \geq 1$, $c \geq 0$ and a finite collection $\Phi$ of non-geodesic paths in $\Gamma(N)$ with the following property. Let $P$ be a path in $\Gamma(N)$ having no parabolic shortenings and not containing any $G$-translate of a path in $\Phi$ as a subpath. Then the derived path $\widehat{P}$ in $\CayG$ is a $2$-local geodesic $(\lambda,c)$-quasi-geodesic that does not backtrack.
\end{thm}

\medskip

We prove Theorem~\ref{thm:g->qg} in Section~\ref{sec:g->qg}. Theorem~\ref{thm:coarsehelly} follows immediately from Proposition~\ref{prop:helly-coarse}. The latter, together with Theorem~\ref{thm:helly} are proved in Section~\ref{sec:helly}. Theorem~\ref{thm:quasiconvex} (implying Theorem~\ref{thm:parabolics}) is proved in Section~\ref{sec:quasiconvex}.

\medskip

\noindent \textbf{Acknowledgements.} We thank Oleg Bogopolski for directing us towards Proposition~\ref{prop:osin-polygons}, and the anonymous referee for their valuable feedback.

\section{Preliminaries} \label{sec:prelim}

\subsection{Graphs and hyperbolicity} \label{ssec:graphs}

In our setting, a \emph{graph} $\Gamma$ is a set $V(\Gamma)$ of vertices together with a multiset $E(\Gamma)$ of edges $\{v,w\}$ for $v,w \in V(\Gamma)$; in particular, we allow loops and multiple edges in a graph. By a \emph{path} $P$ in a graph $\Gamma$ we mean a combinatorial path, i.e.\ a sequence of vertices $v_0,\ldots,v_n \in V(\Gamma)$ and edges $\{ v_i,v_{i+1} \} \in E(\Gamma)$. In this case, $|P| := n$ is said to be the \emph{length} of $P$, and we write $P_- := v_0$ and $P_+ := v_n$ for the starting and ending vertices of a path $P$, respectively. A path of length $0$ is said to be \emph{trivial}. We also write $\overline{P}$ for the path from $P_+$ to $P_-$ following the same edges as $P$ just in opposite order. A \emph{subpath} of $P$ is a path consisting of consecutive edges in $P$. Furthermore, given paths $P_0,\ldots,P_k$ in $\Gamma$ such that $(P_{i-1})_+ = (P_i)_-$ for $1 \leq i \leq k$, we write $P_0 P_1 \cdots P_k$ for the path obtained by concatenating the paths $P_i$.

We require all our graphs $\Gamma$ to be connected, i.e.\ we assume that for any $v,w \in V(\Gamma)$ there exists a path $P$ in $\Gamma$ such that $P_- = v$ and $P_+ = w$. In this case, we view a graph $\Gamma$ as a metric space with underlying set $V(\Gamma)$ and a metric
\[
d_\Gamma(v,w) = \min \{ |P| \mid P \text{ is a path in } \Gamma, P_-=v, P_+=w \}.
\]
Given some constants $\lambda \geq 1$ and $c \geq 0$, a path $P = (v_0,\ldots,v_n)$ in a graph $\Gamma$ is said to be a \emph{$(\lambda,c)$-quasi-geodesic} if $|i-j| \leq \lambda d_\Gamma(v_i,v_j) + c$ whenever $0 \leq i,j \leq n$ (equivalently, if $|Q| \leq \lambda d_\Gamma(Q_-,Q_+) + c$ for every subpath $Q \subseteq P$). A $(1,0)$-quasi-geodesic is said to be a \emph{geodesic}. Given some $k \geq 2$, we also say a path $P$ is a \emph{$k$-local geodesic} (respectively, a \emph{$k$-local $(\lambda,c)$-quasi-geodesic}) if every subpath $Q \subseteq P$ with $|Q| \leq k$ is a geodesic (respectively, a $(\lambda,c)$-quasi-geodesic).

The following definition of hyperbolicity, which we state for graphs, is usually stated for geodesic metric spaces. However, it is easy to see that a graph $\Gamma$ is hyperbolic in the sense of Definition \ref{defn:hypgraph} if and only if it is quasi-isometric to a hyperbolic geodesic metric space in the usual sense. Moreover, even though Lemma~\ref{lem:delta-slim} and Theorem~\ref{thm:qgclose} below are usually stated for geodesic metric spaces, they can be easily seen to apply to graphs (under our terminology) as well.

\begin{defn} \label{defn:hypgraph}
Let $\lambda \geq 1$ and $\delta,c \geq 0$, and let $\Gamma$ be a graph. $\Delta = PQR$ is said to be a \emph{geodesic} (respectively, \emph{$(\lambda,c)$-quasi-geodesic}) \emph{triangle} in $\Gamma$ if $P,Q,R \subseteq \Gamma$ are geodesic (respectively, $(\lambda,c)$-quasi-geodesic) paths with $R_+=P_-$, $P_+=Q_-$ and $Q_+=R_-$. A geodesic triangle $\Delta = PQR$ in $\Gamma$ is said to be \emph{$\delta$-thin} if given any two vertices $u \in R$ and $v \in P$ (respectively, $u \in P$ and $v \in Q$, $u \in Q$ and $v \in R$) such that $d_\Gamma(P_-,u) = d_\Gamma(P_-,v) \leq \frac{|R|+|P|-|Q|}{2}$ (respectively, $d_\Gamma(Q_-,u) = d_\Gamma(Q_-,v) \leq \frac{|P|+|Q|-|R|}{2}$, $d_\Gamma(R_-,u) = d_\Gamma(R_-,v) \leq \frac{|Q|+|R|-|P|}{2}$), we have $d_\Gamma(u,v) \leq \delta$. The graph $\Gamma$ is said to be \emph{$\delta$-hyperbolic} if all its geodesic triangles are $\delta$-thin; we say that $\Gamma$ is \emph{hyperbolic} if it is $\delta$-hyperbolic for some constant $\delta \geq 0$.
\end{defn}

The following result is very well-known. A triangle $\Delta$ satisfying the assumptions of Lemma~\ref{lem:delta-slim} is said to be \emph{$\delta$-slim}.

\begin{lem}[see {\cite[Proposition 2.21]{ghys-harpe}}] \label{lem:delta-slim}
Let $\Gamma$ be a graph and let $\delta \geq 0$. Suppose that for any geodesic triangle $\Delta = PQR$ in $\Gamma$ and for any vertex $u \in R$, there exists a vertex $v \in PQ$ such that $d_\Gamma(u,v) \leq \delta$. Then $\Gamma$ is hyperbolic.
\end{lem}

We also use the following well-known result on hyperbolic metric spaces.

\begin{thm}[see {\cite[Th\'eor\`eme 5.11]{ghys-harpe}}] \label{thm:qgclose}
For any $\delta,c \geq 0$ and $\lambda \geq 1$, there exists a constant $\zeta = \zeta(\delta,\lambda,c) \geq 0$ with the following property. Let $\Gamma$ be a $\delta$-hyperbolic graph, let $P \subseteq \Gamma$ be a $(\lambda,c)$-quasi-geodesic, and let $Q \subseteq \Gamma$ be a geodesic with $Q_- = P_-$ and $Q_+ = P_+$. Then the Hausdorff distance between $P$ and $Q$ is at most $\zeta$.
\end{thm}

A particular case of a graph is the Cayley graph $\Cay(G,Y)$ of a group $G$ with respect to a (not necessarily finite) generating set $Y$. Formally, we view $Y$ as an abstract set together with a map $\epsilon: Y \to G$ such that the image of $\epsilon$ generates $G$: in particular, this allows considering Cayley graphs with multiple edges. We assume that $Y$ is equipped with an involution $\iota: Y \to Y$ such that $\epsilon(\iota(y)) = \epsilon(y)^{-1}$ for all $y \in Y$. For simplicity, we also assume that $1 \notin \epsilon(Y)$, so that the Cayley graphs we consider do not contain loops. The Cayley graph $\Cay(G,Y)$ then has $G$ as its vertex set and edge $e$ from $g$ to $\epsilon(y)g$ for any $g \in G$ and $y \in Y$. We label such a directed edge $e$ by the letter $y$, and identify its inverse $\overline{e}$ with the edge from $\epsilon(y)g$ to $g$ labelled by $\iota(y)$.

For simplicity of notation, we write $d_Y(g,h)$ for $d_{\Cay(G,Y)}(g,h)$ whenever $g,h \in G$. Moreover, for a path $P \subseteq \Cay(G,Y)$ and a symmetric generating set $\Omega$ of $G$ not containing the identity, we write $|P|_\Omega$ for $d_\Omega(P_-,P_+)$; we furthermore allow for the possibility that $\Omega \subseteq G$ is an arbitrary symmetric subset of $G$ not containing the identity, in which case we set $|P|_\Omega = \infty$ whenever $(P_+)(P_-)^{-1}$ is not in the subgroup generated by $\Omega$.

\begin{rmk}
All of this formalism regarding the generating sets might seem unnecessary. However, it allows us to use different `generators' to represent the same element of the group $G$. This allows simplifications in our arguments when $G$ is hyperbolic relative to a collection of subgroups $\{ H_1,\ldots,H_m \}$ in the case when the subgroups $H_j$ intersect non-trivially.
\end{rmk}

We also label any path $P$ in $\Cay(G,Y)$ by a word $y_1 \cdots y_n$ over $Y$ if $P$ is the path from $g$ to $\epsilon(y_n) \cdots \epsilon(y_1) g$ (for some $g \in G$) following the edges labelled by $y_1,\ldots,y_n$, in this order. We say a word over $Y$ is \emph{geodesic} (respectively, \emph{$(\lambda,c)$-quasi-geodesic}) if it labels a geodesic (respectively, $(\lambda,c)$-quasi-geodesic) path $P \subseteq \Cay(G,Y)$; note that this property does not depend on the choice of $P$. We similarly define $k$-local geodesic and $k$-local $(\lambda,c)$-quasi-geodesic words over $Y$.

\subsection{Relatively hyperbolic groups} \label{ssec:relhyp}

Our approach to relatively hyperbolic groups follows the approach of D.~V.~Osin \cite{osin06}.

Let $G$ be a group and let $H_1,\ldots,H_m$ be a finite collection of distinct subgroups of $G$. Suppose $G$ is finitely generated, i.e.\ there exists a surjective group homomorphism $\hat\epsilon: F(X) \to G$ for a finite set $X$; without loss of generality, suppose that $\hat\epsilon|_X$ is injective, and that $\hat\epsilon(X)$ is symmetric and does not contain $1 \in G$. Thus, $\hat\epsilon$ extends to a surjective homomorphism $\epsilon: F = F(X) * (\ast_{i=1}^m \widetilde{H}_i) \to G$, such that $\epsilon$ maps each group $\widetilde{H}_i$ isomorphically onto $H_i$. We say that $G$ is \emph{finitely presented with respect to $\{ H_1,\ldots,H_m \}$} if $\ker \epsilon$ is the normal closure (in $F$) of a finite subset $\mathcal{R} \subseteq F$. We also write
\begin{equation} \label{eq:relpres}
G = \langle X, \{ H_1,\ldots,H_m \} \mid \mathcal{R} \rangle
\end{equation}
for a \emph{relative presentation} of $G$, which is said to be \emph{finite} if $X$ and $\mathcal{R}$ are finite.

Now suppose $G$ is finitely presented with respect to $\{H_1,\ldots,H_m\}$ with a finite relative presentation \eqref{eq:relpres}. Let $\mathcal{H} = \bigsqcup_{j=1}^m (\widetilde{H}_j \setminus \{1\})$. We say $f: \N \to \N$ is a \emph{relative isoperimetric function} of the presentation \eqref{eq:relpres} if for all $n \geq 1$ and all words $W$ over $\XH$ of length $n$ with $\epsilon(W) =_G 1$ we have
\[
W =_F \prod_{i=1}^k g_i^{-1} R_i^{\pm 1} g_i
\]
for some $k \leq f(n)$ and some elements $g_i \in F$, $R_i \in \mathcal{R}$. A minimal relative isoperimetric function (if it exists) is called a \emph{relative Dehn function}, and a function $f: \N \to \N$ is said to be \emph{linear} if there exist $a,b \in \N$ such that $f(n) \leq an+b$ for all $n$.

\begin{defn}[Osin {\cite[Definition 2.35]{osin06}}]
We say that the group $G$ is \emph{hyperbolic relative to $\{ H_1,\ldots,H_m \}$} if $G$ is finitely presented with respect to $\{ H_1,\ldots,H_m \}$, and the relative Dehn function of this presentation exists and is linear.
\end{defn}

Now consider the Cayley graph $\CayG$, defined in Section~\ref{ssec:graphs} for $Y = \XH$. In this case, we take the involution $\iota: \XH \to \XH$ such that $\iota(X) = X$ and $\iota(\widetilde{H}_j \setminus \{1\}) = \widetilde{H}_j \setminus \{1\}$ for each $j$: together with the condition that $\epsilon(\iota(y)) = \epsilon(y)^{-1}$ for each $y \in \XH$, this defines $\iota$ uniquely.

We say a path $\pathXH{P}$ in $\CayG$ is an \emph{$H_j$-path} if all edges of $\pathXH{P}$ are labelled by elements of $\widetilde{H}_j$; an $H_j$-path of length $1$ is called an \emph{$H_j$-edge} or an \emph{$\mathcal{H}$-edge}. A maximal $H_j$-subpath of a path $\pathXH{P}$ is said to be an \emph{$H_j$-component} (or simply a \emph{component}). Given two paths $\pathXH{P}$ and $\pathXH{Q}$ in $\CayG$, an $H_i$-component $\pathXH{P}'$ of $\pathXH{P}$ is said to be \emph{connected} to an $H_j$-component $\pathXH{Q}'$ of $\pathXH{Q}$ if $i = j$ and $\pathXH{P}'_- ( \pathXH{Q}'_- )^{-1} \in H_j$ (note that this defines an equivalence relation). A component $\pathXH{P}'$ of a path $\pathXH{P}$ is said to be \emph{isolated} if it is not connected to any other component of $\pathXH{P}$, and we say a path $\pathXH{P}$ in $\CayG$ \emph{does not backtrack} if all its components are isolated. We also say that a subword of a word $\wordXH{P}$ over $\XH$ is a \emph{component} of $\wordXH{P}$ if the corresponding subpath of a path $\pathXH{P}$ labelled by $\wordXH{P}$ is a component of $\pathXH{P}$, and that $\wordXH{P}$ \emph{does not backtrack} if $\pathXH{P}$ does not backtrack.

A vertex $v$ of a path $\pathXH{P}$ is said to be \emph{non-phase} if it belongs to the interior of some component of $\pathXH{P}$, and \emph{phase} otherwise. Note that a geodesic path does not backtrack and all its vertices are phase. Given $k \in \N$, two paths $\pathXH{P},\pathXH{Q}$ are said to be \emph{$k$-similar} if $d_X(\pathXH{P}_-,\pathXH{Q}_-) \leq k$ and $d_X(\pathXH{P}_+,\pathXH{Q}_+) \leq k$. For future reference, given a path $\pathXH{P} = \pathXH{P}_1 \pathXH{U} \pathXH{P}_2$ in $\CayG$ we define $\pathXH{V}$ to be a \emph{subpath of $\pathXH{P}$ preceding $\pathXH{U}$} if $\pathXH{V}$ is a subpath of $\pathXH{P}_1$.

In our proofs of Theorem~\ref{thm:g->qg} and Proposition~\ref{prop:BCPtriangles} below, we use the following results due to D.~V.~Osin. In all of these results, $G$ is a (fixed) finitely generated group that is hyperbolic relative to subgroups $\{ H_1,\ldots,H_m \}$, $X$ is a finite generating set for $G$, and $\mathcal{H}$ is as above.

\begin{thm}[Bounded Coset Penetration; Osin {\cite[Theorem 3.23]{osin06}}] \label{thm:bcp}
For any $\lambda \geq 1$, $c \geq 0$ and $k \in \N$, there exists a constant $\varepsilon = \varepsilon(\lambda,c,k) \in \N$ with the following property. Let $\pathXH{P}$ and $\pathXH{Q}$ be two $k$-similar $(\lambda,c)$-quasi-geodesics in $\CayG$ that do not backtrack. Then
\begin{enumerate}[label=\textup{(\roman*)}]
\item \label{it:bcp-phase} for any phase vertex $u$ of $\pathXH{Q}$, there exists a phase vertex $v$ of $\pathXH{P}$ such that $d_X(u,v) \leq \varepsilon$;
\item \label{it:bcp-conn} for any component $\pathXH{Q}'$ of $\pathXH{Q}$ with $d_X(\pathXH{Q}'_-,\pathXH{Q}'_+) > \varepsilon$, there exists a component of $\pathXH{P}$ connected to $\pathXH{Q}'$; and
\item \label{it:bcp-close} for any two connected components $\pathXH{P}',\pathXH{Q}'$ of $\pathXH{P},\pathXH{Q}$ (respectively), we have $d_X(\pathXH{P}'_-,\pathXH{Q}'_-) \leq \varepsilon$ and $d_X(\pathXH{P}'_+,\pathXH{Q}'_+) \leq \varepsilon$.
\end{enumerate}
\end{thm}

\begin{thm}[Osin {\cite[Theorem 3.26]{osin06}}] \label{thm:triangles}
There exists a constant $\nu \in \N$ such that the following holds. Let $\Delta = \pathXH{PQR}$ be a geodesic triangle in $\CayG$. Then for any vertex $u$ of $\pathXH{P}$, there exists a vertex $v$ of $\pathXH{Q} \cup \pathXH{R}$ such that $d_X(u,v) \leq \nu$.
\end{thm}

\begin{prop}[Osin {\cite[Proposition 3.2]{osin07}}] \label{prop:osin-polygons}
For any $\lambda \geq 1$ and $c \geq 0$, there exists a finite subset $\Omega \subseteq G$ and a constant $L = L(\lambda,c) \in \N$ such that the following holds. Let $n \geq 1$, let $\pathXH{Q} = \pathXH{P}_1 \cdots \pathXH{P}_n$ be a closed path in $\CayG$, and suppose that there exists a subset $I \subseteq \{1,\ldots,n\}$ such that $\pathXH{P}_i$ is an isolated component of $\pathXH{Q}$ if $i \in I$ and a $(\lambda,c)$-quasi-geodesic otherwise. Then the $\Omega$-lengths of the $\pathXH{P}_i$ for $i \in I$ satisfy
\[
\sum_{i \in I} |\pathXH{P}_i|_\Omega \leq Ln.
\]
\end{prop}

\begin{rmk}
Theorems~\ref{thm:bcp} and~\ref{thm:triangles} are stated in \cite{osin06} only for generating sets $X \subseteq G$ that satisfy a certain technical condition stated in the beginning of \cite[\S 3]{osin06}. However, it is easy to see that the statement of Theorem~\ref{thm:bcp} is independent of the choice of a finite symmetric generating set $X$ (up to possibly increasing the constant $\varepsilon$). Using Theorem~\ref{thm:bcp}\ref{it:bcp-phase}, we can also show that Theorem~\ref{thm:triangles} holds independently of the choice of a finite symmetric generating set $X$.
\end{rmk}

We end our introduction to relatively hyperbolic groups by proving the following result, which we use in our proof of Theorem~\ref{thm:coarsehelly}. This can be viewed as a version of Theorem~\ref{thm:bcp} stated for triangles instead of `bigons'.

\begin{prop}
\label{prop:BCPtriangles}
For any $\lambda \geq 1$ and $c \geq 0$, there exists a constant $\mu = \mu(\lambda,c)$ with the following property. Let $\Delta = \pathXH{PQR}$ be a non-backtracking $(\lambda,c)$-quasi-geodesic triangle in the Cayley graph $\CayG$ (i.e.\ $\pathXH{P},\pathXH{Q},\pathXH{R}$ are $(\lambda,c)$-quasi-geodesics in $\CayG$ that do not backtrack such that $\pathXH{P}_+ = \pathXH{Q}_-$, $\pathXH{Q}_+ = \pathXH{R}_-$ and $\pathXH{R}_+ = \pathXH{P}_-$). Then
\begin{enumerate}[label=\textup{(\roman*)}]
\item \label{it:bcpt-phase} for any phase vertex $u$ of $\pathXH{R}$, there exists a phase vertex $v$ of $\pathXH{P}$ or of $\pathXH{Q}$ such that $d_X(u,v) \leq \mu$;
\item \label{it:bcpt-conn} for any component $\pathXH{R}'$ of $\pathXH{R}$ with $d_X(\pathXH{R}'_-,\pathXH{R}'_+) > \mu$, there exists a component of $\pathXH{P}$ or of $\pathXH{Q}$ connected to $\pathXH{R}'$;
\item \label{it:bcpt-sim2} if a component $\pathXH{R}'$ of $\pathXH{R}$ is connected to a component $\pathXH{P}'$ of $\pathXH{P}$ but is not connected to any component of $\pathXH{Q}$, we have $d_X(\pathXH{R}'_+,\pathXH{P}'_-) \leq \mu$ and $d_X(\pathXH{P}'_+,\pathXH{R}'_-) \leq \mu$; and
\item \label{it:bcpt-sim3} if a component $\pathXH{R}'$ of $\pathXH{R}$ is connected to a component $\pathXH{P}'$ of $\pathXH{P}$ and a component $\pathXH{Q}'$ of $\pathXH{Q}$, then $d_X(\pathXH{R}'_+,\pathXH{P}'_-) \leq \mu$, $d_X(\pathXH{P}'_+,\pathXH{Q}'_-) \leq \mu$ and $d_X(\pathXH{Q}'_+,\pathXH{R}'_-) \leq \mu$.
\end{enumerate}
\end{prop}

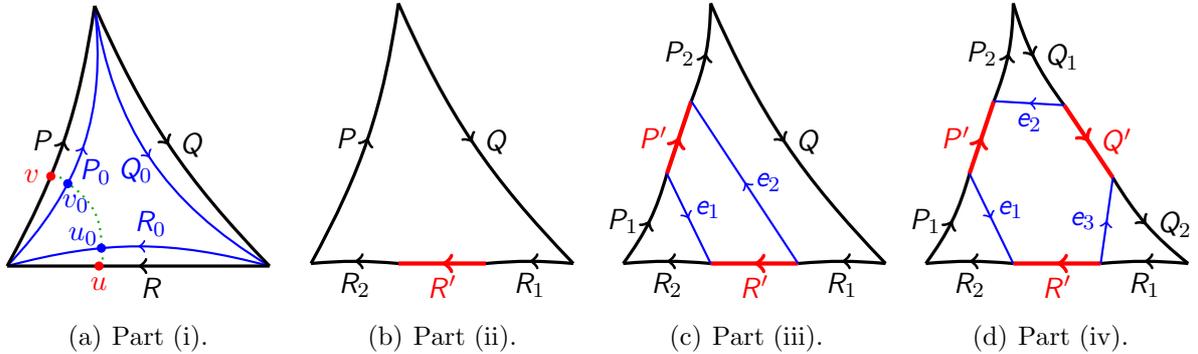
\begin{figure}[ht]

\begin{subfigure}[t]{0.25\textwidth}
\centering
\begin{tikzpicture}[very thick,scale=1.15,decoration={markings,mark=at position 0.5 with {\arrow{>}}}]
\draw [blue,thick,postaction=decorate] (3,0) node (c) {} to[bend right=15] node [pos=0.45,above] {$\pathXH{R}_0$} node[pos=0.65] (u0) {} (0,0) node (a) {};
\draw [blue,thick,postaction=decorate] (1,3) node (b) {} to[bend right=25] node [pos=0.55,left] {$\pathXH{Q}_0\!$} (c.center);
\draw [blue,thick,postaction=decorate] (a.center) to[bend right=25] node [pos=0.4,right] {$\!\pathXH{P}_0$} node[pos=0.35] (v0) {} (b.center);
\draw [postaction=decorate] (c.center) to node [pos=0.45,below] {$\pathXH{R}$} node[pos=0.65] (u) {} (a.center);
\draw [postaction=decorate] (b.center) to[bend right=10] node [midway,right] {$\pathXH{Q}$} (c.center);
\draw [postaction=decorate] (a.center) to[bend right=10] node [midway,left] {$\pathXH{P}$}  node[pos=0.35] (v) {} (b.center);
\draw [green!70!black,dotted,thick] (u.center) to[bend right] (u0.center) to[bend right] (v0.center) to[bend right] (v.center);
\fill [red] (u) circle (1.5pt) node [below] {$u$};
\fill [blue] (u0) circle (1.5pt) node [above left,yshift=-3pt,xshift=3pt] {$u_0$};
\fill [blue] (v0) circle (1.5pt) node [below,xshift=3pt] {$v_0$};
\fill [red] (v) circle (1.5pt) node [left] {$v$};
\end{tikzpicture}
\caption{Part~\ref{it:bcpt-phase}.}
\label{sfig:bcp3-phase}
\end{subfigure}%
\begin{subfigure}[t]{0.25\textwidth}
\centering
\begin{tikzpicture}[very thick,scale=1.15,decoration={markings,mark=at position 0.5 with {\arrow{>}}}]
\draw [postaction=decorate] (3,0) node (c) {} to[bend right=5] node [midway,below] {$\pathXH{R}_1$} (2,0) node (ca) {};
\draw [red,ultra thick,postaction=decorate] (ca.center) to node [midway,below] {$\pathXH{R}'$} (1,0) node (ac) {};
\draw [postaction=decorate] (ac.center) to[bend right=5] node [midway,below] {$\pathXH{R}_2$} (0,0) node (a) {};
\draw [postaction=decorate] (1,3) node (b) {} to[bend right=10] node [midway,right] {$\pathXH{Q}$} (c.center);
\draw [postaction=decorate] (a.center) to[bend right=10] node [midway,left] {$\pathXH{P}$} (b.center);
\end{tikzpicture}
\caption{Part~\ref{it:bcpt-conn}.}
\label{sfig:bcp3-conn}
\end{subfigure}%
\begin{subfigure}[t]{0.25\textwidth}
\centering
\begin{tikzpicture}[very thick,scale=1.15,decoration={markings,mark=at position 0.5 with {\arrow{>}}}]
\draw [white] (0,0) node (a) {} to[bend right=10] node [pos=0.35] (ab) {} node [pos=0.65] (ba) {} (1,3) node (b) {};
\draw [postaction=decorate] (3,0) node (c) {} to[bend right=5] node [midway,below] {$\pathXH{R}_1$} (2,0) node (ca) {};
\draw [red,ultra thick,postaction=decorate] (ca.center) to node [midway,below] {$\pathXH{R}'$} (1,0) node (ac) {};
\draw [postaction=decorate] (ac.center) to[bend right=5] node [midway,below] {$\pathXH{R}_2$} (a.center);
\draw [postaction=decorate] (a.center) to[bend right=10] node [midway,left] {$\pathXH{P}_1$} (ab.center);
\draw [red,ultra thick,postaction=decorate] (ab.center) to node [midway,left] {$\pathXH{P}'$} (ba.center);
\draw [postaction=decorate] (ba.center) to[bend right=10] node [midway,left] {$\pathXH{P}_2$} (b.center);
\draw [postaction=decorate] (b.center) to[bend right=10] node [midway,right] {$\pathXH{Q}$} (c.center);
\draw [blue,thick,postaction=decorate] (ab.center) to node [pos=0.4,right] {$\pathXH{e}_1$} (ac.center);
\draw [blue,thick,postaction=decorate] (ca.center) to node [pos=0.5,right] {$\pathXH{e}_2$} (ba.center);
\end{tikzpicture}
\caption{Part~\ref{it:bcpt-sim2}.}
\label{sfig:bcp3-sim2}
\end{subfigure}%
\begin{subfigure}[t]{0.25\textwidth}
\centering
\begin{tikzpicture}[very thick,scale=1.15,decoration={markings,mark=at position 0.5 with {\arrow{>}}}]
\draw [white] (0,0) node (a) {} to[bend right=10] node [pos=0.35] (ab) {} node [pos=0.65] (ba) {} (1,3) node (b) {};
\draw [white] (b.center) to[bend right=10] node [pos=0.35] (bc) {} node [pos=0.65] (cb) {} (3,0) node (c) {};
\draw [postaction=decorate] (c.center) to[bend right=5] node [midway,below] {$\pathXH{R}_1$} (2,0) node (ca) {};
\draw [red,ultra thick,postaction=decorate] (ca.center) to node [midway,below] {$\pathXH{R}'$} (1,0) node (ac) {};
\draw [postaction=decorate] (ac.center) to[bend right=5] node [midway,below] {$\pathXH{R}_2$} (a.center);
\draw [postaction=decorate] (a.center) to[bend right=10] node [midway,left] {$\pathXH{P}_1$} (ab.center);
\draw [red,ultra thick,postaction=decorate] (ab.center) to node [midway,left] {$\pathXH{P}'$} (ba.center);
\draw [postaction=decorate] (ba.center) to[bend right=10] node [midway,left] {$\pathXH{P}_2$} (b.center);
\draw [postaction=decorate] (b.center) to[bend right=10] node [midway,right] {$\pathXH{Q}_1$} (bc.center);
\draw [red,ultra thick,postaction=decorate] (bc.center) to node [midway,right] {$\pathXH{Q}'$} (cb.center);
\draw [postaction=decorate] (cb.center) to[bend right=10] node [midway,right] {$\pathXH{Q}_2$} (c.center);
\draw [blue,thick,postaction=decorate] (ab.center) to node [pos=0.4,right] {$\pathXH{e}_1$} (ac.center);
\draw [blue,thick,postaction=decorate] (bc.center) to node [pos=0.5,below] {$\pathXH{e}_2$} (ba.center);
\draw [blue,thick,postaction=decorate] (ca.center) to node [pos=0.5,left] {$\pathXH{e}_3$} (cb.center);
\end{tikzpicture}
\caption{Part~\ref{it:bcpt-sim3}.}
\label{sfig:bcp3-sim3}
\end{subfigure}

\caption{The proof of Proposition~\ref{prop:BCPtriangles}.}
\label{fig:bcp3}
\end{figure}

\begin{proof}
Let $\varepsilon = \varepsilon(\lambda,c,0) \in \N$ be as in Theorem~\ref{thm:bcp}, let $\nu \in \N$ be as in Theorem~\ref{thm:triangles}, and let $\Omega \subseteq G$ and $L = L(\lambda,c) \in \N$ be as in Proposition~\ref{prop:osin-polygons}. Since $|\Omega| < \infty$, we have $M := \sup_{\omega \in \Omega} |\omega|_X < \infty$; it follows that if $d_\Omega(g,h) \leq D$ for some $g,h \in G$ and $D \in \N$ then $d_X(g,h) \leq DM$. We set $\mu := \max \{ 2\varepsilon+\nu, 5ML \}$.

We now prove parts \ref{it:bcpt-phase}--\ref{it:bcpt-sim3} in the statement of the theorem.
\begin{enumerate}[label=(\roman*)]

\item Let $\pathXH{P}_0,\pathXH{Q}_0,\pathXH{R}_0 \subseteq \CayG$ be geodesics such that $(\pathXH{P}_0)_+ = \pathXH{P}_+ = \pathXH{Q}_- = (\pathXH{Q}_0)_-$, $(\pathXH{Q}_0)_+ = \pathXH{Q}_+ = \pathXH{R}_- = (\pathXH{R}_0)_-$ and $(\pathXH{R}_0)_+ = \pathXH{R}_+ = \pathXH{P}_- = (\pathXH{P}_0)_-$, so that $\Delta' = \pathXH{P}_0\pathXH{Q}_0\pathXH{R}_0$ is a geodesic triangle in $\CayG$; see Figure~\ref{sfig:bcp3-phase}. Thus the paths $\pathXH{P},\pathXH{Q},\pathXH{R},\pathXH{P}_0,\pathXH{Q}_0,\pathXH{R}_0$ are all $(\lambda,c)$-quasi-geodesics that do not backtrack.

Let $u \in \pathXH{R}$ be a phase vertex. By Theorem~\ref{thm:bcp}\ref{it:bcp-phase}, there is a phase vertex $u_0 \in \pathXH{R}_0$ such that $d_X(u,u_0) \leq \varepsilon$. By Theorem~\ref{thm:triangles}, there is a vertex $v_0$ of either $\pathXH{P}_0$ or $\pathXH{Q}_0$ such that $d_X(u_0,v_0) \leq \nu$; note that since $\pathXH{P}_0$ and $\pathXH{Q}_0$ are geodesics, $v_0$ is necessarily a phase vertex. Finally, by Theorem~\ref{thm:bcp}\ref{it:bcp-phase} again, there exists a phase vertex $v$ of $\pathXH{P}$ (if $v_0 \in \pathXH{P}_0$) or of $\pathXH{Q}$ (if $v_0 \in \pathXH{Q}_0$) such that $d_X(v_0,v) \leq \varepsilon$. We thus have
\[
d_X(u,v) \leq d_X(u,u_0) + d_X(u_0,v_0) + d_X(v_0,v) \leq 2\varepsilon+\nu \leq \mu,
\]
as required.

\item Suppose that $\pathXH{R}'$ is an isolated component of $\pathXH{PQR}$. We can then write $\pathXH{R} = \pathXH{R}_1 \pathXH{R}' \pathXH{R}_2$, so that $\pathXH{P}$, $\pathXH{Q}$, $\pathXH{R}_1$, $\pathXH{R}_2$ are all $(\lambda,c)$-quasi-geodesics, and $\pathXH{R}'$ is an isolated component of $\pathXH{P} \pathXH{Q} \pathXH{R}_1 \pathXH{R}' \pathXH{R}_2$: see Figure~\ref{sfig:bcp3-conn}. It follows from Proposition~\ref{prop:osin-polygons} that $|\pathXH{R}'|_\Omega \leq 5L$, and so $|\pathXH{R}'|_X \leq 5ML$. This contradicts the fact that $d_X(\pathXH{R}'_-,\pathXH{R}'_+) > \mu$.

Thus $\pathXH{R}'$ must be connected to some other component of $\pathXH{PQR}$. As $\pathXH{R}$ does not backtrack, it follows that $\pathXH{R}'$ is connected to a component of either $\pathXH{P}$ or $\pathXH{Q}$, as required.

\item Let $j \in \{ 1,\ldots,m \}$ be such that $\pathXH{P}'$ and $\pathXH{R}'$ are $H_j$-components. Thus there exist $H_j$-edges $\pathXH{e}_1$ and $\pathXH{e}_2$ such that $(\pathXH{e}_1)_- = \pathXH{P}'_-$, $(\pathXH{e}_1)_+ = \pathXH{R}'_+$, $(\pathXH{e}_2)_- = \pathXH{R}'_-$ and $(\pathXH{e}_2)_+ = \pathXH{P}'_+$; moreover, we may write $\pathXH{P} = \pathXH{P}_1 \pathXH{P}' \pathXH{P}_2$ and $\pathXH{R} = \pathXH{R}_1 \pathXH{R}' \pathXH{R}_2$: see Figure~\ref{sfig:bcp3-sim2}. Since $\pathXH{P}$ and $\pathXH{R}$ do not backtrack, and since $\pathXH{R}'$ (and therefore $\pathXH{e}_2$) is not connected to any component of $\pathXH{Q}$, it follows that $\pathXH{e}_1$ and $\pathXH{e}_2$ are isolated components of $\pathXH{R}_2 \pathXH{P}_1 \pathXH{e}_1$ and $\pathXH{P}_2 \pathXH{Q} \pathXH{R}_1 \pathXH{e}_2$, respectively. Moreover, $\pathXH{R}_2$, $\pathXH{P}_1$, $\pathXH{P}_2$, $\pathXH{Q}$ and $\pathXH{R}_1$ are all $(\lambda,c)$-quasi-geodesics.

It follows from Proposition~\ref{prop:osin-polygons} that $|\pathXH{e}_1|_\Omega \leq 3L$ and $|\pathXH{e}_2|_\Omega \leq 4L$; therefore, we have $|\pathXH{e}_1|_X \leq 3ML \leq \mu$ and $|\pathXH{e}_2|_X \leq 4ML \leq \mu$. Hence $d_X(\pathXH{P}'_-,\pathXH{R}'_+),d_X(\pathXH{R}'_-,\pathXH{P}'_+) \leq \mu$, as required.

\item Let $j \in \{ 1,\ldots,m \}$ be such that $\pathXH{P}'$, $\pathXH{Q}'$ and $\pathXH{R}'$ are $H_j$-components. Thus there exist $H_j$-edges $\pathXH{e}_1$, $\pathXH{e}_2$ and $\pathXH{e}_3$ such that $(\pathXH{e}_1)_- = \pathXH{P}'_-$, $(\pathXH{e}_1)_+ = \pathXH{R}'_+$, $(\pathXH{e}_2)_- = \pathXH{Q}'_-$, $(\pathXH{e}_2)_+ = \pathXH{P}'_+$, $(\pathXH{e}_3)_- = \pathXH{R}'_-$, $(\pathXH{e}_3)_+ = \pathXH{Q}'_+$; moreover, we may write $\pathXH{P} = \pathXH{P}_1 \pathXH{P}' \pathXH{P}_2$, $\pathXH{Q} = \pathXH{Q}_1 \pathXH{Q}' \pathXH{Q}_2$ and $\pathXH{R} = \pathXH{R}_1 \pathXH{R}' \pathXH{R}_2$: see Figure~\ref{sfig:bcp3-sim3}. Since $\pathXH{P}$, $\pathXH{Q}$ and $\pathXH{R}$ do not backtrack, it follows that $\pathXH{e}_1$, $\pathXH{e}_2$ and $\pathXH{e}_3$ are isolated components of $\pathXH{R}_2 \pathXH{P}_1 \pathXH{e}_1$, $\pathXH{P}_2 \pathXH{Q}_1 \pathXH{e}_2$ and $\pathXH{Q}_2 \pathXH{R}_1 \pathXH{e}_3$, respectively. Moreover, $\pathXH{R}_2$, $\pathXH{P}_1$, $\pathXH{P}_2$, $\pathXH{Q}_1$, $\pathXH{Q}_2$ and $\pathXH{R}_1$ are all $(\lambda,c)$-quasi-geodesics.

It follows from Proposition~\ref{prop:osin-polygons} that $|\pathXH{e}_i|_\Omega \leq 3L$, and so $|\pathXH{e}_i|_X \leq 3ML \leq \mu$, for each $i \in \{1,2,3\}$. Hence $d_X(\pathXH{P}'_-,\pathXH{R}'_+)$, $d_X(\pathXH{Q}'_-,\pathXH{P}'_+)$ and $d_X(\pathXH{R}'_-,\pathXH{Q}'_+)$ are all bounded above by $\mu$, as required. \qedhere

\end{enumerate}
\end{proof}

\subsection{The graph \texorpdfstring{$\Gamma(N)$}{Gamma(N)}} \label{ssec:GammaN}

We now construct an action of $G$ on a graph $\Gamma(N)$ given an action of $\widetilde{H}_j \leq G$ on a graph $\Gamma_j$ for each $j \in \{1,\ldots,m\}$, so that if the actions $\widetilde{H}_j \acts \Gamma_j$ are all geometric then so is $G \acts \Gamma(N)$. Roughly speaking, we take the barycentric subdivision of the Cayley graph $\Cay(G,X)$ together with a copy of a graph $\Gamma_{j,N}$ (obtained by adding extra edges to $\Gamma_j$) for each right coset of $H_j$ in $G$, and glue them together using `connecting edges' in a consistent way.

Thus, let $G$ be a group with a finite symmetric generating set $X$, let $H_1,\ldots,H_m \leq G$ be a collection of subgroups, and, for each $j$, let $\widetilde{H}_j$ be an isomorphic copy of $H_j$ acting on a simplicial graph $\Gamma_j$ by isometries, and fix a vertex $v_j \in \Gamma_j$. Let $F = F(X) \ast (\ast_{j=1}^m \widetilde{H}_j)$, let $\mathcal{H} = \bigsqcup_{j=1}^m (\widetilde{H}_j \setminus \{1\})$, let $\epsilon: F \to G$ be the canonical surjection, and let $\iota: X \to X$ be an involution such that $\epsilon(\iota(x)) = \epsilon(x)^{-1}$ for all $x \in X$. For each $j$, let $\pi_j: F \to \widetilde{H}_j$ be the canonical retraction, defined as the identity map on $\widetilde{H}_j$ and as the trivial map on $F(X)$ and on $\widetilde{H}_i$ for $i \neq j$.

We now construct a graph $\widetilde\Gamma(N)$ by taking a copy of the barycentric subdivision of the Cayley graph $\Cay(F(X),X)$ for each right coset of $F(X)$ in $F$ and a copy of $\Gamma_{j,N}$ for each right coset of $\widetilde{H}_j$ in $F$, and connecting them using auxiliary edges.

\begin{defn}
Let $N \geq 1$. We construct a simplicial graph $\widetilde\Gamma(N)$ as follows.
\begin{enumerate}[label=(\roman*)]
\item For each $j \in \{1,\ldots,m\}$, define a graph $\Gamma_{j,N}$ with vertices $V(\Gamma_{j,N})=V(\Gamma_j)$ and edges $\{v,w\}$ whenever $d_{\Gamma_j}(v,w) \leq N$.
\item The vertices of $\widetilde\Gamma(N)$ are $V(\widetilde\Gamma(N)) = \Vtf \sqcup \Vtm \sqcup \Vti$, where
\begin{enumerate}
\item $\Vtf = F$, the \emph{free vertices};
\item $\Vtm = (F \times X)/\sim$, where $(g,x) \sim (h,y)$ if and only if $(h,y) \in \{(g,x),(xg,\iota(x))\}$, the \emph{medial vertices}; and
\item $\Vti = \bigsqcup_{j=1}^m (\widetilde{H}_j \backslash F) \times V(\Gamma_j)$, where $\widetilde{H}_j \backslash F$ is the set of right cosets of $\widetilde{H}_j$ in $F$, the \emph{internal vertices}.
\end{enumerate}
\item The edges of $\widetilde\Gamma(N)$ are of three types:
\begin{enumerate}
\item the \emph{free edges}: $\left\{ g, [(g,x)] \right\}$ for vertices $g \in \Vtf$ and $[(g,x)] \in \Vtm$;
\item the \emph{connecting edges}: $\left\{ g, (\widetilde{H}_jg,v_j \cdot \pi_j(g)) \right\}$ for $g \in \Vtf$ and $(\widetilde{H}_jg,v_j \cdot \pi_j(g)) \in \Vti$; and
\item the \emph{internal edges}: $\left\{ (\widetilde{H}_jg,v), (\widetilde{H}_jg,w) \right\}$ for vertices $(\widetilde{H}_jg,v),(\widetilde{H}_jg,w) \in \Vti$ such that $\{v,w\} \in E(\Gamma_{j,N})$.
\end{enumerate}
\end{enumerate}
\end{defn}

We define a right action of $F$ on $V(\widetilde\Gamma(N))$ as follows: for $g \in F$,
\begin{align*}
h \cdot g &= hg && \text{for } h \in \Vf; \\
[(h,x)] \cdot g &= [(hg,x)] && \text{for } [(h,x)] \in \Vm; \text{ and } \\
(\widetilde{H}_jh,v) \cdot g &= (\widetilde{H}_jhg,v \cdot \pi_j(g)) && \text{for } (\widetilde{H}_jh,v) \in \Vi.
\end{align*}
It is easy to see that this is indeed a well-defined action, and that it sends edges of $\widetilde\Gamma(N)$ to edges, thus inducing an action of $F$ on $\widetilde\Gamma(N)$. It is also clear that this action preserves the types of vertices (free, medial or internal) and edges (free, connecting or internal): see Figure~\ref{fig:Gamma-mod-F}.

\begin{figure}[ht]
\begin{tikzpicture}[very thick,x=2cm,y=2cm]
\filldraw [red,dashed,fill=red!20,rounded corners=10] (-1.2,1) rectangle (-0.3,0.4);
\filldraw [red,dashed,fill=red!20,rounded corners=10] (1.2,1) rectangle (0.3,0.4);
\draw [blue] (0,0) -- ({{-0.8*cos(40)}},{{0.8*sin(40)}});
\fill [red] ({{-0.8*cos(40)}},{{0.8*sin(40)}}) circle (3pt);
\draw [blue] (0,0) -- ({{0.8*cos(40)}},{{0.8*sin(40)}});
\fill [red] ({{0.8*cos(40)}},{{0.8*sin(40)}}) circle (3pt);
\node [red] at (-0.75,0.75) {$\Gamma_{1,N}/\widetilde{H}_1$};
\node [red] at (0.75,0.75) {$\Gamma_{m,N}/\widetilde{H}_m$};
\draw [black!70] (0,0) to[out=-160,in=120] ({{-cos(30)}},{{-sin(30)}}) to[out=-60,in=-140] (0,0);
\fill [green!70!black] ({{-cos(30)}},{{-sin(30)}}) circle (3pt);
\draw [black!70] (0,0) to[out=-120,in=160] ({{-cos(70)}},{{-sin(70)}}) to[out=-20,in=-100] (0,0);
\fill [green!70!black] ({{-cos(70)}},{{-sin(70)}}) circle (3pt);
\draw [black!70] (0,0) to[out=-20,in=60] ({{cos(30)}},{{-sin(30)}}) to[out=-120,in=-40] (0,0);
\fill [green!70!black] ({{cos(30)}},{{-sin(30)}}) circle (3pt);
\fill [black!70] (0,0) circle (3pt);
\draw [loosely dotted] (0,-0.8) arc (-90:-50:0.8);
\draw [loosely dotted] ({{-0.8*cos(75)}},{{0.8*sin(75)}}) arc (105:75:0.8);
\end{tikzpicture}
\caption{The quotient graph $\widetilde\Gamma(N)/F$. Orbits of {\color{green!70!black}medial} vertices are shown in green, {\color{blue}connecting} edges in blue, {\color{black!70}free} vertices and edges in gray. Orbits of {\color{red}internal} vertices and edges are represented by the red regions.}
\label{fig:Gamma-mod-F}
\end{figure}
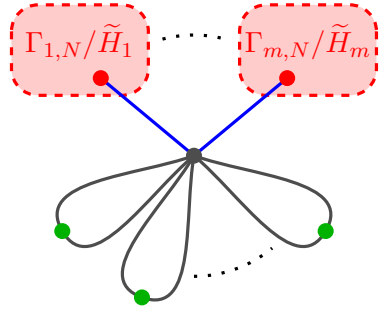

Finally, we define the graph $\Gamma(N)$ as the quotient $\Gamma(N) = \widetilde\Gamma(N)/\ker(\epsilon)$, so that we have an action $G \acts \Gamma(N)$ induced by $F \acts \widetilde\Gamma(N)$. The following result is straightforward.

\begin{lem}[Description of $\Gamma(N)$]
We have $V(\Gamma(N)) = \Vf \sqcup \Vm \sqcup \Vi$, where $\Vf = G$ are the free vertices, $\Vm = (G \times X)/\sim$ are the medial vertices, where $(g,x) \sim (h,y)$ if and only if $(h,y) \in \{(g,x),(\epsilon(x)g,\iota(x))\}$, and $\Vi = \bigsqcup_{j=1}^m (H_j \backslash G) \times V(\Gamma_j)$ are the internal vertices. The edges of $\Gamma(N)$ can be partitioned into
\[
\begin{aligned}
\text{free: } & \left\{ g, [(g,x)] \right\} && \text{for } g \in \Vf \text{ and } [(g,x)] \in \Vm, \\
\text{connecting: } & \left\{ g, (H_jg,u_{j,g}) \right\} && \text{for } g \in \Vf , (H_jg,u_{j,g}) \in \Vi \text{ and some } u_{j,g} \in V(\Gamma_j), \\
\text{internal: } & \left\{ (H_jg,v), (H_jg,w) \right\} && \text{for } (H_jg,v), (H_jg,w) \in \Vi \text{ with } \{v,w\} \in E(\Gamma_{j,N}). &\qed
\end{aligned}
\]
\end{lem}

In particular, it follows that for any $j \in \{1,\ldots,m\}$ and any $H_jg \in H_j \backslash G$, the subset $\{ (H_jg,v) \mid v \in V(\Gamma_j) \} \subset V(\Gamma(N))$ spans a subgraph isomorphic to $\Gamma_{j,N}$. We will refer to this subgraph as the \emph{$g$-copy} (or just a \emph{copy}) of $\Gamma_{j,N}$, and we say that a path $P$ \emph{penetrates} a copy $\Gamma_0$ of $\Gamma_{j,N}$ if $P \cap \Gamma_0 \neq \varnothing$.

Given a path $P \subseteq \Gamma(N)$, we also say that $P$ has \emph{no parabolic shortenings} if every subpath of $P$ all of whose vertices are internal -- that is, a subpath contained in some copy $\Gamma_0 \subseteq \Gamma(N)$ of $\Gamma_{j,N}$ -- is a geodesic when viewed as a path in $\Gamma_0$. (This terminology is taken from \cite{antolin-ciobanu}, which in turn arises from the notion of parabolic subgroups: when $G$ is hyperbolic relative to $\{ H_1,\ldots,H_m \}$, a subgroup of $G$ is said to be \emph{parabolic} if it is conjugate to some $H_j$.)

\begin{lem}[Properties of $G \acts \Gamma(N)$] \label{lem:geomaction}
Suppose that, for each $j \in \{1,\ldots,m\}$, the graph $\Gamma_j$ is proper (as a metric space). Then the graphs $\widetilde\Gamma(N)$ and $\Gamma(N)$ are proper, and the following hold.
\begin{enumerate}[label=\textup{(\roman*)}]
\item \label{it:action-proper} If each action $\widetilde{H}_j \acts \Gamma_j$ is properly discontinuous, then so are the actions $F \acts \widetilde\Gamma(N)$ and $G \acts \Gamma(N)$.
\item \label{it:action-cocompact} If each action $\widetilde{H}_j \acts \Gamma_j$ is cocompact, then so are the actions $F \acts \widetilde\Gamma(N)$ and $G \acts \Gamma(N)$.
\end{enumerate}
\end{lem}

\begin{proof}
To show that $\widetilde\Gamma(N)$ is proper, it is enough to show that each vertex of $\widetilde\Gamma(N)$ is incident to finitely many edges. But a free vertex of $\widetilde\Gamma(N)$ is incident to $|X|+m < \infty$ edges, a medial vertex is incident to two edges, and an $\widetilde{H}_j$-internal vertex $(\widetilde{H}_jg,u)$ is incident to at most $d_{j,N}(u) + |\Stab_{\widetilde{H}_j}(u)| < \infty$ edges, where $d_{j,N}(u)$ is the number of vertices $v \in \Gamma_j$ such that $d_{\Gamma_j}(u,v) \leq N$. Thus $\widetilde\Gamma(N)$ is proper as a metric space. As $\Gamma(N)$ is a quotient of $\widetilde\Gamma(N)$, it follows that $\Gamma(N)$ is proper as well.

We now prove \ref{it:action-proper} and \ref{it:action-cocompact}.

\begin{enumerate}[label=(\roman*)]

\item Since $\widetilde{H}_j \acts \Gamma_j$ is properly discontinuous, each vertex of $\Gamma_j$ has a finite $\widetilde{H}_j$-stabiliser. Now each free vertex and each medial vertex of $\widetilde\Gamma(N)$ has a trivial stabiliser. Moreover, it is easy to check that $\Stab_F(\widetilde{H}_jg,u) = g^{-1} \left[ \Stab_{\widetilde{H}_j}\left(u \cdot \pi_j(g)^{-1}\right) \right] g$ for any $\widetilde{H}_j$-internal vertex $(\widetilde{H}_jg,u)$, and so $|\Stab_F(\widetilde{H}_jg,u)| < \infty$. Thus vertices of $\widetilde\Gamma(N)$ have finite $F$-stabilisers and so the action of $F$ on $\widetilde\Gamma(N)$ is properly discontinuous. Since $G$-stabilisers of vertices in $\Gamma(N)$ are just images of $F$-stabilisers of vertices in $\widetilde\Gamma(N)$ under $\epsilon$, it follows that they are finite, and so $G \acts \Gamma(N)$ is properly discontinuous as well.

\item Since $\Gamma_j$ is proper and the action of $\widetilde{H}_j$ on $\Gamma_j$ is cocompact (for each $j$), the action of $\widetilde{H}_j$ on $\Gamma_{j,N}$ is cocompact as well. It is then easy to check that $\widetilde\Gamma(N)$ has $m < \infty$ orbits of connecting edges, $2 \cdot |X/{\sim}| < \infty$ orbits of free edges, and $\sum_{j=1}^m \left|E(\Gamma_{j,N}/\widetilde{H}_j)\right| < \infty$ orbits of internal edges (see also Figure~\ref{fig:Gamma-mod-F}). Thus the action of $F$ on $\widetilde\Gamma(N)$ is cocompact, and since $\widetilde\Gamma(N)/F \cong \Gamma(N)/G$ so is the action of $G$ on $\Gamma(N)$. \qedhere

\end{enumerate}
\end{proof}

\subsection{Derived paths} \label{ssec:derived}

Here we define a derived path $\widehat{P}$ in $\CayG$ corresponding to a given path $P$ in $\Gamma(N)$. Our definition is an adaptation of a construction of Y.~Antol\'in and L.~Ciobanu \cite[Construction 4.1]{antolin-ciobanu}.

\begin{defn} \label{defn:derived}
Let $N \geq 1$. For each vertex $v \in V(\Gamma(N))$, we define a path $Z_v \subseteq \Gamma(N)$ such that $(Z_v)_- = v$ and $(Z_v)_+$ is a free vertex, and such that $|Z_v|$ is as small as possible under these conditions. Given a path $P \subseteq \Gamma(N)$, we define the \emph{derived path} $\widehat{P} \subseteq \CayG$ in the following steps.
\begin{enumerate}[label=(\roman*)]

\item \label{it:derived-1} If $P$ has no free vertices, then set $n = 1$ and $P_1 = \overline{Z_{P_-}} P Z_{P_+}$, and proceed to step \ref{it:derived-3}. Otherwise, write $P = P' P_2 \cdots P_{n-1} P''$ in such a way that $P'_+ = (P_2)_-$, $(P_2)_- = (P_3)_+$, \dots, $(P_{n-1})_+ = P''_-$ are all free, no other vertices of $P$ are free, and $|P_i| \geq 1$ for $2 \leq i \leq n-1$.

\item \label{it:derived-2} If $|P'| \leq 3$, then set $P_1$ to be the trivial path with $(P_1)_- = (P_1)_+ = P'_+$; otherwise, set $P_1 = \overline{Z_{P_-}} P'$. Similarly, if $|P''| \leq 3$, then set $P_n$ to be the trivial path with $(P_n)_- = (P_n)_+ = P''_-$; otherwise, set $P_n = P'' Z_{P_+}$. If $|P'| \geq 4$ (respectively, $|P''| \geq 4$), then we call $P_1$ (respectively, $P_n$) an \emph{extended subpath} of $P$.

It now follows from the construction that for $1 \leq i \leq n$, the vertices $(P_i)_-$ and $(P_i)_+$ are the only free vertices of $P_i$.

\item \label{it:derived-3} Let $\widehat{P} = \widehat{P}_1 \widehat{P}_2 \cdots \widehat{P}_n$ be the path in $\CayG$ such that, for $1 \leq i \leq n$, $\widehat{P}_i$ is a path of length $\leq 1$ with $(\widehat{P}_i)_- = (P_i)_-$ and $(\widehat{P}_i)_+ = (P_i)_+$, as follows.
\begin{enumerate}[label=(\alph*)]
\item If $(P_i)_- = (P_i)_+$, then we set $\widehat{P}_i$ to be an empty path.
\item If $(P_i)_- \neq (P_i)_+$ and all non-endpoint vertices of $P_i$ are medial, then $|P_i| = 2$ and $(P_i)_+ = \epsilon(x) (P_i)_-$ for some $x \in X$; we then set $\widehat{P}_i$ to be an edge labelled by $x$.
\item Otherwise, all non-endpoint vertices of $P_i$ are $H_j$-internal for some $j$, and $(P_i)_+ = \epsilon(h) (P_i)_-$ for some $h \in \widetilde{H}_j$; we then set $\widehat{P}_i$ to be an $\widetilde{H}_j$-edge labelled by $h$.
\end{enumerate}

\end{enumerate}
\end{defn}

An example construction of $\widehat{P}$ is shown in Figure~\ref{fig:derived}.

\begin{figure}[ht]
\begin{tikzpicture}[thick,scale=1.25]
\draw [thin,decorate,decoration={brace,amplitude=8pt}] (0,0.7) -- (7.2,0.7) node [midway,above,yshift=6pt] {$P$};
\draw [blue,thin,decorate,decoration={brace,amplitude=4pt}] (-0.2,-0.08) -- (-0.2,0.5) node [midway,left,xshift=-2pt] {$\overline{Z_{P_-}}$};
\draw [blue] (0,-0.08) to[out=120,in=180] (0,0.46) node[below right] {$P_1$} to[out=0,in=117] (1,-0.08);
\draw [blue] (1,-0.08) to[out=63,in=180] (1.75,0.335) node[below] {$P_2$} to[out=0,in=117] (2.5,-0.08);
\draw [blue] (2.5,-0.08) to[out=63,in=180] (3,0.26);
\draw [blue] (4.5,0.26) to[out=0,in=117] (5,-0.08);
\draw [blue] (5,-0.08) to[out=63,in=180] (5.75,0.335) node[below] {$P_{n-1}$} to[out=0,in=117] (6.5,-0.08);
\draw (0,0.5) to[out=0,in=120] node[near end,above] {$\ P'$} (1,0) to[out=60,in=120] (2.5,0) to[out=60,in=180] (3,0.3);
\draw [loosely dotted] (3.2,0.2) -- (4.3,0.2);
\draw (4.5,0.3) to[out=0,in=120] (5,0) to[out=60,in=120] (6.5,0) to[out=60,in=180] node[near end,left] {$P''$} (7.2,0.4);
\foreach \x in {0,1,2.5,5,6.5} {
    \fill (\x,-0.06) circle (1.5pt);
    \fill [red] (\x,-0.8) circle (1.5pt);
}
\draw [->,decorate,decoration={snake,amplitude=1.5pt}] (3.75,0.05) -- (3.75,-0.65);
\draw [red] (0,-0.8) -- (3,-0.8) (4.5,-0.8) -- (6.5,-0.8);
\draw [red,loosely dotted] (3.2,-0.8) -- (4.3,-0.8);
\draw [red] (0.5,-0.8) node [below] {$\widehat{P}_1$} (1.75,-0.8) node [below] {$\widehat{P}_2$} (5.75,-0.8) node [below] {$\widehat{P}_{n-1}$};
\end{tikzpicture}
\caption{An example construction of a derived path $\widehat{P} \subseteq \CayG$ given a path $P \subseteq \Gamma(N)$. In this case, $|P'| \geq 4$, $|P''| \leq 3$, and the paths $P_n$ and $\widehat{P}_n$ are trivial.}
\label{fig:derived}
\end{figure}
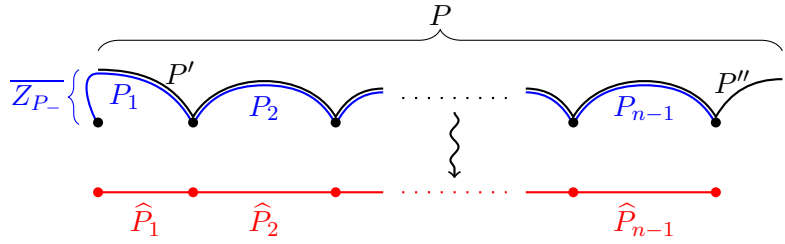

\section{Mapping geodesics to quasi-geodesics} \label{sec:g->qg}

In this section we prove Theorem \ref{thm:g->qg}. In order to do this, we proceed in two steps. We first show that the word labelling a $2$-local geodesic path in $\CayG$, that does not contain any of the finitely many `strongly non-geodesic' prohibited subwords, labels a $(\lambda,c)$-quasi-geodesic that does not backtrack (for some fixed $\lambda \geq 1$ and $c \geq 0$): see Proposition~\ref{prop:lg=qg}. We then show that for $N$ big enough, if $P$ is a $5$-local geodesic in $\Gamma(N)$ with no parabolic shortenings and with no subpaths $Q$ such that $\widehat{Q}$ is labelled by one of the aforementioned prohibited words, then $\widehat{P}$ satisfies the premise of Proposition~\ref{prop:lg=qg}.

\subsection{Local geodesics are quasi-geodesics} \label{ssec:lg=qg}

We start by analysing $2$-local geodesics in the graph $\CayG$. The following result is a strengthening of a result of Y.~Antol\'in and L.~Ciobanu \cite[Theorem~5.2]{antolin-ciobanu}, and the proof given in \cite{antolin-ciobanu} carries through to prove the following Proposition after several minor modifications. Here, we say a word $\wordXH{P}$ over $\XH$ \emph{vertex backtracks} if $\wordXH{P}$ contains a subword $\wordXH{Q}$ representing an element of $H_j$ (for some $j$) with $|\wordXH{Q}| \geq 2$. Clearly, a word that does not vertex backtrack does not backtrack either.

\begin{prop} \label{prop:lg=qg}
There exist constants $\lambda \geq 1$, $c \geq 0$, and a finite collection $\widehat\Phi$ of words over $\XH$ labelling paths $\pathXH{Q} \subseteq \CayG$ with $|\pathXH{Q}| > \frac{3}{2}d_{\XH}(\pathXH{Q}_-,\pathXH{Q}_+)$ such that the following holds. Let $\wordXH{P}$ be a $2$-local geodesic word in $\CayG$ that does not contain any element of $\widehat\Phi$ as a subword. Then $\wordXH{P}$ is a $(\lambda,c)$-quasi-geodesic word that does not vertex backtrack.
\end{prop}

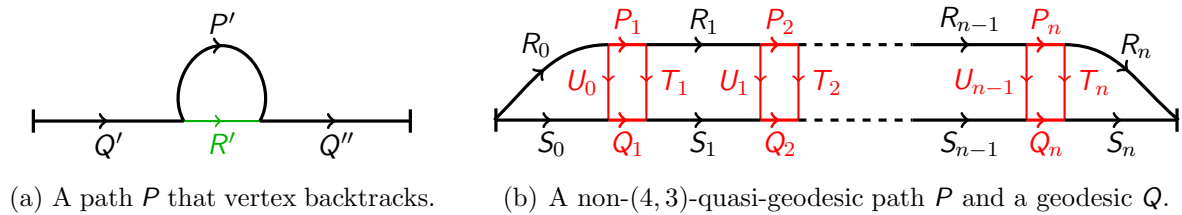
\begin{figure}[ht]
\begin{subfigure}[t]{0.35\textwidth}
\centering
\begin{tikzpicture}[very thick,decoration={markings,mark=at position 0.5 with {\arrow{>}}}]
\draw [green!70!black,thick,postaction=decorate] (2,0) to node[midway,below] {$\pathXH{R}'$} (3,0);
\draw [|-,postaction=decorate] (0,0) to node[midway,below] {$\pathXH{Q}'$} (2,0);
\draw [postaction=decorate] (2,0) to[out=120,in=180] (2.5,1) node [above] {$\pathXH{P}'$} to[out=0,in=60] (3,0);
\draw [-|,postaction=decorate] (3,0) to node[midway,below] {$\pathXH{Q}''$} (5,0);
\end{tikzpicture}
\caption{A path $\pathXH{P}$ that vertex backtracks.}
\label{sfig:lgqg-backtrack}
\end{subfigure}\hfill%
\begin{subfigure}[t]{0.64\textwidth}
\centering
\begin{tikzpicture}[very thick,decoration={markings,mark=at position 0.5 with {\arrow{>}}}]
\draw [red,postaction=decorate] (1.5,0) -- (2,0) node [midway,below] {$\pathXH{Q}_1$};
\draw [red,postaction=decorate] (3.5,0) -- (4,0) node [midway,below] {$\pathXH{Q}_2$};
\draw [red,postaction=decorate] (7,0) -- (7.5,0) node [midway,below] {$\pathXH{Q}_n$};
\draw [red,postaction=decorate] (1.5,1) -- (2,1) node [midway,above] {$\pathXH{P}_1$};
\draw [red,postaction=decorate] (3.5,1) -- (4,1) node [midway,above] {$\pathXH{P}_2$};
\draw [red,postaction=decorate] (7,1) -- (7.5,1) node [midway,above] {$\pathXH{P}_n$};
\draw [red,thick,postaction=decorate] (1.5,1) -- (1.5,0) node [midway,left] {$\pathXH{U}_0$};
\draw [red,thick,postaction=decorate] (2,1) -- (2,0) node [midway,right] {$\pathXH{T}_1$};
\draw [red,thick,postaction=decorate] (3.5,1) -- (3.5,0) node [midway,left] {$\pathXH{U}_1$};
\draw [red,thick,postaction=decorate] (4,1) -- (4,0) node [midway,right] {$\pathXH{T}_2$};
\draw [red,thick,postaction=decorate] (7,1) -- (7,0) node [midway,left] {$\pathXH{U}_{n-1}$};
\draw [red,thick,postaction=decorate] (7.5,1) -- (7.5,0) node [midway,right] {$\pathXH{T}_n$};
\draw [|-,postaction=decorate] (0,0) -- (1.5,0) node [midway,below] {$\pathXH{S}_0$};
\draw [postaction=decorate] (2,0) -- (3.5,0) node [midway,below] {$\pathXH{S}_1$};
\draw [dashed] (4,0) -- (5.5,0);
\draw [postaction=decorate] (5.5,0) -- (7,0) node [midway,below] {$\pathXH{S}_{n-1}$};
\draw [-|,postaction=decorate] (7.5,0) -- (9,0) node [midway,below] {$\pathXH{S}_n$};
\draw [postaction=decorate] (0,0) to[in=180] node [midway,above] {$\pathXH{R}_0\ \ $} (1.5,1);
\draw [postaction=decorate] (2,1) to node [midway,above] {$\pathXH{R}_1$} (3.5,1);
\draw [dashed] (4,1) -- (5.5,1);
\draw [postaction=decorate] (5.5,1) to node [midway,above] {$\pathXH{R}_{n-1}$} (7,1);
\draw [postaction=decorate] (7.5,1) to[out=0,in=140] node [midway,above] {$\ \ \pathXH{R}_n$} (9,0);
\end{tikzpicture}
\caption{A non-$(4,3)$-quasi-geodesic path $\pathXH{P}$ and a geodesic $\pathXH{Q}$.}
\label{sfig:lgqg-43qg}
\end{subfigure}
\caption{The proof of Proposition \ref{prop:lg=qg}.}
\label{fig:lgqg}
\end{figure}

\begin{proof}
We use the following `local to global' property of quasi-geodesics for a hyperbolic spaces \cite[Chapter 3, Theorem 1.4]{cdp}: given a hyperbolic metric space $Y$ and constants $\lambda' \geq 1$ and $c' \geq 0$, there exist constants $\lambda \geq 1$ and $c,k \geq 0$ such that every $k$-local $(\lambda',c')$-quasi-geodesic in $Y$ is a $(\lambda,c)$-quasi-geodesic. In particular, since $\CayG$ is hyperbolic (see Proposition~\ref{prop:BCPtriangles}\ref{it:bcpt-phase}), we may use this property for $\lambda' = 4$ and $c' = 3$: there exist constants $\lambda \geq 1$ and $c,k \geq 0$ such that every $k$-local $(4,3)$-quasi-geodesic in $\CayG$ is a $(\lambda,c)$-quasi-geodesic. We may further increase $k$ if necessary to assume that $k \geq \lambda+c$.

Now let $\widehat\Psi$ be the set of cyclic paths $\pathXH{Q}$ in $\CayG$ that do not backtrack with $|\pathXH{Q}| < \frac{5}{3}k$. By Proposition~\ref{prop:osin-polygons}, there exists a finite subset $\Omega \subseteq G$ and a constant $L \geq 0$ such that for any $\pathXH{Q} \in \widehat\Psi$ we have
\[
\sum_{i=1}^{n(\pathXH{Q})} \left| \pathXH{P}_{\pathXH{Q},i} \right|_\Omega < \frac{5}{3}kL,
\]
where $\pathXH{P}_{\pathXH{Q},1},\ldots,\pathXH{P}_{\pathXH{Q},n(\pathXH{Q})}$ is the set of words labelling the components of $\pathXH{Q}$. In particular, as $X$ and $\Omega$ are finite, it follows that the set of cyclic words labelling elements of $\widehat\Psi$ is finite.

We then define $\widehat\Phi$ to be the set of words labelling a path $\pathXH{Q}$ in $\CayG$ such that $\pathXH{QP} \in \widehat\Psi$ (in particular, $\pathXH{Q}_- = \pathXH{P}_+$ and $\pathXH{Q}_+ = \pathXH{P}_-$) and $|\pathXH{Q}| > \frac{3}{2}|\pathXH{P}|$ for some path $\pathXH{P}$ in $\CayG$. It is clear that $\widehat\Phi$ is finite, and that every word in $\widehat\Phi$ labels a path $\pathXH{Q}$ with $|\pathXH{Q}| > \frac{3}{2}d_{\XH}(\pathXH{Q}_-,\pathXH{Q}_+)$.

Now let $\wordXH{P}$ be a $2$-local geodesic word over $\XH$ that does not contain any element of $\widehat\Phi$ as a subword. We claim that $\wordXH{P}$ is a $k$-local $(4,3)$-quasi-geodesic that does not vertex backtrack. This will imply the result. To show that $\wordXH{P}$ is a $k$-local $(4,3)$-quasi-geodesic, we may assume, without loss of generality, that $|\wordXH{P}| \leq k$. Let $\pathXH{P} \subseteq \CayG$ be a path labelled by $\wordXH{P}$.

We first claim that $\wordXH{P}$ does not vertex backtrack. Indeed, if it did, then we would have $\pathXH{P} = \pathXH{Q}' \pathXH{P}' \pathXH{Q}''$ where $\pathXH{P}'$ labels an element of $H_j$ (for some $j$) and $|\pathXH{P}'| \geq 2$. Without loss of generality, suppose that $\pathXH{P}'$ is a minimal subpath with this property, and let $\pathXH{R}'$ be a geodesic path with $\pathXH{R}'_\pm = \pathXH{P}'_\pm$ (so that either $|\pathXH{R}'| = 1$ and $\pathXH{R}'$ is an $\widetilde{H}_j$-edge, or $|\pathXH{R}'| = 0$): see Figure~\ref{sfig:lgqg-backtrack}. As $\pathXH{P}$ is a $2$-local geodesic, we have $|\pathXH{P}'| \geq 3$; by minimality of $\pathXH{P}'$, no two components of $\pathXH{P}'$ are connected, and no component of $\pathXH{P}'$ is connected to the $H_j$-component $\pathXH{R}'$ (if $|\pathXH{R}'| = 1$). As $3 \leq |\pathXH{P}'| \leq |\pathXH{P}| \leq k$ and so $|\pathXH{R}'| \leq 1 < 2 \leq \frac{2}{3}k$, it follows that $|\pathXH{P}'\overline{\pathXH{R}'}| < \frac{5}{3}k$ and so $\pathXH{P}'\overline{\pathXH{R}'} \in \widehat\Psi$. But as $|\overline{\pathXH{R}'}| \leq 1$, we have $|\pathXH{P}'| \geq 3 > \frac{3}{2} \geq \frac{3}{2}|\overline{\pathXH{R}'}|$ and so $\pathXH{P}'$ is labelled by a word in $\widehat\Phi$, contradicting the choice of $\wordXH{P}$.

Thus, $\pathXH{P}$ is a $2$-local geodesic path of length $|\pathXH{P}| \leq k$ that does not vertex backtrack. Suppose for contradiction that $\pathXH{P}$ is not a $(4,3)$-quasi-geodesic. By passing to a subpath of $\pathXH{P}$ if necessary, we may assume that there exists a geodesic path $\pathXH{Q}$ in $\CayG$ such that $\pathXH{Q}_- = \pathXH{P}_-$, $\pathXH{Q}_+ = \pathXH{P}_+$, and such that $|\pathXH{P}| > 4|\pathXH{Q}|+3$. As $\pathXH{P}$ does not vertex backtrack and as $\pathXH{Q}$ is a geodesic, both $\pathXH{P}$ and $\pathXH{Q}$ do not backtrack. Also, as $\pathXH{P}$ and $\pathXH{Q}$ are $2$-local geodesics, all components of $\pathXH{P}$ and of $\pathXH{Q}$ are edges.

Now for some $n \geq 0$, we may write $\pathXH{P} = \pathXH{R}_0 \pathXH{P}_1 \pathXH{R}_1 \cdots \pathXH{P}_n \pathXH{R}_n$ and $\pathXH{Q} = \pathXH{S}_0 \pathXH{Q}_1 \pathXH{S}_1 \cdots \pathXH{Q}_n \pathXH{S}_n$ in such a way that, for each $i$, $\pathXH{P}_i$ is a component of $\pathXH{P}$ connected to a component $\pathXH{Q}_i$ of $\pathXH{Q}$, and no component of $\pathXH{R}_i$ is connected to a component of $\pathXH{S}_i$. Let $\pathXH{T}_i$ (respectively, $\pathXH{U}_i$) be the geodesic in $\CayG$ from $(\pathXH{R}_i)_-$ to $(\pathXH{S}_i)_-$ (respectively, from $(\pathXH{R}_i)_+$ to $(\pathXH{S}_i)_+$), so that either $\pathXH{T}_i$ (respectively, $\pathXH{U}_i$) has length $0$, or it is an $\mathcal{H}$-edge that is connected, as a component, to $\pathXH{P}_i$ and $\pathXH{Q}_i$ (respectively, $\pathXH{P}_{i+1}$ and $\pathXH{Q}_{i+1}$); see Figure~\ref{sfig:lgqg-43qg}.

If we had $|\pathXH{R}_i| \leq \frac{3}{2}(|\pathXH{S}_i|+2)$ for each $i$, then we would have
\[
|\pathXH{P}| = n+\sum_{i=0}^n |\pathXH{R}_i| \leq 4n+3+\frac{3}{2}\sum_{i=0}^n |\pathXH{S}_i| \leq 4\left( n+\sum_{i=0}^n |\pathXH{S}_i| \right)+3 = 4|\pathXH{Q}|+3,
\]
contradicting the assumption that $|\pathXH{P}| > 4|\pathXH{Q}|+3$. Thus, there exists some $i \in \{0,\ldots,n\}$, which we fix from now on, such that $|\pathXH{R}_i| > \frac{3}{2}(|\pathXH{S}_i|+2)$.

Now if $|\pathXH{T}_i| = 1$ (as opposed to $|\pathXH{T}_i|=0$), since $\pathXH{T}_i$ is connected to a component $\pathXH{P}_i$ of $\pathXH{P}$ and since $\pathXH{P}$ does not backtrack, it follows that $\pathXH{T}_i$ is not connected to any component of $\pathXH{R}_i$. Similarly, if $|\pathXH{T}_i| = 1$ then $\pathXH{T}_i$ is not connected to any component of $\pathXH{S}_i$, and if $|\pathXH{U}_i| = 1$ then $\pathXH{U}_i$ is not connected to any component of $\pathXH{R}_i$ or of $\pathXH{S}_i$. Since $\pathXH{P}$ and $\pathXH{Q}$ do not backtrack, no two components of $\pathXH{R}_i$ and no two components of $\pathXH{S}_i$ are connected. By construction, no component of $\pathXH{R}_i$ is connected to a component of $\pathXH{S}_i$, and if $\pathXH{T}_i$ was a component connected to a component $\pathXH{U}_i$ then we would have $|\pathXH{R}_i| \leq 1$ (as $\pathXH{P}$ does not vertex backtrack), contradicting the fact that $|\pathXH{R}_i| > \frac{3}{2}(|\pathXH{S}_i|+2) \geq 3$.

It hence follows that all components of the cyclic path $\pathXH{R}_i\pathXH{U}_i\overline{\pathXH{S}_i}\,\overline{\pathXH{T}_i}$ are isolated. Moreover, since $k \geq |\pathXH{P}| \geq |\pathXH{R}_i| > \frac{3}{2}(|\pathXH{S}_i|+2)$ we have $|\pathXH{R}_i\pathXH{U}_i\overline{\pathXH{S}_i}\,\overline{\pathXH{T}_i}| = |\pathXH{R}_i| + (|\pathXH{S}_i| + 2) < \frac{5}{3}k$. Thus $\pathXH{R}_i\pathXH{U}_i\overline{\pathXH{S}_i}\,\overline{\pathXH{T}_i} \in \widehat\Psi$, and as $|\pathXH{R}_i| > \frac{3}{2}(|\pathXH{S}_i|+2) \geq \frac{3}{2}|\pathXH{U}_i\overline{\pathXH{S}_i}\,\overline{\pathXH{T}_i}|$, the path $\pathXH{R}_i$ is labelled by a word in $\widehat\Phi$, contradicting the choice of $\wordXH{P}$.

Thus, if $\wordXH{P}$ is a $2$-local geodesic word in $\CayG$ that does not contain any element of $\widehat\Phi$ as a subword, then $\wordXH{P}$ is a $k$-local $(4,3)$-quasi-geodesic word and any subword of $\wordXH{P}$ of length $\leq k$ does not vertex backtrack. In particular, $\wordXH{P}$ is a $(\lambda,c)$-quasi-geodesic. If $\wordXH{P}$ did vertex backtrack then we would have $\wordXH{P} = \wordXH{Q}'\wordXH{P}'\wordXH{Q}''$ where $\wordXH{P}'$ represents an element of some $H_j$ and $|\wordXH{P}'| > k$ (cf Figure~\ref{sfig:lgqg-backtrack}); but $k$ was chosen so that $k \geq \lambda+c$, so this is impossible since $\wordXH{P}$ is a $(\lambda,c)$-quasi-geodesic. Hence $\wordXH{P}$ is a $(\lambda,c)$-quasi-geodesic word that does not vertex backtrack, as required.
\end{proof}

\subsection{Local geodesics map to local geodesics} \label{ssec:lg->lg}

Throughout the remainder of this section, we adopt the following assumption.

\begin{ass} \label{ass:Nbigger}
We assume that each $\Gamma_j$ is proper (as a metric space), and $\widetilde{H}_j$ acts on $\Gamma_j$ geometrically. We furthermore assume that $N \geq 1$ is chosen in such a way that the following properties hold.
\begin{enumerate}[label=(\roman*)]
\item \label{it:Nbigg-quot} The distance in $\Gamma_j / \widetilde{H}_j$ between $v_j \cdot \widetilde{H}_j$ and any other vertex is at most $N$.
\item \label{it:Nbigg-Phi} For each word $\wordXH{P} \in \widehat\Phi$, where $\widehat\Phi$ is the finite collection of words in Proposition~\ref{prop:lg=qg}, fix a geodesic word $\wordXH{U_P}$ over $\XH$ such that $\wordXH{U_P}$ and $\wordXH{P}$ represent the same element of $G$. Then for any $j$ and any $\widetilde{H}_j$-letter $h$ of $\wordXH{U_P}$, we have $d_{\Gamma_j}(v_j,v_j \cdot h) \leq N$.
\item \label{it:Nbigg-loop3} If $\wordXH{P}$ is a cyclic word over $\XH$ of length $|\wordXH{P}| \leq 3$ with all its components single letters that does not backtrack, and if $h$ is an $\widetilde{H}_j$-letter of $\wordXH{P}$, then $d_{\Gamma_j}(v_j,v_j \cdot h) \leq N$.
\end{enumerate}
\end{ass}

Note that, under the assumption that the actions $\widetilde{H}_j \acts \Gamma_j$ are all geometric, the points \ref{it:Nbigg-quot}--\ref{it:Nbigg-loop3} in Assumption~\ref{ass:Nbigger} will all be satisfied whenever $N \geq 1$ is large enough. Indeed, since $\widetilde{H}_j$ acts on $\Gamma_j$ geometrically (and so cocompactly), the graph $\Gamma_j / \widetilde{H}_j$ is finite, and so \ref{it:Nbigg-quot} is satisfied for $N$ large enough. Furthermore, as the collection $\widehat\Phi$ in Proposition~\ref{prop:lg=qg} is finite, \ref{it:Nbigg-Phi} is true when $N$ is large enough. Finally, it follows from Proposition~\ref{prop:osin-polygons} that there are finitely many cyclic words over $\XH$ of length $\leq 3$ that do not backtrack all of whose components are single letters, and so \ref{it:Nbigg-loop3} is satisfied for large enough $N$. In fact, one can see from the proof of Proposition~\ref{prop:lg=qg} that \ref{it:Nbigg-Phi} implies \ref{it:Nbigg-loop3} -- but both points are stated here for future reference.

In the following two lemmas, let $P$ be a path in $\Gamma(N)$ that has no parabolic shortenings and at least one free vertex, and let $P_1,\ldots,P_n \subseteq \Gamma(N)$ and $\widehat{P} = \widehat{P}_1 \cdots \widehat{P}_n \subseteq \CayG$ be as in Definition~\ref{defn:derived}.

\begin{figure}[ht]
\begin{subfigure}[t]{0.4\textwidth}
\centering
\begin{tikzpicture}[very thick,decoration={markings,mark=at position 0.5 with {\arrow{>}}}]
\filldraw [dashed,thick,red!50,fill=red!10] (1,1.15) ellipse (3 and 0.5);
\draw [blue,postaction=decorate] (0,0) to node[midway,right] {$e$} (0,1);
\draw [blue,postaction=decorate] (0,1) to node[midway,above] {$P_2'$} (2,1);
\draw [blue,postaction=decorate] (2,0) to node[midway,left] {$e'$} (2,1);
\fill (0,0) circle (2pt) node [left] {$g$};
\fill (0,1) circle (2pt) node [left] {$(H_jg,u)$};
\fill (2,0) circle (2pt) node [right] {$\epsilon(h)g$};
\fill (2,1) circle (2pt) node [right] {$(H_jg,w)$};
\end{tikzpicture}
\caption{Lemma~\ref{lem:comp3}: the path $P_2$.}
\label{sfig:lem3}
\end{subfigure}\hfill%
\begin{subfigure}[t]{0.55\textwidth}
\centering
\begin{tikzpicture}[very thick,decoration={markings,mark=at position 0.5 with {\arrow{>}}}]
\filldraw [dashed,thick,red!50,fill=red!10] (0.5,1.2) ellipse (2.5 and 0.75);
\draw [thick,postaction=decorate] (0,1) to node[midway,below] {$T$} (2,1);
\draw [green!60!black,postaction=decorate] (0,0) to node[pos=0.3,right] {$e_1$} (0,1);
\draw [green!60!black,postaction=decorate] (-1,1) to node[midway,below] {$T_1$} (0,1);
\draw [blue,postaction=decorate] (2,0) to node[pos=0.3,right] {$e_2$} (2,1);
\draw [blue,postaction=decorate] (-1,1) to[bend left] node[midway,above] {$T_2$} (2,1);
\fill (0,0) circle (2pt) (0,1) circle (2pt) (2,0) circle (2pt) (2,1) circle (2pt) (-1,1) circle (2pt) node [left] {$P_-$};
\end{tikzpicture}
\caption{Lemma~\ref{lem:comp4}: the paths $Z_{P_-}$ (green) and $P'$ (blue).}
\label{sfig:lem4}
\end{subfigure}
\caption{The proofs of Lemmas \ref{lem:comp3} and \ref{lem:comp4}. The red shaded area represents a copy of $\Gamma_{j,N}$ in $\Gamma(N)$.}
\label{fig:lem34}
\end{figure}
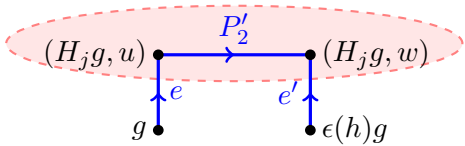
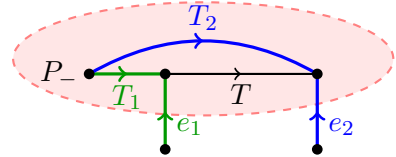

\begin{lem} \label{lem:comp3}
Suppose that $P$ is a closed path, that $P_- = P_+$ is a free vertex, and that $|\widehat{P}| \leq 3$. If, for some $i$, $\widehat{P}_i$ is an isolated component of $\widehat{P}$, then $|P_i| \leq 3$.
\end{lem}

\begin{proof}
Note that since $P_-$ and $P_+$ are free vertices, it follows from the construction in Definition~\ref{defn:derived} that $P = P_2 \cdots P_{n-1}$, and that $P_1$ and $P_n$ are empty paths. Suppose that $\widehat{P}_i$ is an isolated component of $\widehat{P}$: without loss of generality, $i = 2$. Then the non-endpoint vertices of $P_2$ are either all $H_j$-internal (for some $j$) or all medial. In the latter case, we have $|P_2| = 2$ and we are done. Thus, without loss of generality, suppose all non-endpoint vertices of $P_2$ are $H_j$-internal, and so $\widehat{P}_2$ is labelled by an element $h \in \widetilde{H}_j$.

Now it follows from Assumption \ref{ass:Nbigger}\ref{it:Nbigg-loop3} that $d_{\Gamma_j}(v_j,v_j \cdot h) \leq N$. By construction, it also follows that $P_2 = eP_2'\overline{e'}$, where $e$ and $e'$ are the connecting edges with $e_- = g$, $e'_- = \epsilon(h)g$, $e_+ = (H_j,v_j) \cdot g = (H_jg,u)$ and $e'_+ = (H_j,v_j \cdot h) \cdot g = (H_jg,w)$ for some $g \in G$ and some $u,w \in V(\Gamma_j)$ with $d_{\Gamma_j}(u,w) = d_{\Gamma_j}(v_j,v_j \cdot h) \leq N$, and all vertices of $P_2'$ are $H_j$-internal; see Figure~\ref{sfig:lem3}. Since $P$ has no parabolic shortenings, it follows that $P_2'$ is a geodesic as a path in the $g$-copy of $\Gamma_{j,N}$, and so
\[
|P_2'| = d_{\Gamma_{j,N}}(u,w) = \left\lceil \frac{d_{\Gamma_j}(u,w)}{N} \right\rceil \leq 1.
\]
Thus $|P_2| = 2+|P_2'| \leq 3$, as required.
\end{proof}

\begin{lem} \label{lem:comp4}
If, for some $i \in \{1,n\}$, $P_i$ is an extended subpath of $P$, then $\widehat{P}_i$ is an $\widetilde{H}_j$-edge (for some $j$) and, for any path $Q$ in $\Gamma(N)$ with $Q_- = (P_i)_-$, $Q_+ = (P_i)_+$ and all other vertices of $Q$ being $H_j$-internal, we have $|Q| \geq 4$.
\end{lem}

\begin{proof}
Suppose that $i = 1$, and so $P_1$ is an extended subpath of $P$: the case $i = n$ is similar. By the construction described in Definition~\ref{defn:derived}\ref{it:derived-2}, we have $P_1 = \overline{Z_{P_-}} P'$, where $P'$ is an initial subpath of $P$ with $|P'| \geq 4$. Here, $Z_{P_-}$ is a shortest path in $\Gamma(N)$ with $(Z_{P_-})_- = P_-$ and $(Z_{P_-})_+$ a free vertex, and so it follows from Assumption~\ref{ass:Nbigger}\ref{it:Nbigg-quot} that $|Z_{P_-}| \leq 2$.

More precisely, we know that $Z_{P_-} = T_1\overline{e_1}$ and $P' = T_2\overline{e_2}$, where $e_1$ and $e_2$ are connecting edges and all edges on the path $\overline{T_1}T_2$ are $H_j$-internal; see Figure~\ref{sfig:lem4}. Since $|T_2| - |T_1| \geq 3-1=2$ and since $P$ (and so its subpath $T_2$) has no parabolic shortenings, it follows that any path between $(T_1)_+$ and $(T_2)_+$ consisting only of internal edges has length $\geq 2$. But a path $Q$ satisfying the conditions stated must be of the form $Q = e_1T\,\overline{e_2}$, where $T$ is a path between $(T_1)_+$ and $(T_2)_+$ consisting only of internal edges, and so we have $|Q| = 2+|T| \geq 4$, as required.
\end{proof}

We now show that local geodesics in $\Gamma(N)$ with no parabolic shortenings are transformed to local geodesics in $\CayG$.

\begin{prop} \label{prop:lg->lg}
Let $P$ be a $5$-local geodesic in $\Gamma(N)$ which has no parabolic shortenings. Then the derived path $\widehat{P}$ is a $2$-local geodesic in $\CayG$.
\end{prop}

\begin{figure}[ht]
\begin{subfigure}[t]{0.21\textwidth}
\centering
\begin{tikzpicture}[very thick,decoration={markings,mark=at position 0.5 with {\arrow{>}}}]
\draw [blue,postaction=decorate] (0,0) to node[midway,below] {$R'$} (2,0);
\draw [red,postaction=decorate] (0,0) to[out=90,in=180] node[midway,left] {$P_i'$} (1,2);
\draw [red,postaction=decorate] (1,2) to[out=0,in=90] node[midway,right] {$P_{i+1}'$} (2,0);
\draw [red,thick,postaction=decorate] (0,0) to node[near start,right] {$\!P_i$} (1,2);
\draw [red,thick,postaction=decorate] (1,2) to node[midway,below] {$P_{i+1}\ \ \ $} (2,0);
\fill (0,0) circle (2pt) (2,0) circle (2pt) (1,2) circle (2pt);
\end{tikzpicture}
\caption{General setup.}
\label{sfig:lglg-general}
\end{subfigure}\hfill%
\begin{subfigure}[t]{0.38\textwidth}
\centering
\begin{tikzpicture}[very thick,decoration={markings,mark=at position 0.5 with {\arrow{>}}}]
\draw [white] (0,0) to[out=-45,in=-90] (1.5,0);
\filldraw [dashed,thick,red!50,fill=red!10] (1.5,1.7) ellipse (2 and 0.9);
\draw [red,postaction=decorate] (0,0) to node[pos=0.4,left] {$e_1$} (0,1.5);
\draw [red,postaction=decorate] (0,1.5) to node[midway,above] {$T_1$} (1.5,1.5);
\draw [red,postaction=decorate] (1.5,0) to node[pos=0.4,left] {$e_2$} (1.5,1.5);
\draw [red,postaction=decorate] (1.5,1.5) to node[midway,above] {$T_2$} (3,1.5);
\draw [red,postaction=decorate] (3,0) to node[pos=0.4,left] {$e_3$} (3,1.5);
\fill (0,0) circle (2pt) (1.5,0) circle (2pt) (3,0) circle (2pt) (0,1.5) circle (2pt) (1.5,1.5) circle (2pt) (3,1.5) circle (2pt);
\end{tikzpicture}
\caption{$\widehat{P}_i'$, $\widehat{P}_{i+1}'$ in the same component.}
\label{sfig:lglg-pipi}
\end{subfigure}\hfill%
\begin{subfigure}[t]{0.38\textwidth}
\centering
\begin{tikzpicture}[very thick,decoration={markings,mark=at position 0.5 with {\arrow{>}}}]
\filldraw [dashed,thick,red!50,fill=red!10] (1.5,1.7) ellipse (2 and 0.9);
\draw [blue,postaction=decorate] (0,0) to node[pos=0.4,left] {$e_1$} (0,1.5);
\draw [blue,postaction=decorate] (0,1.5) to[bend left] node[midway,above] {$T_1$} (3.03,1.5);
\draw [blue] (3.03,1.5) to (3.03,0);
\draw [red,postaction=decorate] (0,0) to[out=-45,in=-90] node [pos=0.4,above] {$P_i'$} (1.5,0);
\draw [red,postaction=decorate] (1.5,0) to node[pos=0.4,left] {$e_2$} (1.5,1.5);
\draw [red,postaction=decorate] (1.5,1.5) to node[midway,below] {$T_2$} (2.97,1.5);
\draw [red] (2.97,1.5) to (2.97,0);
\draw [red!50!blue,decorate] (3,0) to node[midway,right] {$e_3$} (3,1.5);
\draw [thick,postaction=decorate] (0,1.5) to node[midway,below] {$T$} (1.5,1.5);
\fill (0,0) circle (2pt) (1.5,0) circle (2pt) (3,0) circle (2pt) (0,1.5) circle (2pt) (1.5,1.5) circle (2pt) (3,1.5) circle (2pt);
\end{tikzpicture}
\caption{$\widehat{P}_{i+1}'$, $R'$ in the same component.}
\label{sfig:lglg-pir}
\end{subfigure}
\caption{The proof of Proposition~\ref{prop:lg->lg}.}
\label{fig:lglg}
\end{figure}
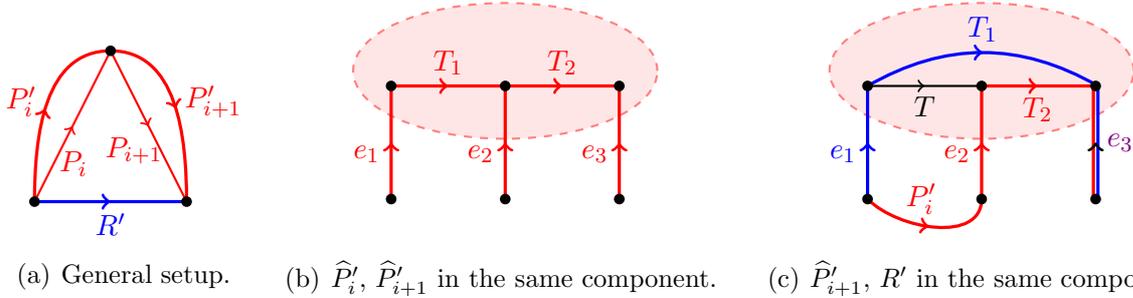

\begin{proof}
The conclusion is trivial if $|\widehat{P}| \leq 1$: we will thus assume that $|\widehat{P}| \geq 2$, and so $P$ has at least one free vertex. Let $P_1,\ldots,P_n \subseteq \Gamma(N)$ and $\widehat{P} = \widehat{P}_1 \cdots \widehat{P}_n \subseteq \CayG$ be as in Definition~\ref{defn:derived}. Note that by the construction we have $|\widehat{P}_i| \leq 1$ for $1 \leq i \leq n$.

Notice first that for $2 \leq i \leq n-1$, we have $(P_i)_- \neq (P_i)_+$. Indeed, by construction we have $|P_i| \geq 1$, and $(P_i)_-$ and $(P_i)_+$ are the only free vertices of $P_i$. If we had $(P_i)_- = (P_i)_+$, then, since $P$ has no parabolic shortenings, it would follow that $P_i = e \overline{e}$ for some connecting edge $e$, contradicting the fact that $P$ is $2$-local geodesic. Thus we have $(P_i)_- \neq (P_i)_+$, and so $|\widehat{P}_i| = 1$, for $2 \leq i \leq n-1$. It follows, in particular, that any subpath of $\widehat{P}$ of length $2$ is of the form $\widehat{P}_i \widehat{P}_{i+1}$ for some $i \in \{ 1,\ldots,n-1 \}$.

Now suppose for contradiction that $\widehat{P}$ is not a $2$-local geodesic, and so there exists an $i$ such that $|\widehat{P}_i| = |\widehat{P}_{i+1}| = 1$ and $d_{\XH}\left((\widehat{P}_i)_-,(\widehat{P}_{i+1})_+\right) \leq 1$. Let $\pathXH{R} \subseteq \CayG$ be a geodesic with $\pathXH{R}_- = (\widehat{P}_i)_-$ and $\pathXH{R}_+ = (\widehat{P}_{i+1})_+$, so that $|\pathXH{R}| \leq 1$ and $\pathXH{C} := \widehat{P}_i \widehat{P}_{i+1} \overline{\pathXH{R}}$ is a closed path in $\CayG$ of length $\leq 3$.

We now choose paths $P_i'$, $P_{i+1}'$ and $R'$ in $\Gamma(N)$ such that $P_i' P_{i+1}' \overline{R'}$ is a closed path in $\Gamma(N)$ that has no parabolic shortenings whose derived path is $\pathXH{C}$, as follows (see Figure~\ref{sfig:lglg-general}). For $i' \in \{i,i+1\}$ we take $P_{i'}' = P_{i'}$ if $P_{i'}$ is not an extended subpath of $P$. Otherwise, $P_{i'}$ is an extended subpath of $P$ of the form $e_1 T' \overline{e_2}$, where $e_1$ and $e_2$ are connecting edges and $T'$ is contained in a copy $\Gamma_0 \subseteq \Gamma(N)$ of $\Gamma_{j,N}$ for some $j \in \{1,\ldots,m\}$; we then take $P_{i'}' = e_1 T \overline{e_2}$, where $T$ is a geodesic in $\Gamma_0$ with $T_- = T'_-$ and $T_+ = T'_+$ (see Figure~\ref{sfig:lem3} for the case $i' = 1$). We take $R'$ to be a path such that $\widehat{R'} = \pathXH{R}$, so that $R'$ is a trivial path if $|\pathXH{R}| = 0$, and $R'$ consists of two free edges if $\pathXH{R}$ is labelled by an element of $X$; otherwise (if $\pathXH{R}$ is an $\widetilde{H}_j$-edge) $R' = e S \overline{e'}$ for some connecting edges $e$, $e'$ and a path $S \subseteq \Gamma_0$ for some copy $\Gamma_0$ of $\Gamma_{j,N}$ in $\Gamma(N)$ such that $S$ is geodesic in $\Gamma_0$.

In order to obtain a contradiction, we now study the components of $\pathXH{C}$.

\begin{description}

\item[If all components of $\pathXH{C}$ are isolated and of length $1$] In this case, the subpath $\pathXH{R}$ is either trivial (in which case $|R'| = 0$), or an edge labelled by an element of $X$ (in which case $|R'| = 2$), or an isolated component of $\pathXH{C}$ (in which case, by Lemma~\ref{lem:comp3} and since the path $P_i' P_{i+1}' R'$ has no parabolic shortenings, $|R'| \leq 3$). Thus $|R'| \leq 3$ in either case; similarly, $|\widehat{P}_i'|,|\widehat{P}_{i+1}'| \leq 3$. It then follows from Lemma~\ref{lem:comp4} that neither $P_i$ nor $P_{i+1}$ can be an extended subpath of $P$: thus, it follows from our construction that $P_i' = P_i$ and $P_{i+1}' = P_{i+1}$.

We therefore have $|P_i|,|P_{i+1}|,|R'| \leq 3$. But since $|\widehat{P}_i| = 1$, we have $(P_i)_- \neq (P_i)_+$, and as both of these vertices are free and no two free vertices of $\Gamma(N)$ are adjacent, it follows that $|P_i| \geq 2$; similarly, $|P_{i+1}| \geq 2$. We thus have $|R'| \leq 3 < 4 \leq |P_iP_{i+1}|$, and either $|P_iP_{i+1}| \leq 5$ or $|P_iP_{i+1}| = 6 \geq |R'|+3$. As $P_iP_{i+1}$ is a subpath of $P$ and as $R'_\pm = (P_iP_{i+1})_\pm$, this contradicts the fact that $P$ is a $5$-local geodesic.

\item[If $\widehat{P}_i$ and $\widehat{P}_{i+1}$ are in the same component of $\pathXH{C}$] In this case we have $P_i' = e_1 T_1 \overline{e_2}$ and $P_{i+1}' = e_2 T_2 \overline{e_3}$ where $e_1$, $e_2$ and $e_3$ are connecting edges, and $T_1,T_2 \subseteq \Gamma_0$ for some copy $\Gamma_0$ of $\Gamma_{j,N}$ in $\Gamma(N)$: see Figure~\ref{sfig:lglg-pipi}. By construction, we then have $P_i = e_1 T_1' \overline{e_2}$ and $P_{i+1} = e_2 T_2' \overline{e_3}$ for some paths $T_1',T_2' \subseteq \Gamma_0$, and so $\overline{e_2}e_2$ is a subpath of $P$. This contradicts the fact that $P$ is a $2$-local geodesic.

\item[Otherwise] It follows that either $\widehat{P_i}$ and $\pathXH{R}$ are connected components, or $\widehat{P}_{i+1}$ and $\pathXH{R}$ belong to the same component of $\pathXH{C}$ -- but not both. Without loss of generality, suppose that the latter is true. Then $R' = e_1 T_1 \overline{e_3}$ and $P_{i+1}' = e_2 T_2 \overline{e_3}$ where $e_1$, $e_2$ and $e_3$ are connecting edges, and $T_1,T_2 \subseteq \Gamma_0$ for some copy $\Gamma_0$ of $\Gamma_{j,N}$ in $\Gamma(N)$: see Figure~\ref{sfig:lglg-pir}. Moreover, $\pathXH{P}_i$ is either an edge labelled by an element of $X$ (in which case $|P_i'| = 2$) or an isolated component of $\pathXH{C}$ (in which case, by Lemma~\ref{lem:comp3}, $|P_i'| \leq 3$) -- therefore, $|P_i'| \leq 3$ in either case. In particular, by Lemma~\ref{lem:comp4} $P_i$ cannot be an extended subpath of $P$, and so $P_i' = P_i$.

Now let $T \subseteq \Gamma_0$ be a geodesic in $\Gamma_0$ with $T_- = (T_1)_-$ and $T_+ = (T_2)_-$, and let $R'' = e_1 T \overline{e_2}$. Then $C' := P_i \overline{R''}$ is a closed path in $\Gamma(N)$ with no parabolic shortenings, and $\overline{\widehat{R''}}$ is an isolated component of $\widehat{C'}$: therefore, by Lemma \ref{lem:comp3}, we have $|R''| \leq 3$, and so $|T| \leq 1$. Now consider the paths $e_1T$ and $P_ie_2$: we have $(e_1T)_\pm = (P_ie_2)_\pm$, and, by construction, $P_ie_2$ is a subpath of $P$. We have $|P_ie_2| = |P_i|+1 \leq 4$; on the other hand, as $(P_i)_-$, $(P_i)_+$ are distinct free vertices of $\Gamma(N)$, we have $|P_i| \geq 2$ and so $|e_1T| = 2 < 3 \leq |P_ie_2|$. This contradicts the fact that $P$ is a $4$-local geodesic. \qedhere

\end{description}
\end{proof}

Finally, we show that a word that belongs to the set $\widehat\Phi$ defined in Proposition~\ref{prop:lg=qg} cannot label the derived path of a geodesic in $\Gamma(N)$.

\begin{lem} \label{lem:Plong}
Let $P \subseteq \Gamma(N)$ be a path such that $\widehat{P}$ is labelled by a word in $\widehat\Phi$. Then $P$ is not a geodesic.
\end{lem}

\begin{figure}[ht]
\begin{tikzpicture}[very thick]
\draw [red] (0.03,0.06) to[out=115,in=-150] node [midway,left] {$P'$} (0.5,1) to[out=30,in=180] node [midway,below] {$P_2$} (2.5,1.5) to (3,1.5);
\draw [red,dotted] (3,1.5) -- (5.5,1.5);
\draw [red] (5.5,1.5) to (6,1.5) to[out=0,in=115] node [pos=0.6,left] {$P_{n-1}\!$} (7,0.06) to[out=-65,in=-115] node [midway,above] {$P''$} (7.97,0.06);
\draw (0.03,0.06) to[out=-65,in=-115] node[midway,above] {$Z_{P_-}$} (1.5,0.06);
\draw [blue] (0,0) to[out=-60,in=-120] node[midway,below] {$W'$} (1.53,0) to node[midway,below] {$U_1$} (3,0) to (3.5,0);
\draw [blue,dotted] (3.5,0) -- (5,0);
\draw [blue] (5,0) to (5.5,0) to node[midway,below] {$U_k$} (6.97,0) to[out=-60,in=-120] node [midway,below] {$W''$} (8,0);
\foreach \x in {1.5,3,5.5,7} \fill (\x,0) circle (2pt);
\fill (0.03,0.03) circle (3pt);
\fill (7.97,0.03) circle (3pt);
\fill (0.5,1) circle (2pt);
\fill (2.5,1.5) circle (2pt);
\fill (6,1.5) circle (2pt);
\end{tikzpicture}
\caption{An example of the paths $U$ (blue) and $P$ (red) in the proof of Lemma~\ref{lem:Plong}. In this case, $P_1$ is an extended subpath of $P$, and $P_n$ is trivial.}
\label{fig:Plong}
\end{figure}
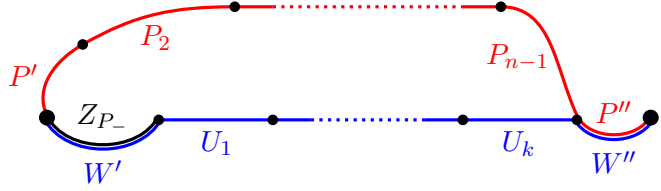

\begin{proof}
Note that, as $\widehat\Phi$ only consists of non-geodesic words, we have $|\widehat{P}| \geq 2$, and so $P$ contains at least one free vertex. Let $P = P' P_2 \cdots P_{n-1} P''$, $P_1$ and $P_n$ be as in Definition~\ref{defn:derived}, and let $\wordXH{P} \in \widehat\Phi$ be the word labelling $\widehat{P}$. Let $\wordXH{U_P}$ be the word chosen in Assumption~\ref{ass:Nbigger}\ref{it:Nbigg-Phi}, and let $\pathXH{U}_{\wordXH{P}} \subseteq \CayG$ be a path labelled by $\wordXH{U_P}$ with $(\pathXH{U}_{\wordXH{P}})_- = \widehat{P}_-$ and $(\pathXH{U}_{\wordXH{P}})_+ = \widehat{P}_+$.

Now write $\pathXH{U}_{\wordXH{P}} = \pathXH{e}_1 \cdots \pathXH{e}_k$, where each $\pathXH{e}_i$ is an edge in $\CayG$. We define a path $U \subseteq \Gamma(N)$ as $U = W' U_1 \cdots U_k W''$, where (see Figure~\ref{fig:Plong}):
\begin{enumerate}[label=(\roman*)]
\item $W' = Z_{P_-}$ if $P_1$ is an extended subpath of $P$, and $W' = P'$ otherwise (i.e.\ if $P_1$ is trivial);
\item for $1 \leq i \leq k$, $U_i$ is a path with no parabolic shortenings and with $(U_i)_-$ and $(U_i)_+$ free such that $\widehat{U}_i = \pathXH{e}_i$: that is, $U_i$ consists of two free edges if $\pathXH{e}_i$ is labelled by some $x \in X$, and all non-endpoint vertices of $U_i$ are $H_j$-internal if $\pathXH{e}_i$ is labelled by some $h \in \widetilde{H}_j$;
\item $W'' = \overline{Z_{P_+}}$ if $P_n$ is an extended subpath of $P$, and $W'' = P''$ otherwise.
\end{enumerate}
It follows from the construction that $U_\pm = P_\pm$ and that $\widehat{U} = \pathXH{U}_{\wordXH{P}}$. We aim to show that $|U| < |P|$.

Note first that for $1 \leq i \leq k$, the edge $\pathXH{e}_i$ is labelled either by an element of $X$ (in which case $|U_i| = 2$) or by some element $h \in \widetilde{H}_j$ (in which case by Assumption~\ref{ass:Nbigger}\ref{it:Nbigg-Phi} we have $d_{\Gamma_j}(v_j,v_j \cdot h) \leq N$ and hence, as $U_i$ has no parabolic shortenings, $|U_i| \leq 3$). Therefore, $|U_i| \leq 3$ for each $i$. We thus have
\begin{equation} \label{eq:Ushort}
|U| = |W'|+|W''|+\sum_{i=1}^k |U_i| \leq |W'|+|W''|+3k.
\end{equation}

On the other hand, the path $P_1$ is either an extended subpath of $P$ (in which case, by the construction in Definition~\ref{defn:derived}, we have $|P'| \geq 4$, $|\widehat{P}_1| = 1$, and $|W'| = |Z_{P_-}| \leq 2$: the latter follows from Assumption~\ref{ass:Nbigger}\ref{it:Nbigg-quot}), or trivial (in which case we have $|P'| = |W'|$ and $|\widehat{P}_1| = 0$). Therefore, we have $|P'| \geq |W'| + 2|\widehat{P}_1|$ in either case; similarly, we have $|P''| \geq |W''| + 2|\widehat{P}_n|$. Moreover, since by construction we have $|P_i| \geq 1$ for $2 \leq i \leq n-1$ and since $(P_i)_-$ and $(P_i)_+$ are free, we actually have $|P_i| \geq 2$ for $2 \leq i \leq n-1$. Therefore,
\begin{equation} \label{eq:Plong}
\begin{aligned}
|P| &= |P'| + |P''| + \sum_{i=2}^{n-1} |P_i| \geq \left( |W'|+2|\widehat{P}_1| \right) + \left( |W''|+2|\widehat{P}_n| \right) + 2(n-2)
\\ &= |W'|+|W''| + 2 \left( |\widehat{P}_1| + (n-2) + |\widehat{P}_n| \right) \geq |W'|+|W''| + 2|\widehat{P}|,
\end{aligned}
\end{equation}
where the last inequality follows since $|\widehat{P}_i| \leq 1$ for $2 \leq i \leq n-1$.

Finally, note that since $\widehat{P}$ is labelled by an element of $\widehat\Phi$, it follows from Proposition~\ref{prop:lg=qg} that $2|\widehat{P}| > 3d_{\XH}(\widehat{P}_-,\widehat{P}_+) = 3|\pathXH{U}_{\wordXH{P}}| = 3k$. This implies, together with \eqref{eq:Ushort} and \eqref{eq:Plong}, that
\[
|P| \geq |W'|+|W''|+2|\widehat{P}| > |W'|+|W''|+3k \geq |U|.
\]
Thus $P$ is not a geodesic, as required.
\end{proof}

We now use Propositions~\ref{prop:lg=qg}~\&~\ref{prop:lg->lg} together with Lemma~\ref{lem:Plong} to prove Theorem~\ref{thm:g->qg}. In particular, the set $\Phi$ in Theorem~\ref{thm:g->qg} is constructed using Proposition~\ref{prop:lg=qg}, and shown to only contain non-geodesic words using Lemma~\ref{lem:Plong}, after which the result follows from Propositions~\ref{prop:lg=qg}~\&~\ref{prop:lg->lg}.

\begin{proof}[Proof of Theorem~\ref{thm:g->qg}]
Let $\lambda \geq 0$, $c \geq 0$ and $\widehat\Phi$ be the constants and the set of words over $\XH$ given in Proposition~\ref{prop:lg=qg}, and let $N \geq 1$ be a constant such that Assumption~\ref{ass:Nbigger} is true. Let $\Phi_0 = \mathcal{A} \cup \mathcal{B}$, where
\begin{enumerate}[label=(\roman*)]
\item $\mathcal{A}$ is the set of paths $P \subseteq \Gamma(N)$ with no parabolic shortenings such that $\widehat{P}$ is labelled by a word in $\widehat\Phi$, and
\item $\mathcal{B}$ is the set of non-geodesic paths $P \subseteq \Gamma(N)$ of length $|P| \leq 5$.
\end{enumerate}

It is clear that both $\mathcal{A}$ and $\mathcal{B}$ are $G$-invariant. Moreover, note that, by Lemma~\ref{lem:geomaction}, the graph $\Gamma(N)$ is proper (as a metric space) and the action $G \acts \Gamma(N)$ is geometric. Now it is easy to see from Definition~\ref{defn:derived} that given a word $\wordXH{Q}$ over $\XH$ there is a bound on the lengths of paths $Q \subseteq \Gamma(N)$ without parabolic shortenings such that $\widehat{Q}$ is labelled by $\wordXH{Q}$. As $\widehat{\Phi}$ is finite, this implies that $\mathcal{A}$ is a union of finitely many $G$-orbits. Similarly, we can see that $\mathcal{B}$ consists of finitely many $G$-orbits, and so the orbit space $\Phi_0 / G$ is finite. We set $\Phi$ to be a (finite) set of paths in $\Gamma(N)$ consisting of one representative for each $G$-orbit in $\Phi_0$. By Lemma~\ref{lem:Plong}, none of the paths in $\Phi_0$ (and so in $\Phi$) are geodesic.

Now let $P \subseteq \Gamma(N)$ be a path in $\Gamma(N)$ that has no parabolic shortenings and does not contain any $G$-translate of a path in $\Phi$ (i.e.\ any path in $\Phi_0$) as a subpath. It then follows that $P$ is $5$-local geodesic, and that $\widehat{P}$ does not have a subpath labelled by an element of $\widehat\Phi$ (as, if it did, any such subpath could be expressed as $\widehat{Q}$ for a subpath $Q \subseteq P$). By Proposition~\ref{prop:lg->lg}, $\widehat{P}$ is a $2$-local geodesic, and thus, by Proposition~\ref{prop:lg=qg}, it is a $(\lambda,c)$-quasi-geodesic that does not vertex backtrack (and so does not backtrack). This establishes the result.
\end{proof}

\section{Hellyness} \label{sec:helly}

In this section, we prove Proposition~\ref{prop:helly-coarse}, implying immediately Theorem~\ref{thm:coarsehelly} and, consequently, we prove Theorem~\ref{thm:helly}. We first recall a few definitions, following 
\cite{ccgho}.

\begin{defn} \label{defn:helly}
Let $\Gamma$ be a graph.
\begin{enumerate}[label=(\roman*)]
\item \label{it:helly-stint} Given $\beta \geq 1$, we say $\Gamma$ has \emph{$\beta$-stable intervals} (or simply \emph{stable intervals}) if for any geodesic path $P \subseteq \Gamma$, any vertex $w \in P$ and any $u \in V(\Gamma)$ with $d_\Gamma(u,P_-) = 1$, there exists a geodesic $Q \subseteq \Gamma$ and a vertex $v \in Q$ such that $Q_- = u$, $Q_+ = P_+$ and $d_\Gamma(w,v) \leq \beta$.
\item Given $\rho \geq 0$ and $w \in V(\Gamma)$, the \emph{ball} $B_\rho(w) = B_\rho(w;\Gamma)$ is the set of all $u \in V(\Gamma)$ such that $d_\Gamma(u,w) \leq \rho$. Given a constant $\xi \geq 0$ and a collection $\mathcal{B} = \{ B_{\rho_i}(w_i) \mid i \in \mathcal{I} \}$ of balls in $\Gamma$, we say $\mathcal{B}$ satisfies the \emph{$\xi$-coarse Helly property} if $\bigcap_{i \in \mathcal{I}} B_{\rho_i+\xi}(w_i) \neq \varnothing$. We say $\Gamma$ is \emph{coarsely Helly} (respectively, \emph{Helly}) if there exists a constant $\xi \geq 0$ such that every collection of pairwise intersecting balls in $\Gamma$ satisfies the $\xi$-coarse Helly property (respectively, the $0$-coarse Helly property).
\item A group $G$ is said to be \emph{Helly} (respectively, \emph{coarsely Helly}) if it acts \emph{geometrically}, that is, properly discontinuously and cocompactly, by graph automorphisms on a Helly (respectively, coarsely Helly) graph.
\end{enumerate}
\end{defn}

The merit of these definitions is supported by the following result.

\begin{thm}[{\cite[Theorem~1.2]{ccgho}}] \label{thm:ccgho-helly}
If $G$ is a group acting geometrically on a coarsely Helly graph that has stable intervals, then $G$ is Helly.
\end{thm}

We now start our proof of Theorem~\ref{thm:helly}. Throughout the remainder of this section, we adopt the following terminology.

Let $G$ be a finitely generated group (with a finite generating set $X$, say), and suppose that $G$ is hyperbolic relative to a collection of its subgroups. As $G$ is finitely generated, such a collection is finite (see \cite[Corollary~2.48]{osin06}): $H_1,\ldots,H_m$, say. Let $\widetilde{H}_1,\ldots,\widetilde{H}_m$ be isomorphic copies of $H_1,\ldots,H_m$ (respectively), $\epsilon: F(X) * (*_{j=1}^m \widetilde{H}_j) \to G$ the canonical surjection, and $\mathcal{H} = \bigsqcup_{j=1}^m (\widetilde{H}_j \setminus \{1\})$, as in Section~\ref{ssec:relhyp}.

Furthermore, let $\Gamma_1,\ldots,\Gamma_m$ be proper graphs such that $\widetilde{H}_j$ acts on $\Gamma_j$ geometrically (for each $j$); let $v_j \in V(\Gamma_j)$ be (fixed, but arbitrarily chosen) basepoints. We fix constants $\lambda,N \geq 1$ and $c \geq 0$ given by Theorem~\ref{thm:g->qg}, and construct the graphs $\Gamma_{j,N}$ (respectively, $\Gamma(N)$) with geometric actions of $\widetilde{H}_j$ (respectively, $G$) as in Section~\ref{ssec:GammaN}.

The idea of our proof of Theorem~\ref{thm:helly} is to first use Theorem~\ref{thm:g->qg} to transform geodesics in $\Gamma(N)$ into $(\lambda,c)$-quasi-geodesics in $\CayG$, and then use Theorem~\ref{thm:bcp} or Proposition~\ref{prop:BCPtriangles} to show that if each $\Gamma_j$ is coarsely Helly (respectively, has stable intervals), then $\Gamma(N)$ is coarsely Helly (respectively, has stable intervals) as well. Theorem~\ref{thm:helly} will then follow from Theorem~\ref{thm:ccgho-helly}.

\subsection{Stable intervals} \label{ssec:stable-intervals}

Here we show that if each $\Gamma_j$ has stable intervals then so does $\Gamma(N)$. We start by taking a geodesic $P \subseteq \Gamma(N)$ and vertices $w \in P$, $u \in V(\Gamma(N))$ as in Definition~\ref{defn:helly}\ref{it:helly-stint}, and choosing any geodesic $R \subseteq \Gamma(N)$ with $R_- = u$ and $R_+ = P_+$. Then, by Theorem~\ref{thm:g->qg}, the derived paths $\widehat{P}$ and $\widehat{R}$ are quasi-geodesics that do not backtrack. If $w$ is `close' to a free vertex of $P$, then we can use Theorem~\ref{thm:bcp} to show that $w$ is also `close' to a vertex of $R$. Otherwise, $w$ is in a copy $\Gamma_0$ of some $\Gamma_{j,N}$, and $S := R \cap \Gamma_0$ is a non-empty subpath by Theorem~\ref{thm:bcp}; we may then use the stable interval property in $\Gamma_0$ to replace the geodesic subpath $S \subseteq \Gamma_0$ of $R$ by another geodesic subpath containing a vertex that is `close' to $w$ in $\Gamma_0$.

We first show how we can pass the stable interval property from $\Gamma_j$ to $\Gamma_{j,N}$, as follows.

\begin{lem} \label{lem:GammajN-stint}
Let $\beta \geq 1$. If $\Gamma_j$ has $\beta$-stable intervals, then $\Gamma_{j,N}$ has $\left( 3\beta+1 \right)$-stable intervals.
\end{lem}

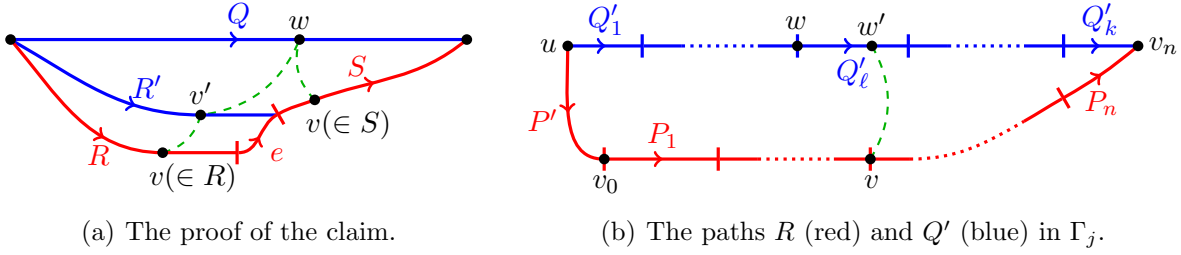
\begin{figure}[ht]
\begin{subfigure}[t]{0.42\textwidth}
\centering
\begin{tikzpicture}[thick,decoration={markings,mark=at position 0.5 with {\arrow{>}}}]
\draw [green!70!black,dashed,thick] (2,-1.5) to[bend right] (2.5,-1) to[bend right] (3.8,0);
\draw [green!70!black,dashed,thick] (4,-0.8) to[bend left] (3.8,0);
\draw [blue,very thick,postaction=decorate] (0,0) -- (6,0) node [midway,above] {$Q$};
\draw [blue,very thick,postaction=decorate] (0,0) to[out=-30,in=180] node [near end,above] {$R'$} (2.5,-1) -- (3.5,-1);
\draw [red,very thick,-|,postaction=decorate] (0,0) to[out=-50,in=180] node [pos=0.65,below] {$R$} (2,-1.5) -- (3,-1.5);
\draw [red,very thick,postaction=decorate] (3,-1.5) to[out=0,in=-150] node [midway,below right] {$e$} (3.5,-1);
\draw [red,very thick,|-,postaction=decorate] (3.5,-1) to[out=30,in=-160] (4,-0.8) to[out=20,in=-140] node [near start,above] {$S$} (6,0);
\fill (0,0) circle (2pt);
\fill (6,0) circle (2pt);
\fill (3.8,0) circle (2pt) node [above] {$w$};
\fill (2.5,-1) circle (2pt) node [above] {$v'$};
\fill (2,-1.5) circle (2pt) node [below right] {$\!\!\!\!\!v(\in R)$};
\fill (4,-0.8) circle (2pt) node [below right] {$\!\!\!\!v(\in S)$};
\end{tikzpicture}
\caption{The proof of the claim.}
\label{sfig:Gamma-stint-cm}
\end{subfigure}\hfill%
\begin{subfigure}[t]{0.57\textwidth}
\centering
\begin{tikzpicture}[thick,decoration={markings,mark=at position 0.5 with {\arrow{>}}}]
\draw [green!70!black,dashed,thick] (3.98,-1.5) to[bend right] (4,0);
\draw [blue,very thick,-|,postaction=decorate] (0,0) -- (1,0) node [midway,above] {$Q'_1$};
\draw [blue,very thick] (1,0) -- (1.5,0);
\draw [blue,very thick,dotted] (1.5,0) -- (2.5,0);
\draw [blue,very thick] (2.5,0) -- (3,0);
\draw [blue,very thick,|-|,postaction=decorate] (3,0) node[above,yshift=2pt,black] {$w$} -- (4.5,0) node [midway,below] {$Q'_\ell$};
\draw [blue,very thick] (4.5,0) -- (5,0);
\draw [blue,very thick,dotted] (5,0) -- (6,0);
\draw [blue,very thick] (6,0) -- (6.5,0);
\draw [blue,very thick,|-,postaction=decorate] (6.5,0) -- (7.5,0) node [midway,above] {$Q'_k$};
\draw [red,very thick,-|,postaction=decorate] (0,0) to[out=-90,in=180] node [left] {$P'$} (0.5,-1.5) node [below,yshift=-2pt,black] {$v_0$};
\draw [red,very thick,-|,postaction=decorate] (0.5,-1.5) -- (2,-1.5) node [midway,above] {$P_1$};
\draw [red,very thick] (2,-1.5) -- (2.5,-1.5);
\draw [red,very thick,dotted] (2.5,-1.5) -- (3.5,-1.5);
\draw [red,very thick,-|] (3.5,-1.5) -- (4,-1.5) node [below,yshift=-2pt,black] {$v$};
\draw [red,very thick] (4,-1.5) -- (4.5,-1.5);
\draw [red,very thick,dotted] (4.5,-1.5) to[out=0,in=-150] (6,-1);
\draw [red,very thick] (6,-1) to[out=30,in=-150] (6.5,-0.7);
\draw [red,very thick,|-,postaction=decorate] (6.5,-0.7) to[out=30,in=-140] node [midway,below,yshift=-2pt] {$P_n$} (7.5,0);
\fill (0,0) circle (2pt) node [left] {$u$};
\fill (4,0) circle (2pt) node [above] {$w'$};
\fill (7.5,0) circle (2pt) node [right] {$v_n$};
\fill (3.02,0) circle (2pt);
\fill (0.48,-1.5) circle (2pt);
\fill (3.98,-1.5) circle (2pt);
\end{tikzpicture}
\caption{The paths $R$ (red) and $Q'$ (blue) in $\Gamma_j$.}
\label{sfig:Gamma-stint-j}
\end{subfigure}
\caption{The proof of Lemma~\ref{lem:GammajN-stint}.}
\label{fig:Gamma-stint}
\end{figure}

\begin{proof}
We first prove the following Claim.

\begin{cm}
Let $k \geq 0$. If $P \subseteq \Gamma_j$ is a path of length $d_{\Gamma_j}(P_-,P_+)+k$ and $v \in P$ is a vertex, then there exists a geodesic $Q \subseteq \Gamma_j$ with $Q_\pm = P_\pm$ and a vertex $w \in Q$ such that $d_{\Gamma_j}(v,w) \leq \beta k$.
\end{cm}

\begin{cmproof}
By induction on $k$. The base case, $k = 0$, is trivial.

Suppose that the claim holds for $k \leq r-1$, for some $r \geq 1$. Let $P \subseteq \Gamma_j$ be a path of length $d_{\Gamma_j}(P_-,P_+)+r$, and let $v \in P$ be a vertex. Let $R \subseteq P$ be the longest geodesic subpath of $P$ with $R_- = P_-$. Thus $Re \subseteq P$ is a non-geodesic subpath for an edge $e \subseteq P$. Let $S \subseteq P$ be the subpath such that $P = ReS$; see Figure~\ref{sfig:Gamma-stint-cm}.

Suppose first that $v \notin R$. Then we have $v \in S$, and so $v \in R'S$, where $R' \subseteq \Gamma_j$ is any geodesic with $R'_\pm = (Re)_\pm$. Moreover, we have
\begin{equation} \label{eq:RpS}
|R'S| = |R'| + |S| \leq |Re|-1 + |S| = |P|-1 = d_{\Gamma_j}((R'S)_-,(R'S)_+) + (r-1).
\end{equation}
Thus, by the inductive hypothesis, there exists a geodesic $Q \subseteq \Gamma_j$ with $Q_\pm = (R'S)_\pm = P_\pm$ and a vertex $w \in Q$ with $d_{\Gamma_j}(v,w) \leq \beta(r-1) < \beta r$, as claimed.

Suppose now that $v \in R$. Since $\Gamma_j$ has $\beta$-stable intervals, there exists a geodesic $R' \subseteq \Gamma_j$ with $R'_\pm = (Re)_\pm$ and a vertex $v' \subseteq R'$ such that $d_{\Gamma_j}(v,v') \leq \beta$. By \eqref{eq:RpS}, we can apply the inductive hypothesis to $R'S$; in particular, since $v' \in R'S$, there exists a geodesic $Q \subseteq \Gamma_j$ with $Q_\pm = (R'S)_\pm = P_\pm$ and a vertex $w \in Q$ with $d_{\Gamma_j}(v',w) \leq \beta(r-1)$. It then follows that $d_{\Gamma_j}(v,w) \leq d_{\Gamma_j}(v,v') + d_{\Gamma_j}(v',w) \leq \beta r$, as claimed.
\end{cmproof}

Now let $P \subseteq \Gamma_{j,N}$ be a geodesic, let $u \in V(\Gamma_{j,N})$ be a vertex with $d_{\Gamma_{j,N}}(u,P_-) = 1$, and let $v \in P$ be any vertex. Let $n = |P|$ and let $P_- = v_0,\ldots,v_n = P_+$ be the vertices of $P$, so that $d_{\Gamma_{j,N}}(v_{i-1},v_i) = 1$ for each $i$. By construction, it then follows that $u,v_0,\ldots,v_n$ are in $V(\Gamma_j)$, and that there exist paths $P',P_1,\ldots,P_n \subseteq \Gamma_j$ such that we have $(P_i)_- = v_{i-1}$, $(P_i)_+ = v_i$ and $|P_i| \leq N$ for each $i$, as well as $P'_- = u$, $P'_+ = v_0$ and $|P'| \leq N$.

Consider the path $R = P'P_1\cdots P_n \subseteq \Gamma_j$. By construction, $v \in R$, and $|R| = |P'| + \sum_{i=1}^n |P_i| \leq (n+1)N$. On the other hand, since $P \subseteq \Gamma_{j,N}$ is a geodesic, we have $n = |P| = d_{\Gamma_{j,N}}(v_0,v_n) = \left\lceil d_{\Gamma_j}(v_0,v_n)/N \right\rceil$; therefore, $d_{\Gamma_j}(v_0,v_n) > (n-1)N$. It thus follows that
\[
d_{\Gamma_j}(R_-,R_+) \geq d_{\Gamma_j}(v_0,v_n) - d_{\Gamma_j}(u,v_0) > (n-1)N - N = (n-2)N \geq |R|-3N.
\]
Therefore, by the Claim, there exists a geodesic $Q' \subseteq \Gamma_j$ with $Q'_\pm = R_\pm$ and a vertex $w' \in Q'$ such that $d_{\Gamma_j}(v,w') \leq 3\beta N$; see Figure~\ref{sfig:Gamma-stint-j}.

Finally, we can write $Q' = Q'_1 \cdots Q'_k$, where $k = \lceil |Q'|/N \rceil$, with $|Q'_i| \leq N$ for each $i$. By construction, there exist edges $e_1,\ldots,e_k \subseteq \Gamma_{j,N}$ with $(e_i)_\pm = (Q'_i)_\pm$ for each $i$. We have $|e_1 \cdots e_k| = k = \left\lceil d_{\Gamma_j}(Q'_-,Q'_+) / N \right\rceil$ since $Q' \subseteq \Gamma_j$ is a geodesic, and so $Q := e_1 \cdots e_k \subseteq \Gamma_{j,N}$ is a geodesic. Moreover, since $w' \in Q'$ we have $w' \in Q'_\ell$ for some $\ell$, and so $d_{\Gamma_j}(w',w) \leq |Q'_\ell| \leq N$, where $w = (Q'_\ell)_-$; note that $w$ is a vertex of $Q$. Thus $Q$ is a geodesic in $\Gamma_{j,N}$ with $Q_- = u$ and $Q_+ = P_+$, and $w \in Q$ is a vertex such that
\[
d_{\Gamma_j}(v,w) \leq d_{\Gamma_j}(v,w') + d_{\Gamma_j}(w',w) \leq 3\beta N + N,
\]
and so $d_{\Gamma_{j,N}}(v,w) = \left\lceil d_{\Gamma_j}(v,w)/N \right\rceil \leq 3\beta+1$. This proves that $\Gamma_{j,N}$ has $(3\beta+1)$-stable intervals, as required.
\end{proof}

\begin{prop} \label{prop:helly-stint}
If each $\Gamma_j$ has stable intervals, then so does $\Gamma(N)$.
\end{prop}

\begin{figure}[ht]
\begin{subfigure}[t]{0.98\textwidth}
\centering
\begin{tikzpicture}[very thick,decoration={markings,mark=at position 0.5 with {\arrow{>}}}]
\begin{scope}[xshift=0]
\draw [black!70] (0,0) to[out=90,in=-100] (0.1,1) to[out=80,in=-120] (0.5,2);
\fill [green!70!black] (0.1,1) circle (2.5pt) node [right] {$R_- = [(P_-,x)]$};
\fill [black!70] (0,0) circle (2.5pt) node (x1) {} node [below] {$P_-=\widehat{P}_-$} (0.5,2) circle (2.5pt) node (x2) {} node [above] {$\epsilon(x)P_-$};
\draw (-1,1.2) node [circle,inner sep=0pt] (y) {$\widehat{R}_-$};
\draw [dotted,->,thick] (y) to[bend left] (x2);
\draw [dotted,->,thick] (y) to[bend right] (x1);
\end{scope}
\begin{scope}[xshift=4.5cm]
\filldraw [red!50,fill=red!10,dashed,thick] (1,1.7) ellipse (1.5 and 0.8) node [above] {$\Gamma_0$};
\draw [blue,postaction=decorate] (0,0) to node[midway,left] {$e_P$} (0,1.5);
\draw [red] (0,1.5) to[bend right] (2,1.5);
\draw [blue,postaction=decorate] (2.5,0) to[bend right] node[midway,right] {$e_R$} (2,1.5);
\fill [red] (0,1.5) circle (2.5pt) node [above] {$R_-$} (2,1.5) circle (1.5pt);
\fill [black!70] (0,0) circle (2.5pt) node [below] {$P_-=\widehat{P}_-$} (2.5,0) circle (2.5pt) node [below] {$\widehat{R}_-$};
\end{scope}
\begin{scope}[xshift=10cm]
\filldraw [red!50,fill=red!10,dashed,thick] (1,1.7) ellipse (1.7 and 0.8) node [above] {$\Gamma_0$};
\draw [blue,postaction=decorate] (-0.5,0) to[bend left=20] node[midway,left] {$e_P$} (-0.2,1.5);
\draw [red] (-0.2,1.5) to[bend right] (0.7,1.5) to (1.3,1.5) to[bend right] (2.2,1.5);
\draw [blue,postaction=decorate] (2.5,0) to[bend right=20] node[midway,right] {$e_R$} (2.2,1.5);
\fill [red] (-0.2,1.5) circle (1.5pt) (0.7,1.5) circle (2.5pt) node [below] {$P_-\ $} (1.3,1.5) circle (2.5pt) node [below] {$\ \ R_-$} (2.2,1.5) circle (1.5pt);
\fill [black!70] (-0.5,0) circle (2.5pt) node [below] {$\widehat{P}_-$} (2.5,0) circle (2.5pt) node [below] {$\widehat{R}_-$};
\end{scope}
\end{tikzpicture}
\caption{Bounding $d_X(\widehat{P}_-,\widehat{R}_-)$: $P_-$ free, $R_-$ medial (left); $P_-$ free, $R_- \in \Gamma_0$ (centre); $P_-,R_- \in \Gamma_0$ (right).}
\label{sfig:stint-dpmrm}
\end{subfigure}

\begin{subfigure}[t]{0.5\textwidth}
\centering
\begin{tikzpicture}[very thick,decoration={markings,mark=at position 0.5 with {\arrow{>}}}]
\filldraw [red!50,fill=red!10,dashed,thick] (0,0) ellipse (3.5 and 1.5) node [left] {$\Gamma_0$};
\draw [blue,postaction=decorate] (-3,-2) to [bend left=20] node [midway,left] {$e$} (-2.5,-0.5);
\draw [blue,postaction=decorate] (3,-2) to [bend right=20] node [midway,right] {$e'$} (2.5,-0.5);
\draw [blue,postaction=decorate] (-3,2) to [bend right=20] node [midway,left] {$f$} (-2.5,0.5);
\draw [blue,postaction=decorate] (3,2) to [bend left=20] node [midway,right] {$f'$} (2.5,0.5);
\draw [red,postaction=decorate] (-2.5,-0.5) to [bend right] node [midway,below] {$e_0$} (-1.5,-0.5);
\draw [red,-|,postaction=decorate] (-1.5,-0.5) to [out=-15,in=-175] node [midway,below] {$P_i''$} (0.5,-0.8) node [below,xshift=5pt] {$w$};
\draw [red] (0.5,-0.8) to[out=5,in=-165] (1.5,-0.5);
\draw [red,postaction=decorate] (1.5,-0.5) to [bend right] node [midway,below,yshift=3pt] {$e_0'$} (2.5,-0.5);
\draw [red,postaction=decorate] (-2.5,0.5) to [bend left=10] node [midway,above] {$R_{i'}'$} (2.5,0.5);
\draw [red,dashed] (-2.5,-0.5) to (-2.5,0.5) (2.5,-0.5) to (2.5,0.5);
\fill [red] (-2.5,-0.5) circle (2.5pt) (-1.5,-0.5) circle (2.5pt) (-2.5,0.5) circle (2.5pt) (2.5,-0.5) circle (2.5pt) (1.5,-0.5) circle (2.5pt) (2.5,0.5) circle (2.5pt);
\fill [black!70] (-3,-2) circle (2.5pt) node [right] {$g$} (3,-2) circle (2.5pt) node [left] {$hg$} (-3,2) circle (2.5pt) node [right] {$h_0g$} (3,2) circle (2.5pt) node [left] {$h_1g$};
\end{tikzpicture}
\caption{The paths $P_i$ and $R_{i'}$.}
\label{sfig:stint-wfar}
\end{subfigure} \hfill
\begin{subfigure}[t]{0.4\textwidth}
\centering
\begin{tikzpicture}[very thick,decoration={markings,mark=at position 0.5 with {\arrow{>}}}]
\filldraw [red!50,fill=red!10,dashed,thick] (0,0) ellipse (3 and 2);
\draw [red,postaction=decorate] (-2,-1) to [out=-15,in=-175] node [midway,below] {$P_i''$} (0.5,-1.3);
\draw [red] (0.5,-1.3) node [below] {$w$} to [out=5,in=-165] (2,-1);
\draw [red,thick] (-2,-0.5) to [out=-25,in=175] (0.55,-1.1) to [out=-5,in=-175] (2,-1);
\draw [red,thick] (-2,0.5) to [out=-40,in=165] (0.65,-0.85) to [out=-15,in=175] (2,-1);
\draw [red,thick] (-2,1) to [out=-45,in=160] (0.7,-0.7) to [out=-20,in=165] (2,-1);
\draw [red,thick] (-2,1) to [out=-40,in=165] (0.8,-0.3) to [out=-15,in=175] (2,-0.5);
\draw [red,postaction=decorate] (-2,1) to [out=-15,in=-175] node [pos=0.7,below] {$S'$} (1,0.8) node [above] {$v$} to [out=5,in=-165] (2,1);
\draw [red,postaction=decorate] (-2,1) to [out=15,in=165] node [midway,above] {$S$} (2,1);
\draw [red,postaction=decorate] (-2,-1) -- (-2,1) node [midway,left] {$T$};
\draw [red,postaction=decorate] (2,-1) -- (2,1) node [midway,right] {$T'$};
\draw [dotted] (-1.6,-0.55) -- (-1.6,0.1) (1.6,-0.3) -- (1.6,0.8);
\fill [red] (-2,1) circle (2.5pt) (-2,-1) circle (2.5pt) (2,1) circle (2.5pt) (2,-1) circle (2.5pt);
\fill [red] (0.5,-1.3) circle (1.5pt) (0.55,-1.1) circle (1.5pt) (0.65,-0.85) circle (1.5pt) (0.7,-0.7) circle (1.5pt) (0.8,-0.3) circle (1.5pt) (1,0.8) circle (1.5pt);
\draw [red,densely dotted,thick] (0.5,-1.3) -- (0.55,-1.1) (0.65,-0.85) -- (0.7,-0.7) -- (0.8,-0.3);
\draw [red,dotted,thick] (0.55,-1.1) -- (0.65,-0.85) (0.8,-0.3) -- (1,0.8);
\end{tikzpicture}
\caption{`Moving' along $T$ and $T'$.}
\label{sfig:stint-ttp}
\end{subfigure}

\caption{The proof of Proposition~\ref{prop:helly-stint}. Colours are the same as in Figure~\ref{fig:Gamma-mod-F}.}
\label{fig:helly-stint}
\end{figure}
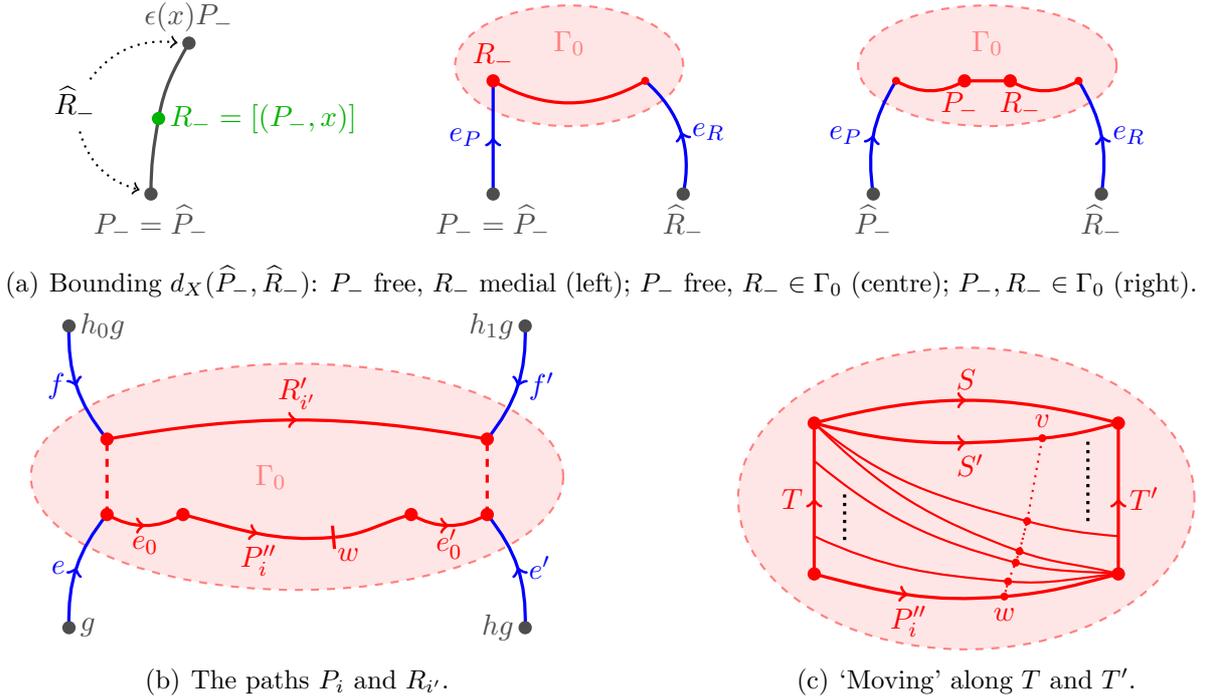

\begin{proof}
Since each $\Gamma_{j,N}$ is locally finite and the action $\widetilde{H}_j \acts \Gamma_{j,N}$ is properly discontinuous, there are finitely many elements $h \in \widetilde{H}_j$ satisfying $d_{\Gamma_{j,N}}(v_j,v_j \cdot h) \leq 5$. We may therefore choose a constant $k \geq 1$ such that $d_X(1,h) \leq k$ whenever $j \in \{1,\ldots,m \}$ and $h \in \widetilde{H}_j$ are such that $d_{\Gamma_{j,N}}(v_j,v_j \cdot h) \leq 5$. Let $\varepsilon = \varepsilon(\lambda,c,k) \geq 0$ be the constant given by Theorem~\ref{thm:bcp}. Furthermore, since $X$ is finite and, for each $j$, the homomorphism $\epsilon|_{\widetilde{H}_j}$ is injective, there are finitely many elements $h \in \widetilde{H}_j$ such that $d_X(1,\epsilon(h)) \leq \varepsilon$. We may thus choose a constant $\overline\varepsilon \geq 0$ such that we have $d_{\Gamma_{j,N}}(v_j,v_j \cdot h) \leq \overline\varepsilon$ whenever $j \in \{1,\ldots,m \}$ and $h \in \widetilde{H}_j$ are such that $d_X(1,\epsilon(h)) \leq \varepsilon$.

Let $\beta_0 \geq 1$ be such that $\Gamma_1,\ldots,\Gamma_m$ all have $\beta_0$-stable intervals. We set
\[
\beta := \max \left\{ \frac{\overline\varepsilon}{2}+2\varepsilon+5, (3\beta_0+1)(2\overline\varepsilon+4) \right\}.
\]
We aim to show that $\Gamma(N)$ has $\beta$-stable intervals. In particular, let $P \subseteq \Gamma(N)$ be a geodesic, let $w \in P$ be a vertex, and let $x \in V(\Gamma(N))$ be such that $d_{\Gamma(N)}(P_-,x) = 1$. We will find a geodesic $Q \subseteq \Gamma(N)$ with $Q_- = x$ and $Q_+ = P_+$ and a vertex $v \in Q$ such that $d_{\Gamma(N)}(w,v) \leq \beta$.

Let $P = P' P_2 \cdots P_{n-1} P''$ and $P_1,P_n \subseteq \Gamma(N)$ be as in Definition~\ref{defn:derived}. Let $R \subseteq \Gamma(N)$ be an arbitrary geodesic with $R_- = x$ and $R_+ = P_+$, and, similarly to the case of $P$, let $R = R' R_2 \cdots R_{\ell-1} R''$ and $R_1,R_\ell \subseteq \Gamma(N)$ be as in Definition~\ref{defn:derived}. 

Note first that since $P$ and $R$ are geodesics, by Theorem~\ref{thm:g->qg} the derived paths $\widehat{P}$ and $\widehat{R}$ are $2$-local geodesic $(\lambda,c)$-quasi-geodesics that do not backtrack. We furthermore claim that $\widehat{P}$ and $\widehat{R}$ are $k$-similar: that is, $d_X(\widehat{P}_-,\widehat{R}_-) \leq k$ and $d_X(\widehat{P}_+,\widehat{R}_+) \leq k$. Indeed, we have $d_X(\widehat{P}_-,\widehat{R}_-) \leq k$ (see Figure~\ref{sfig:stint-dpmrm}):
\begin{enumerate}[label=(\roman*)]
\item If either $P_-$ or $R_-$ is a medial vertex, then the other one is free, and $d_X(\widehat{P}_-,\widehat{R}_-) \leq 1$ by the construction in Definition~\ref{defn:derived}.
\item Otherwise, one of $P_-$ and $R_-$ is $H_j$-internal (for some $j$) -- and so belongs to some copy $\Gamma_0$ of $\Gamma_{j,N}$ in $\Gamma(N)$ -- and the other one is either in $\Gamma_0$ as well or free. Then, by construction, there exist connecting edges $e_P$ and $e_R$ in $\Gamma(N)$ such that $(e_P)_- = \widehat{P}_-$, $(e_R)_- = \widehat{R}_-$ and $(e_P)_+,(e_R)_+ \in \Gamma_0$. Moreover, if $P_-$ is $H_j$-internal, then by construction $d_{\Gamma_0}((e_P)_+,P_-) \leq 2$, whereas if $P_-$ is free then $(e_P)_+ = R_-$; similarly for $R_-$. It follows that $d_{\Gamma_0}((e_P)_+,(e_R)_+) \leq 5$. Therefore, $d_X(\widehat{P}_-,\widehat{R}_-) \leq k$, as required.
\end{enumerate}
A similar argument shows that $d_X(\widehat{P}_+,\widehat{R}_+) \leq k$. We may thus apply Theorem~\ref{thm:bcp} to the paths $\widehat{P}$ and $\widehat{R}$.

We argue in two parts, based on the minimal distance between $w$ and a vertex of $\widehat{P}$.

\begin{description}

\item[If $d_{\Gamma(N)}(w,w') \leq \frac{\overline\varepsilon}{2}+3$ for some vertex $w'$ of $\widehat{P}$] Since $\widehat{P}$ is a $2$-local geodesic, all vertices of $\widehat{P}$ are phase -- in particular, $w' \in \widehat{P}$ is phase. It then follows by Theorem~\ref{thm:bcp}\ref{it:bcp-phase} that $d_X(w',v') \leq \varepsilon$ for some phase vertex $v'$ of $\widehat{R}$. In particular, there is a path in $\Gamma(N)$ from $w'$ to $v'$ consisting of $\leq 2\varepsilon$ free edges, and so we have $d_{\Gamma(N)}(w',v') \leq 2\varepsilon$. By construction, every vertex of $\widehat{R}$ is distance $\leq 2$ away (in $\Gamma(N)$) from a vertex of $R$, and so $d_{\Gamma(N)}(v',v) \leq 2$ for some $v \in R$. Therefore,
\[
d_{\Gamma(N)}(w,v) \leq d_{\Gamma(N)}(w,w')+d_{\Gamma(N)}(w',v')+d_{\Gamma(N)}(v',v) \leq \left( \frac{\overline\varepsilon}{2}+3 \right) + 2\varepsilon + 2 \leq \beta,
\]
as required.

\item[Otherwise] Note that, by construction, any vertex of $P$ that is not a vertex of $P_1 \cdots P_n$ must be distance $\leq 3$ away from either $\widehat{P}_-$ or $\widehat{P}_+$. Since $d_{\Gamma(N)}(w,\widehat{P}_\pm) > \frac{\overline\varepsilon}{2}+3 \geq 3$, it follows that $w$ is a vertex of $P_1 \cdots P_n$, and so $w \in P_i$ for some $i$. By our construction (see Definition~\ref{defn:derived}), $(P_i)_-$ and $(P_i)_+$ are vertices of $\widehat{P}$, and either $P_i$ is a path of length $\leq 2$ or all non-endpoint vertices of $P_i$ are $H_j$-internal for some (fixed) $j$. But $|P_i| \geq d_{\Gamma(N)}((P_i)_-,w)+d_{\Gamma(N)}(w,(P_i)_+) > 2(\frac{\overline\varepsilon}{2}+3) > 4$, and so we must have $P_i = e P_i' \overline{e'}$, where $e,e'$ are connecting edges and $P_i'$ is a path with $|P_i'| > 2$ lying in the $g$-copy $\Gamma_0$ of $\Gamma_{j,N}$ for some $g \in G$. In particular, without loss of generality we have $e = \{ g, (H_jg,u) \}$ and $e' = \{ hg,(H_jg,u') \}$ for some $h \in H_j$ and $u,u' \in V(\Gamma_j)$.

We now claim that $d_{\Gamma(N)}((P_i')_-,(P_i')_+) > \overline\varepsilon$. Indeed, let $e_0$ and $e_0'$ be the first and the last edges of $P_i'$, so that $P_i' = e_0 P_i'' e_0'$ for some path $P_i'' \subseteq \Gamma_0$; see Figure~\ref{sfig:stint-wfar}. By Definition~\ref{defn:derived}, $P_i''$ is a subpath of $P$ and so a geodesic in $\Gamma(N)$. Moreover, since $d_{\Gamma(N)}(w,(P_i)_\pm) > \frac{\overline\varepsilon}{2}+3 > 1$ we have $w \neq (P_i')_\pm$, and so $w \in P_i''$; furthermore, $d_{\Gamma(N)}(w,(P_i'')_\pm) > \frac{\overline\varepsilon}{2}+1$, implying that
\begin{align*}
d_{\Gamma(N)}\left((P_i')_-,(P_i')_+\right) &\geq d_{\Gamma(N)}\left((P_i'')_-,(P_i'')_+\right)-d_{\Gamma(N)}\left((P_i')_-,(P_i'')_-\right)-d_{\Gamma(N)}\left((P_i')_+,(P_i'')_+\right) \\ &= |P_i''|-2 = d_{\Gamma(N)}\left((P_i'')_-,w\right) + d_{\Gamma(N)}\left(w,(P_i'')_+\right) - 2 \\ &> 2\left( \frac{\overline\varepsilon}{2}+1 \right) - 2 = \overline\varepsilon,
\end{align*}
as claimed.

Since $(P_i')_- = (H_jg,u)$ and $(P_i')_+ = (H_jg,u')$, it follows that $d_{\Gamma_{j,N}}(u,u') > \overline\varepsilon$. Therefore, by the choice of $\overline\varepsilon$, it follows that $d_X\left( (\widehat{P}_i)_-,(\widehat{P}_i)_+ \right) = d_X(g,hg) = d_X(1,h) > \varepsilon$. Therefore, by Theorem~\ref{thm:bcp}\ref{it:bcp-conn}, $\widehat{P}_i$ is connected to a component of $\widehat{R}$, that is, to $\widehat{R}_{i'}$ for some $i' \in \{ 1,\ldots,\ell \}$. Furthermore, Theorem~\ref{thm:bcp}\ref{it:bcp-close} implies that $d_X\left( (\widehat{P}_i)_-,(\widehat{R}_{i'})_- \right) \leq \varepsilon$ and $d_X\left( (\widehat{P}_i)_+,(\widehat{R}_{i'})_+ \right) \leq \varepsilon$. By the construction (Definition~\ref{defn:derived}), we have $R_{i'} = f R_{i'}' \overline{f'}$, where $f = \{ h_0g, (H_jg,t) \}$ and $f' = \{ h_1g, (H_jg,t') \}$ are connecting edges (here $h_0,h_1 \in H_j$ and $t,t' \in V(\Gamma_j)$), and $R_{i'}' \subseteq \Gamma_0$; see Figure~\ref{sfig:stint-wfar}. It then follows that $d_X(1,h_0) = d_X(g,h_0g) \leq \varepsilon$ and so, by the choice of $\overline\varepsilon$, we have $d_{\Gamma_{j,N}}(u,t) \leq \overline\varepsilon$; consequently, $d_{\Gamma_0}\left((P_i')_-,(R_{i'}')_-\right) = d_{\Gamma_{j,N}}(u,t) \leq \overline\varepsilon$. Similarly, we have $d_{\Gamma_0}\left((P_i')_+,(R_{i'}')_+\right) \leq \overline\varepsilon$.

Now by the construction, as the vertex $(R_{i'}')_-$ belongs to $\Gamma_0$, it is distance $\leq 1$ away from some vertex $s \in \Gamma_0 \cap R$; similarly, $d_{\Gamma_0}\left((R_{i'}')_+,s'\right) \leq 1$ for some $s' \in \Gamma_0 \cap R$. Furthermore, there exists subpath $S$ of $R$ such that $S \subseteq \Gamma_0$, $S_- = s$ and $S_+ = s'$. Note that we have
\begin{align*}
d_{\Gamma_0}\left( (P_i'')_-,S_- \right) &\leq d_{\Gamma_0}\left( (P_i'')_-,(P_i')_- \right) + d_{\Gamma_0}\left( (P_i')_-,(R_{i'}')_- \right) + d_{\Gamma_0}\left( (R_{i'}')_-,S_- \right) \\ &\leq 1 + \overline\varepsilon + 1 = \overline\varepsilon+2,
\end{align*}
and similarly $d_{\Gamma_0}\left( (P_i'')_+,S_+ \right) \leq \overline\varepsilon+2$. Hence there exist paths $T,T' \subseteq \Gamma_0$ such that $T_- = (P_i'')_-$, $T_+ = S_-$, $T'_- = (P_i'')_+$, $T'_+ = S_+$, and $|T|,|T'| \leq \overline\varepsilon+2$.

Finally, note that the graph $\Gamma_0$ is isomorphic to $\Gamma_{j,N}$. By Lemma~\ref{lem:GammajN-stint} and the choice of $\beta_0$, it follows that $\Gamma_0$ has $(3\beta_0+1)$-stable intervals. Since $P_i'' \subseteq \Gamma_0$ is a geodesic and $w \in P_i''$, we may use this fact $|T|+|T'|$ times, `moving' the endpoints of a geodesic joining a vertex of $T$ to a vertex of $T'$ along the paths $T$ and $T'$ (see Figure~\ref{sfig:stint-ttp}), to find a geodesic $S' \subseteq \Gamma_0$ and a vertex $v \in S'$ such that $S'_- = S_-$, $S'_+ = S_+$ and
\[
d_{\Gamma_0}(w,v) \leq (3\beta_0+1)(|T|+|T'|) \leq (3\beta_0+1)(2\overline\varepsilon+4) \leq \beta.
\]
Since $S \subseteq \Gamma_0$ is a geodesic in $\Gamma(N)$ (and so in $\Gamma_0$), it follows that $|S'| = |S|$. If we express the geodesic $R \subseteq \Gamma(N)$ as $R = S_0 S S_1$ for some paths $S_0,S_1 \subseteq \Gamma(N)$, it then follows that the path $Q = S_0 S' S_1 \subseteq \Gamma(N)$ is a geodesic as well, and we have $Q_- = R_- = w'$ and $Q_+ = R_+ = P_+$. Since $v \in Q$, this establishes the result. \qedhere

\end{description}

\end{proof}

\subsection{The coarse Helly property} \label{ssec:coarse-helly}

The remainder of this section is dedicated to a proof of the following result which immediately implies Theorem~\ref{thm:coarsehelly}. Using it, at the end of the subsection we prove Theorem~\ref{thm:helly}.

\begin{prop} \label{prop:helly-coarse}
If each $\Gamma_j$ is coarsely Helly, then so is $\Gamma(N)$.
\end{prop}

In order to prove the Proposition, we start by taking a collection $\{ B_{\rho_i}(x_i;\Gamma(N))\mid i \in \mathcal{I} \}$ of pairwise intersecting balls, picking a basepoint $y \in \Gamma(N)$, and choosing an index $I \in \mathcal{I}$ maximising $d_{\Gamma(N)}(x_I,y)-\rho_I$. We consider geodesic triangles with vertices $y$, $x_I$ and $x_i$ (where $i \in \mathcal{I} \setminus \{I\}$): by Theorem~\ref{thm:g->qg}, the edges of such triangles correspond to quasi-geodesics in $\CayG$ that do not backtrack. The strategy of the proof depends on whether the point $z$, located on a geodesic $R \subseteq \Gamma(N)$ from $y$ to $x_I$ and such that $d_{\Gamma(N)}(x_I,z) = \rho_I$, is `close' to a free vertex of $R$. If it is, we may use Proposition~\ref{prop:BCPtriangles} and maximality of $d_{\Gamma(N)}(x_I,z)-\rho_I$ to find a universal upper bound for the numbers $d_{\Gamma(N)}(x_i,z)-\rho_i$: see Lemma~\ref{lem:chelly-zout}. Otherwise, $z$ is in a copy $\Gamma_0$ of some $\Gamma_{j,N}$, and we complete the proof by using Proposition~\ref{prop:BCPtriangles} together with the coarse Helly property in $\Gamma_0$ for a collection of balls centered at vertices where geodesics from $x_i$ to $x_I$ `enter' $\Gamma_0$ (where $i \in \mathcal{I} \setminus \{I\}$), along with a ball centered at the vertex where $R$ `leaves' $\Gamma_0$: see Lemma~\ref{lem:chelly-zin}.

We start by making the following elementary observation.

\begin{lem} \label{lem:GammajN-chelly}
Let $\xi \geq 0$. If $\Gamma_j$ is $\xi$-coarsely Helly, then $\Gamma_{j,N}$ is $\lceil \xi/N \rceil$-coarsely Helly.
\end{lem}

\begin{proof}
Note that $B_\rho(w;\Gamma_{j,N}) = B_{\rho N}(w;\Gamma_j)$ for each $w \in V(\Gamma_j)$ and $\rho \geq 0$. Thus, for any collection $\{ B_{\rho_i}(w_i;\Gamma_{j,N}) \mid i \in \mathcal{I} \}$ of pairwise intersecting balls in $\Gamma_{j,N}$, the collection $\{ B_{\rho_i N}(w_i;\Gamma_j) \mid i \in \mathcal{I} \}$ is a collection of pairwise intersecting balls in $\Gamma_j$. By the $\xi$-coarse Helly property, it follows that there exists $v \in V(\Gamma_j)$ with $d_{\Gamma_j}(v,w_i) \leq \rho_i N + \xi$ for each $i \in \mathcal{I}$. We thus have $d_{\Gamma_{j,N}}(v,w_i) \leq \left\lceil (\rho_i N + \xi) / N \right\rceil = \rho_i + \lceil \xi/N \rceil$ for each $i \in \mathcal{I}$, proving the $\lceil \xi/N \rceil$-coarse Helly property for the collection of balls in $\Gamma_{j,N}$.
\end{proof}

Now as before, let $\lambda \geq 1$ and $c \geq 0$ be constants such that for any geodesic $P \subseteq \Gamma(N)$, the path $\widehat{P} \subseteq \CayG$ is a $2$-local geodesic $(\lambda,c)$-quasi-geodesic that does not backtrack: such constants exist by Theorem~\ref{thm:g->qg}. Let $\mu = \mu(\lambda,c)$ be the constant given by Proposition~\ref{prop:BCPtriangles}. Since $X$ is finite and $\epsilon|_{\widetilde{H}_j}$ is injective for each $j$, there are finitely many elements $h \in \mathcal{H}$ such that $|\epsilon(h)|_X \leq \mu$. We may therefore choose a constant $\overline\mu \geq 0$ such that $d_{\Gamma_{j,N}}(v_j,v_j \cdot h) \leq \overline\mu$ whenever $h \in \widetilde{H}_j$ satisfies $|\epsilon(h)|_X \leq \mu$.

For the remainder of this section, we assume that $\Gamma_1,\ldots,\Gamma_m$ are coarsely Helly. Let $\xi_0 \geq 0$ be such that $\Gamma_1,\ldots,\Gamma_m$ are all $\xi_0$-coarsely Helly, and set
\begin{equation} \label{eq:defxi}
\xi = \max \left\{ \frac{\overline\mu}{2}+4\mu+5, \frac{5\overline\mu}{2}+\left\lceil \frac{\xi_0}{N} \right\rceil+4 \right\}.
\end{equation}
We aim to show that every collection of pairwise intersecting balls in $\Gamma(N)$ satisfies the $\xi$-coarse Helly property. This will establish Proposition~\ref{prop:helly-coarse}.

Thus, let $\mathcal{B}' = \left\{ B_{\rho_i'}(x_i';\Gamma(N)) \:\middle|\: i \in \mathcal{I} \right\}$ be a collection of pairwise intersecting balls in $\Gamma(N)$. By the choice of $N$ (see Assumption~\ref{ass:Nbigger}\ref{it:Nbigg-quot}), any vertex of $\Gamma(N)$ is distance $\leq 2$ away from a free vertex, and so for each $i \in \mathcal{I}$ we may pick a free vertex $x_i \in V(\Gamma(N))$ such that $d_{\Gamma(N)}(x_i,x_i') \leq 2$. By letting $\rho_i = \rho_i'+2$ (for all $i \in \mathcal{I}$), we see that the collection $\mathcal{B} = \left\{ B_{\rho_i}(x_i;\Gamma(N)) \:\middle|\: i \in \mathcal{I} \right\}$ of balls in $\Gamma(N)$ satisfies $B_{\rho_i}(x_i) \supseteq B_{\rho_i'}(x_i')$ for all $i \in \mathcal{I}$, and hence, as the balls in $\mathcal{B}'$ have pairwise non-empty intersections, so do the balls in $\mathcal{B}$. Moreover, since $B_{\rho_i+\xi-4}(x_i) = B_{\rho_i'+\xi-2}(x_i) \subseteq B_{\rho_i'+\xi}(x_i')$ for each $i$, in order to show that $\mathcal{B}'$ satisfies the $\xi$-coarse Helly property, it is enough to show that $\mathcal{B}$ satisfies the $(\xi-4)$-coarse Helly property. We will show that the latter is indeed the case.

Let $y \in V(\Gamma(N))$ be an arbitrary free vertex ($y = 1_G$, say), and consider the set $\mathcal{D} = \{ d_{\Gamma(N)}(y,x_i)-\rho_i \mid i \in \mathcal{I} \} \subseteq \mathbb{Z}$. Note that $\mathcal{D}$ is bounded from above: indeed, for any (fixed) $j \in \mathcal{I}$ and any $i \in \mathcal{I}$ we have
\[
d_{\Gamma(N)}(y,x_i) \leq d_{\Gamma(N)}(y,x_j)+d_{\Gamma(N)}(x_j,x_i) \leq d_{\Gamma(N)}(y,x_j)+\rho_j+\rho_i
\]
since $B_{\rho_j}(x_j) \cap B_{\rho_i}(x_i) \neq \varnothing$, and hence $D \leq d_{\Gamma(N)}(y,x_j)+\rho_j$ for any $D \in \mathcal{D}$. We may therefore fix an index $I \in \mathcal{I}$ such that $d_{\Gamma(N)}(y,x_I)-\rho_I = \max \mathcal{D}$.

We pick a geodesic $R \subseteq \Gamma(N)$ and, for each $i \in \mathcal{I}$, geodesics $P_i,Q_i \subseteq \Gamma(N)$, such that $(P_i)_+ = (Q_i)_- = x_i$, $(Q_i)_+ = R_- = y$ and $R_+ = (P_i)_- = x_I$, so that $P_iQ_iR$ is a geodesic triangle in $\Gamma(N)$ with vertices $x_I$, $x_i$ and $y$; note that we do not exclude the `degenerate' cases when $x_i \in \{y,x_I\}$. If $|R| < \rho_I$, it then follows that $D < 0$ for all $D \in \mathcal{D}$ and so $y \in \bigcap_{B \in \mathcal{B}} B$, proving the Helly property (and so the $(\xi-4)$-coarse Helly property) for $\mathcal{B}$. We may therefore without loss of generality assume that $|R| \geq \rho_I$. Let $z \in R$ be the vertex such that $d_{\Gamma(N)}(x_I,z) = \rho_I$; see Figure~\ref{fig:coarse-general}.

\begin{figure}[ht]
\begin{tikzpicture}[decoration={markings,mark=at position 0.5 with {\arrow{>}}}]
\draw [very thick,postaction=decorate] (0,0) to node [midway,below] {$R$} (-8,0);
\draw [very thick,postaction=decorate] (-6,3) to[bend right=20] node [midway,above] {$Q_i$} node [near start] (vr) {} (0,0);
\draw [very thick,postaction=decorate] (-8,0) to[bend right] node [midway,left] {$P_i$} node [near end] (vq) {} (-6,3);
\draw [decorate,decoration={brace,amplitude=5pt}] (-5.5,-0.1) -- (-8,-0.1) node [midway,below,yshift=-3pt] {$\rho_I$};
\draw [green!70!black,thick,dashed] (-5,0) to[bend left=10] (vr.center);
\draw [green!70!black,thick,dashed] (-5,0) to[bend right=10] (vq.center);
\fill (0,0) circle (2pt) node [right] {$y$};
\fill (-5.5,0) circle (2pt) node [above] {$z$};
\fill (-8,0) circle (2pt) node [left] {$x_I$};
\fill (-6,3) circle (2pt) node [left] {$x_i$};
\fill [red] (-5,0) circle (2pt) node [below] {$z'$};
\fill [blue] (vq) circle (2pt) node [left] {$v_i(\in P_i)$};
\fill [blue] (vr) circle (2pt) node [right,yshift=7pt,xshift=-5pt] {$v_i(\in Q_i)$};
\end{tikzpicture}
\caption{The proof of Proposition~\ref{prop:helly-coarse}: general setup and Lemma~\ref{lem:chelly-zout}.}
\label{fig:coarse-general}
\end{figure}
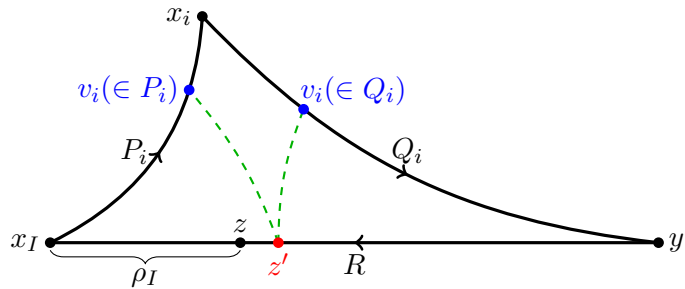

Our proof of the $(\xi-4)$-coarse Helly property for $\mathcal{B}$ splits into two cases, based on whether or not $z$ is `close' to a vertex of $\widehat{R}$.

\begin{lem} \label{lem:chelly-zout}
If $d_{\Gamma(N)}(z,z') \leq \frac{\overline\mu}{2}+1$ for some vertex $z' \in \widehat{R}$, then $\mathcal{B}$ satisfies the $\left( \frac{\overline\mu}{2}+4\mu+1 \right)$-coarse Helly property.
\end{lem}

\begin{proof}
Fix $i \in \mathcal{I}$, and let $P = P_i$ and $Q = Q_i$. By Theorem~\ref{thm:g->qg}, since $P,Q,R \subseteq \Gamma(N)$ are geodesics, $\widehat{P},\widehat{Q},\widehat{R} \subseteq \CayG$ are $2$-local geodesic $(\lambda,c)$-quasi-geodesics that do not backtrack; in particular, every vertex of $\widehat{R}$ is phase. Moreover, as the endpoints of $P$, $Q$ and $R$ are free vertices, we have $\widehat{P}_\pm = P_\pm$, $\widehat{Q}_\pm = Q_\pm$ and $\widehat{R}_\pm = R_\pm$, and so $\widehat{P}\widehat{Q}\widehat{R}$ is a triangle in $\CayG$.

It then follows from Proposition~\ref{prop:BCPtriangles}\ref{it:bcpt-phase} that there exists a vertex $v = v_i$ of $\widehat{P}$ or of $\widehat{Q}$ such that $d_X(z',v) \leq \mu$. Since vertices of $\widehat{P}$ (respectively, $\widehat{Q}$) are precisely the free vertices of $P$ (respectively, $Q$), it follows that either $v \in P$ or $v \in Q$; see Figure~\ref{fig:coarse-general}. We claim that $d_{\Gamma(N)}(x_i,v) \leq \rho_i+d_{\Gamma(N)}(z,v)$.

If $v \in P$, then note that $d_{\Gamma(N)}(x_I,v) \geq d_{\Gamma(N)}(x_I,z)-d_{\Gamma(N)}(z,v) = \rho_I - d_{\Gamma(N)}(z,v)$. Since $B_{\rho_I}(x_I) \cap B_{\rho_i}(x_i) \neq \varnothing$, it follows that $d_{\Gamma(N)}(x_i,x_I) \leq \rho_i+\rho_I$, and so, as $P$ is a geodesic,
\[
d_{\Gamma(N)}(x_i,v) = d_{\Gamma(N)}(x_i,x_I) - d_{\Gamma(N)}(x_I,v) \leq (\rho_i+\rho_I) - (\rho_I-d_{\Gamma(N)}(z,v)) = \rho_i + d_{\Gamma(N)}(z,v),
\]
as claimed.

On the other hand, if $v \in Q$, then note that $d_{\Gamma(N)}(y,z) = d_{\Gamma(N)}(y,x_I)-\rho_I \geq d_{\Gamma(N)}(y,x_i)-\rho_i$ by the choice of $I$. It follows that
\[
d_{\Gamma(N)}(y,v) \geq d_{\Gamma(N)}(y,z)-d_{\Gamma(N)}(z,v) \geq d_{\Gamma(N)}(y,x_i) - (\rho_i+d_{\Gamma(N)}(z,v)),
\]
and so, as $Q$ is a geodesic, we have $d_{\Gamma(N)}(x_i,v) = d_{\Gamma(N)}(y,x_i)-d_{\Gamma(N)}(y,v) \leq \rho_i+d_{\Gamma(N)}(z,v)$, as claimed.

Therefore, we have $d_{\Gamma(N)}(x_i,v) \leq \rho_i+d_{\Gamma(N)}(z,v)$ in either case, and hence
\begin{equation} \label{eq:dxizp}
d_{\Gamma(N)}(x_i,z') \leq d_{\Gamma(N)}(x_i,v) + d_{\Gamma(N)}(z',v) \leq \rho_i+d_{\Gamma(N)}(z,v) + d_{\Gamma(N)}(z',v).
\end{equation}
Now since $d_X(z',v) \leq \mu$, there is a path in $\Gamma(N)$ from $z'$ to $v$ consisting of $\leq 2\mu$ free edges, and so $d_{\Gamma(N)}(z',v) \leq 2\mu$. It follows that $d_{\Gamma(N)}(z,v) \leq d_{\Gamma(N)}(z,z') + d_{\Gamma(N)}(z',v) \leq \frac{\overline\mu}{2}+1+2\mu$, and so \eqref{eq:dxizp} implies that $d_{\Gamma(N)}(x_i,z') \leq \rho_i + \frac{\overline\mu}{2} + 4\mu + 1$.

Thus the intersection $\bigcap_{i \in \mathcal{I}} B_{\rho_i + \frac{\overline\mu}{2} + 4\mu + 1}(x_i;\Gamma(N))$ contains the vertex $z'$ and so is non-empty. Therefore, $\mathcal{B}$ satisfies the $\left( \frac{\overline\mu}{2}+4\mu+1 \right)$-coarse Helly property, as required.
\end{proof}

\begin{lem} \label{lem:chelly-zin}
If $d_{\Gamma(N)}(z,z') > \frac{\overline\mu}{2}+1$ for all $z' \in \widehat{R}$, then $\mathcal{B}$ satisfies the $\left( \frac{5\overline\mu}{2} + \left\lceil \frac{\xi_0}{N} \right\rceil \right)$-coarse Helly property.
\end{lem}

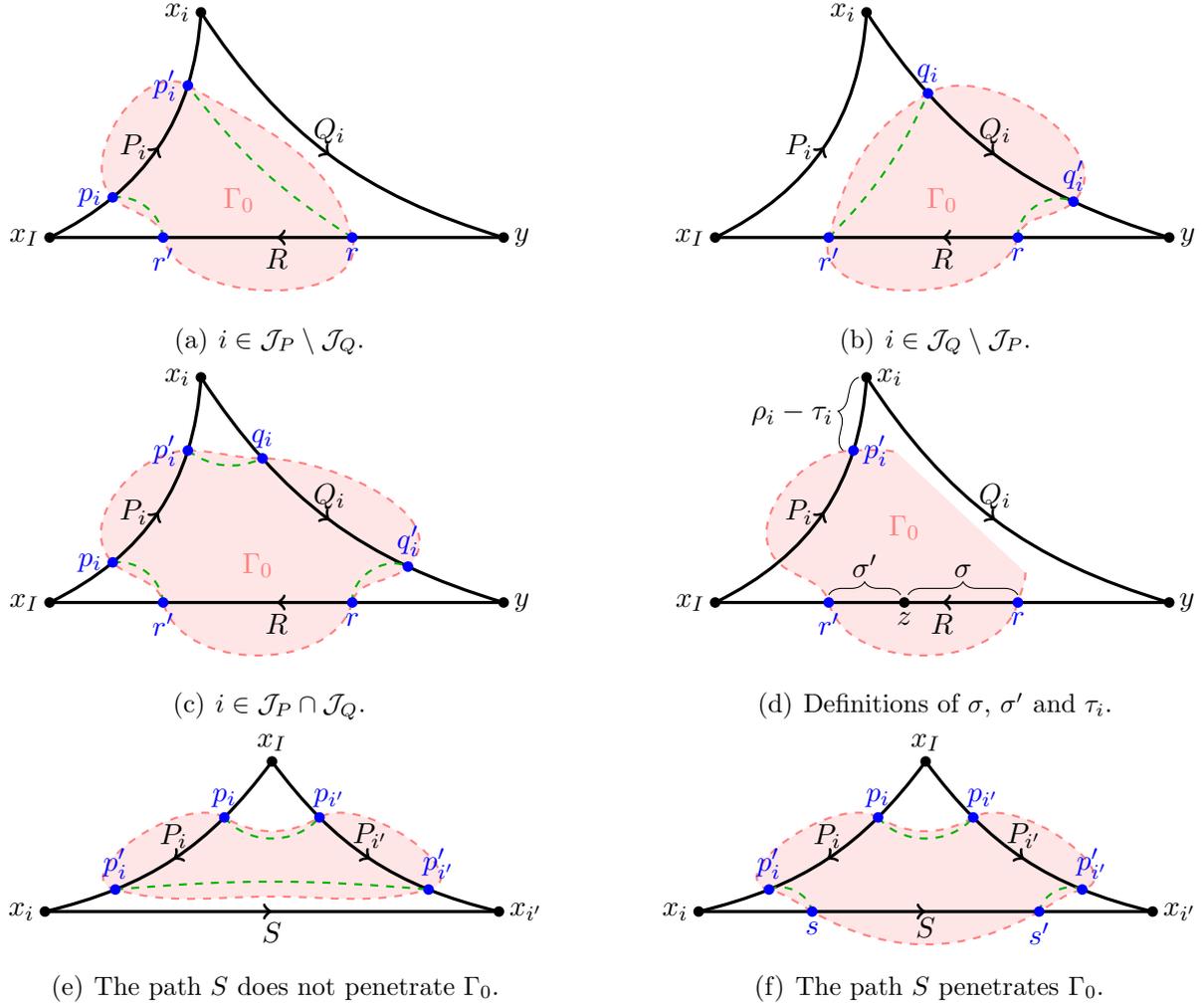
\begin{figure}[ht]
\begin{subfigure}[t]{0.45\textwidth}
\centering
\begin{tikzpicture}[thick,decoration={markings,mark=at position 0.5 with {\arrow{>}}}]
\draw [white] (-6,0) to[bend right] node [near start] (pi) {} node [near end] (pip) {} (-4,3);
\filldraw [dashed,red!50,fill=red!10] (-2,0) to [out=-75,in=-75] (-4.5,0) to [out=105,in=-45] (pi.center) to [out=135,in=150] (pip.center) to [out=-30,in=105] cycle;
\draw [very thick,postaction=decorate] (0,0) to node [midway,below] {$R$} (-6,0);
\draw [very thick,postaction=decorate] (-4,3) to[bend right=20] node [midway,above] {$Q_i$} (0,0);
\draw [very thick,postaction=decorate] (-6,0) to[bend right] node [midway,left] {$P_i$} (-4,3);
\fill (0,0) circle (2pt) node [right] {$y$};
\fill (-6,0) circle (2pt) node [left] {$x_I$};
\fill (-4,3) circle (2pt) node [left] {$x_i$};
\draw [green!70!black,dashed] (-4.5,0) to[bend right=45] (pi.center);
\draw [green!70!black,dashed] (-2,0) to[bend left=10] (pip.center);
\fill [blue] (-2,0) circle (2pt) node [below] {$r$};
\fill [blue] (-4.5,0) circle (2pt) node [below] {$r'$};
\fill [blue] (pi) circle (2pt) node [left] {$p_i$};
\fill [blue] (pip) circle (2pt) node [left] {$p_i'$};
\node [red!50] at (-3.5,0.5) {$\Gamma_0$};
\end{tikzpicture}
\caption{$i \in \mathcal{J}_P \setminus \mathcal{J}_Q$.}
\label{sfig:czin-jp}
\end{subfigure}\hfill%
\begin{subfigure}[t]{0.45\textwidth}
\centering
\begin{tikzpicture}[thick,decoration={markings,mark=at position 0.5 with {\arrow{>}}}]
\draw [white] (-4,3) to[bend right=20] node [near start] (qi) {} node [near end] (qip) {} (0,0);
\filldraw [dashed,red!50,fill=red!10] (-2,0) to [out=-105,in=-105] (-4.5,0) to [out=75,in=-150] (qi.center) to [out=30,in=45] (qip.center) to [out=-135,in=75] cycle;
\draw [very thick,postaction=decorate] (0,0) to node [midway,below] {$R$} (-6,0);
\draw [very thick,postaction=decorate] (-4,3) to[bend right=20] node [midway,above] {$Q_i$} (0,0);
\draw [very thick,postaction=decorate] (-6,0) to[bend right] node [midway,left] {$P_i$} (-4,3);
\fill (0,0) circle (2pt) node [right] {$y$};
\fill (-6,0) circle (2pt) node [left] {$x_I$};
\fill (-4,3) circle (2pt) node [left] {$x_i$};
\draw [green!70!black,dashed] (-4.5,0) to[bend right=10] (qi.center);
\draw [green!70!black,dashed] (-2,0) to[bend left=60] (qip.center);
\fill [blue] (-2,0) circle (2pt) node [below] {$r$};
\fill [blue] (-4.5,0) circle (2pt) node [below] {$r'$};
\fill [blue] (qi) circle (2pt) node [above] {$q_i$};
\fill [blue] (qip) circle (2pt) node [above] {$q_i'$};
\node [red!50] at (-3,0.5) {$\Gamma_0$};
\end{tikzpicture}
\caption{$i \in \mathcal{J}_Q \setminus \mathcal{J}_P$.}
\label{sfig:czin-jq}
\end{subfigure}

\begin{subfigure}[t]{0.45\textwidth}
\centering
\begin{tikzpicture}[thick,decoration={markings,mark=at position 0.5 with {\arrow{>}}}]
\draw [white] (-6,0) to[bend right] node [near start] (pi) {} node [near end] (pip) {} (-4,3);
\draw [white] (-4,3) to[bend right=20] node [near start] (qi) {} node [near end] (qip) {} (0,0);
\filldraw [dashed,red!50,fill=red!10] (-2,0) to [out=-105,in=-75] (-4.5,0) to [out=105,in=-45] (pi.center) to [out=135,in=180] (pip.center) to [out=0,in=180] (qi.center) to [out=0,in=45] (qip.center) to [out=-135,in=75] cycle;
\draw [very thick,postaction=decorate] (0,0) to node [midway,below] {$R$} (-6,0);
\draw [very thick,postaction=decorate] (-4,3) to[bend right=20] node [midway,above] {$Q_i$} (0,0);
\draw [very thick,postaction=decorate] (-6,0) to[bend right] node [midway,left] {$P_i$} (-4,3);
\fill (0,0) circle (2pt) node [right] {$y$};
\fill (-6,0) circle (2pt) node [left] {$x_I$};
\fill (-4,3) circle (2pt) node [left] {$x_i$};
\draw [green!70!black,dashed] (-4.5,0) to[bend right=45] (pi.center);
\draw [green!70!black,dashed] (-2,0) to[bend left=60] (qip.center);
\draw [green!70!black,dashed] (pip.center) to[bend right=30] (qi.center);
\fill [blue] (-2,0) circle (2pt) node [below] {$r$};
\fill [blue] (-4.5,0) circle (2pt) node [below] {$r'$};
\fill [blue] (pi) circle (2pt) node [left] {$p_i$};
\fill [blue] (pip) circle (2pt) node [left] {$p_i'$};
\fill [blue] (qi) circle (2pt) node [above] {$q_i$};
\fill [blue] (qip) circle (2pt) node [above] {$q_i'$};
\node [red!50] at (-3.25,0.5) {$\Gamma_0$};
\end{tikzpicture}
\caption{$i \in \mathcal{J}_P \cap \mathcal{J}_Q$.}
\label{sfig:czin-jpq}
\end{subfigure}\hfill%
\begin{subfigure}[t]{0.45\textwidth}
\centering
\begin{tikzpicture}[decoration={markings,mark=at position 0.5 with {\arrow{>}}}]
\draw [white] (-6,0) to[bend right] node [near start] (pi) {} node [near end] (pip) {} (-4,3);
\filldraw [thick,dashed,red!50,fill=red!10] (-1.9,0.4) to [out=-90,in=75] (-2,0) to [out=-105,in=-75] (-4.5,0) to [out=105,in=-45] (pi.center) to [out=135,in=180] (pip.center) to [out=0,in=165] (-3.6,2);
\draw [very thick,postaction=decorate] (0,0) to node [midway,below] {$R$} (-6,0);
\draw [very thick,postaction=decorate] (-4,3) to[bend right=20] node [midway,above] {$Q_i$} (0,0);
\draw [very thick,postaction=decorate] (-6,0) to[bend right] node [midway,left] {$P_i$} (-4,3);
\fill (0,0) circle (2pt) node [right] {$y$};
\fill (-6,0) circle (2pt) node [left] {$x_I$};
\fill (-4,3) circle (2pt) node [right] {$x_i$};
\fill (-3.5,0) circle (2pt) node [below] {$z$};
\fill [blue] (-2,0) circle (2pt) node [below] {$r$};
\fill [blue] (-4.5,0) circle (2pt) node [below] {$r'$};
\fill [blue] (pip) circle (2pt) node [right] {$p_i'$};
\draw [decorate,decoration={brace,amplitude=5pt},yshift=0.1cm] (-4.5,0) -- (-3.55,0) node [midway,above,yshift=3pt] {$\sigma'$};
\draw [decorate,decoration={brace,amplitude=5pt},yshift=0.1cm] (-3.45,0) -- (-2,0) node [midway,above,yshift=3pt] {$\sigma$};
\draw [decorate,decoration={brace,amplitude=5pt}] (pip.west) -- (-4.1,3) node [midway,left,xshift=-3pt] {$\rho_i-\tau_i$};
\node [red!50] at (-3.5,1) {$\Gamma_0$};
\end{tikzpicture}
\caption{Definitions of $\sigma$, $\sigma'$ and $\tau_i$.}
\label{sfig:czin-dist}
\end{subfigure}

\begin{subfigure}[t]{0.46\textwidth}
\centering
\begin{tikzpicture}[thick,decoration={markings,mark=at position 0.5 with {\arrow{>}}}]
\draw [white,thick] (1.5,0) to [out=-30,in=-150] (4.5,0);
\draw [white] (3,2) to[bend left=20] node [near start] (pi) {} node [near end] (pip) {} (0,0);
\draw [white] (3,2) to[bend right=20] node [near start] (ppi) {} node [near end] (ppip) {} (6,0);
\filldraw [dashed,red!50,fill=red!10] (pi.center) to [out=150,in=165] (pip.center) to [out=-15,in=180] (3,0.2) to[out=0,in=-165] (ppip.center) to [out=15,in=30] (ppi.center) to [out=-150,in=-30] cycle;
\draw [very thick,postaction=decorate] (0,0) to node [midway,below] {$S$} (6,0);
\draw [very thick,postaction=decorate] (3,2) to[bend left=20] node [midway,above] {$P_i$} (0,0);
\draw [very thick,postaction=decorate] (3,2) to[bend right=20] node [midway,above] {$P_{i'}$} (6,0);
\draw [green!70!black,dashed] (pi.center) to[bend right=50] (ppi.center);
\draw [green!70!black,dashed] (pip.center) to[bend left=5] (ppip.center);
\fill (0,0) circle (2pt) node [left] {$x_i$};
\fill (6,0) circle (2pt) node [right] {$x_{i'}$};
\fill (3,2) circle (2pt) node [above] {$x_I$};
\fill [blue] (pi) circle (2pt) node [above] {$p_i$};
\fill [blue] (pip) circle (2pt) node [above] {$p_i'$};
\fill [blue] (ppi) circle (2pt) node [above] {$\ \ p_{i'}$};
\fill [blue] (ppip) circle (2pt) node [above] {$\ \ p_{i'}'$};
\end{tikzpicture}
\caption{The path $S$ does not penetrate $\Gamma_0$.}
\label{sfig:czin-Snpen}
\end{subfigure}\hfill%
\begin{subfigure}[t]{0.46\textwidth}
\centering
\begin{tikzpicture}[thick,decoration={markings,mark=at position 0.5 with {\arrow{>}}}]
\draw [white] (3,2) to[bend left=20] node [near start] (pi) {} node [near end] (pip) {} (0,0);
\draw [white] (3,2) to[bend right=20] node [near start] (ppi) {} node [near end] (ppip) {} (6,0);
\filldraw [dashed,red!50,fill=red!10] (pi.center) to [out=150,in=165] (pip.center) to [out=-15,in=150] (1.5,0) to [out=-30,in=-150] (4.5,0) to[out=30,in=-165] (ppip.center) to [out=15,in=30] (ppi.center) to [out=-150,in=-30] cycle;
\draw [very thick,postaction=decorate] (0,0) to node [midway,below,yshift=2pt] {$S$} (6,0);
\draw [very thick,postaction=decorate] (3,2) to[bend left=20] node [midway,above] {$P_i$} (0,0);
\draw [very thick,postaction=decorate] (3,2) to[bend right=20] node [midway,above] {$P_{i'}$} (6,0);
\draw [green!70!black,dashed] (pi.center) to[bend right=50] (ppi.center);
\draw [green!70!black,dashed] (pip.center) to[bend left=50] (1.5,0);
\draw [green!70!black,dashed] (ppip.center) to[bend right=50] (4.5,0);
\fill (0,0) circle (2pt) node [left] {$x_i$};
\fill (6,0) circle (2pt) node [right] {$x_{i'}$};
\fill (3,2) circle (2pt) node [above] {$x_I$};
\fill [blue] (pi) circle (2pt) node [above] {$p_i$};
\fill [blue] (pip) circle (2pt) node [above] {$p_i'$};
\fill [blue] (ppi) circle (2pt) node [above] {$\ \ p_{i'}$};
\fill [blue] (ppip) circle (2pt) node [above] {$\ \ p_{i'}'$};
\fill [blue] (1.5,0) circle (2pt) node [below] {$s$};
\fill [blue] (4.5,0) circle (2pt) node [below] {$s'$};
\end{tikzpicture}
\caption{The path $S$ penetrates $\Gamma_0$.}
\label{sfig:czin-Spen}
\end{subfigure}
\caption{The proof of Lemma~\ref{lem:chelly-zin}.}
\label{fig:coarse-zin}
\end{figure}

\begin{proof}
Since by construction the vertices of $\widehat{R}$ are precisely the free vertices of $R$, it follows that the vertex $z \in R$ is distance $> \frac{\overline\mu}{2}+1 \geq 1$ away from any free vertex of $R$. It follows that $z$ is $H_j$-internal (for some $j$), and so $z \in \Gamma_0$, where $\Gamma_0$ is the $g$-copy of $\Gamma_{j,N}$ in $\Gamma(N)$ for some $g \in G$. In particular, $R$ penetrates $\Gamma_0$.

Throughout the proof, we adopt the following terminology. Suppose $S \subseteq \Gamma(N)$ is a geodesic that penetrates $\Gamma_0$ such that $S_-$ and $S_+$ are free vertices. By Theorem~\ref{thm:g->qg}, $\widehat{S} \subseteq \CayG$ is then a $2$-local geodesic that does not backtrack. It then follows that $S = S_0 e_0 S' \overline{e_1} S_1$, where $S' = S \cap \Gamma_0$, and $e_0 = \{ h_0g, (H_jg,u_0) \}$, $e_1 = \{ h_1g, (H_jg,u_1) \}$ are connecting edges for some $h_0,h_1 \in H_j$ and $u_0,u_1 \in V(\Gamma_j)$. In this case, we say $S$ \emph{enters} (respectively, \emph{leaves}) $\Gamma_0$ at the vertex $S'_-$ (respectively, $S'_+$). The path $\widehat{S} \subseteq \CayG$ then has an edge $\pathXH{e} \subseteq \widehat{S}$ with $\pathXH{e}_- = h_0g$ and $\pathXH{e}_+ = h_1g$ such that $\pathXH{e}$ is an $H_j$-component of $\widehat{S}$; we say that $\pathXH{e}$ is the $H_j$-component of $\widehat{S}$ \emph{associated} to $\Gamma_0$.

The following observation follows from the choice of $\overline\mu$.
\begin{obs}
Let $S \subseteq \Gamma(N)$ be a geodesic that penetrates $\Gamma_0$ with $S_\pm$ free, let $u_0$ and $u_1$ be the vertices at which $S$ enters and leaves $\Gamma_0$, respectively, and let $\pathXH{e}$ be the $H_j$-component of $\widehat{S}$ associated to $\Gamma_0$. If $d_X(\pathXH{e}_-,\pathXH{e}_+) \leq \mu$, then $d_{\Gamma_0}(u_0,u_1) \leq \overline\mu$ (and therefore $d_{\Gamma(N)}(u_0,u_1) \leq \overline\mu$).
\end{obs}

Let $r$ and $r'$ be the vertices of $\Gamma(N)$ at which $R$ enters and leaves $\Gamma_0$, respectively. Since $r$ and $r'$ are adjacent to free vertices of $R$, it follows from the premise that $d_{\Gamma(N)}(z,r),d_{\Gamma(N)}(z,r') > \frac{\overline\mu}{2}$. Therefore, as $R$ is a geodesic, we have
\begin{equation} \label{eq:dwr}
d_{\Gamma(N)}(r,r') = d_{\Gamma(N)}(r,z) + d_{\Gamma(N)}(z,r') > \overline\mu.
\end{equation}

Now for each $i \in \mathcal{I}$, it follows from Theorem~\ref{thm:g->qg} that $\widehat{P_i}\widehat{Q_i}\widehat{R}$ is a non-backtracking $(\lambda,c)$-quasi-geodesic triangle in $\CayG$. By \eqref{eq:dwr} and the Observation, if $\pathXH{e}$ is the $H_j$-component of $\widehat{R}$ associated to $\Gamma_0$ then $d_X(\pathXH{e}_-,\pathXH{e}_+) > \mu$. Therefore, it follows from Proposition~\ref{prop:BCPtriangles}\ref{it:bcpt-conn} that $\pathXH{e}$ is connected to an $H_j$-component of either $\widehat{P_i}$ or $\widehat{Q_i}$, and by construction such a component must be associated to $\Gamma_0$. In particular, either $P_i$ or $Q_i$ (or both) must penetrate $\Gamma_0$.

Let $\mathcal{J}_P \subseteq \mathcal{I}$ (respectively, $\mathcal{J}_Q \subseteq \mathcal{I}$) be the set of all $i \in \mathcal{I}$ such that $P_i$ (respectively, $Q_i$) penetrates $\Gamma_0$, so that $\mathcal{I} = \mathcal{J}_P \cup \mathcal{J}_Q$. For each $i \in \mathcal{J}_P$, let $p_i$ and $p_i'$ be the vertices of $\Gamma(N)$ at which $P_i$ enters and leaves $\Gamma_0$, respectively; similarly, for $i \in \mathcal{J}_Q$, let $q_i$ and $q_i'$ be the vertices of $\Gamma(N)$ at which $Q_i$ enters and leaves $\Gamma_0$, respectively. By parts \ref{it:bcpt-sim2} and \ref{it:bcpt-sim3} of Proposition~\ref{prop:BCPtriangles} and the choice of $\overline\mu$ (cf the Observation above), it follows that
\begin{enumerate}[label=(\alph*)]
\item \label{it:jp} If $i \in \mathcal{J}_P \setminus \mathcal{J}_Q$, then $d_{\Gamma_0}(r',p_i) \leq \overline\mu$ and $d_{\Gamma_0}(p_i',r) \leq \overline\mu$.
\item \label{it:jq} If $i \in \mathcal{J}_Q \setminus \mathcal{J}_P$, then $d_{\Gamma_0}(r',q_i) \leq \overline\mu$ and $d_{\Gamma_0}(q_i',r) \leq \overline\mu$.
\item \label{it:jpq} If $i \in \mathcal{J}_P \cap \mathcal{J}_Q$, then $d_{\Gamma_0}(r',p_i) \leq \overline\mu$, $d_{\Gamma_0}(p_i',q_i) \leq \overline\mu$ and $d_{\Gamma_0}(q_i',r) \leq \overline\mu$.
\end{enumerate}
This is depicted in parts \subref{sfig:czin-jp}, \subref{sfig:czin-jq} and \subref{sfig:czin-jpq} of Figure~\ref{fig:coarse-zin}.

Now let $\sigma = d_{\Gamma(N)}(z,r)$ and $\sigma' = d_{\Gamma(N)}(z,r')$. For each $i \in \mathcal{J}_P$, let also $\tau_i = \rho_i - d_{\Gamma(N)}(x_i,p_i')$. Consider the following collection of balls in $\Gamma_0$:
\[
\overline{\mathcal{B}} = \{ B_{\sigma'}(r';\Gamma_0) \} \cup \{ B_{\tau_i+\frac{5\overline\mu}{2}}(p_i';\Gamma_0) \mid i \in \mathcal{J}_P \}.
\]
We claim that $\tau_i+2\overline\mu \geq 0$ for each $i \in \mathcal{J}_P$ and that the balls in $\overline{\mathcal{B}}$ have pairwise non-empty intersections.

\begin{description}

\item[$\tau_i + 2\overline\mu \geq 0$ for each $i \in \mathcal{J}_P$] Suppose first that $i \in \mathcal{J}_P \setminus \mathcal{J}_Q$; see Figure~\ref{sfig:czin-jp}. By the point \ref{it:jp} above we then have $d_{\Gamma(N)}(p_i',r) \leq \overline\mu$, and so
\[
d_{\Gamma(N)}(x_I,p_i') \geq d_{\Gamma(N)}(x_I,r) - d_{\Gamma(N)}(p_i',r) \geq \rho_I+\sigma-\overline\mu.
\]
Since $d_{\Gamma(N)}(x_i,x_I) \leq \rho_i+\rho_I$ and since $P_i$ is a geodesic, we thus have
\begin{align*}
\rho_i-\tau_i &= d_{\Gamma(N)}(x_i,p_i') = d_{\Gamma(N)}(x_i,x_I) - d_{\Gamma(N)}(x_I,p_i') \\ &\leq (\rho_i+\rho_I) - (\rho_I+\sigma-\overline\mu) = \rho_i-(\sigma-\overline\mu).
\end{align*}
Therefore, $\tau_i+2\overline\mu \geq \sigma+\overline\mu \geq 0$, as claimed.

Suppose now that $i \in \mathcal{J}_P \cap \mathcal{J}_Q$; see Figure~\ref{sfig:czin-jpq}. It then follows from the point \ref{it:jpq} above that $d_{\Gamma(N)}(p_i',q_i) \leq \overline\mu$ and $d_{\Gamma(N)}(r,q_i') \leq \overline\mu$. Therefore, we have
\begin{equation} \label{eq:dxiqi}
d_{\Gamma(N)}(x_i,q_i) \geq d_{\Gamma(N)}(x_i,p_i') - d_{\Gamma(N)}(p_i',q_i) \geq \rho_i-\tau_i-\overline\mu
\end{equation}
and
\begin{align*}
d_{\Gamma(N)}(y,q_i') &\geq d_{\Gamma(N)}(y,x_I) - d_{\Gamma(N)}(x_I,z) - d_{\Gamma(N)}(z,r) - d_{\Gamma(N)}(r,q_i') \\ &\geq d_{\Gamma(N)}(y,x_I) - \rho_I - \sigma - \overline\mu \geq d_{\Gamma(N)}(y,x_i) - \rho_i - \sigma - \overline\mu,
\end{align*}
where the last inequality follows from the choice of $I$. As $Q_i$ is a geodesic, this implies that
\[
d_{\Gamma(N)}(x_i,q_i') = d_{\Gamma(N)}(y,x_i) - d_{\Gamma(N)}(y,q_i') \leq \rho_i+\sigma+\overline\mu,
\]
and combining this with \eqref{eq:dxiqi} gives
\begin{equation} \label{eq:dqiqi}
d_{\Gamma(N)}(q_i,q_i') = d_{\Gamma(N)}(x_i,q_i') - d_{\Gamma(N)}(x_i,q_i) \leq (\rho_i+\sigma+\overline\mu) - (\rho_i-\tau_i-\overline\mu) = \sigma+\tau_i+2\overline\mu.
\end{equation}

Suppose for contradiction that $\tau_i + 2\overline\mu < 0$. We then get
\[
d_{\Gamma(N)}(x_i,p_i') > \rho_i+2\overline\mu
\]
by the definition of $\tau_i$. Moreover, \eqref{eq:dqiqi} implies that $d_{\Gamma(N)}(q_i,q_i') < \sigma$, and so
\begin{align*}
d_{\Gamma(N)}(x_I,p_i') &\geq d_{\Gamma(N)}(x_I,r) - d_{\Gamma(N)}(r,q_i') - d_{\Gamma(N)}(q_i',q_i) - d_{\Gamma(N)}(q_i,p_i') \\ &> (\rho_I+\sigma) - \overline\mu - \sigma - \overline\mu = \rho_I-2\overline\mu.
\end{align*}
Therefore, as $P_i$ is a geodesic, we get
\[
d_{\Gamma(N)}(x_i,x_I) = d_{\Gamma(N)}(x_i,p_i') + d_{\Gamma(N)}(x_I,p_i') > \rho_i+\rho_I,
\]
contradicting the fact that $B_{\rho_i}(x_i) \cap B_{\rho_I}(x_I) \neq \varnothing$. Thus we must have $\tau_i+2\overline\mu \geq 0$, as claimed.

\item[$B_{\sigma'}(r';\Gamma_0) \cap B_{\tau_i+\frac{5\overline\mu}{2}}(p_i';\Gamma_0) \neq \varnothing$ for each $i \in \mathcal{J}_P$] As shown above, $\tau_i + 2\overline\mu \geq 0$, and so both $\sigma'$ and $\tau_i+\frac{5\overline\mu}{2}$ are non-negative. It therefore suffices to show that $d_{\Gamma_0}(r',p_i') \leq \sigma' + \tau_i + \frac{5\overline\mu}{2}$.

It follows from the points \ref{it:jp} and \ref{it:jpq} above that $d_{\Gamma(N)}(r',p_i) \leq d_{\Gamma_0}(r',p_i) \leq \overline\mu$, and hence
\[
d_{\Gamma(N)}(x_I,p_i) \geq d_{\Gamma(N)}(x_I,z) - d_{\Gamma(N)}(z,r') - d_{\Gamma(N)}(r',p_i) \geq \rho_I-\sigma'-\overline\mu.
\]
Since $P_i$ is a geodesic and since $d_{\Gamma(N)}(x_I,x_i) \leq \rho_I+\rho_i$, we thus have
\begin{align*}
d_{\Gamma_0}(p_i,p_i') &= d_{\Gamma(N)}(p_i,p_i') = d_{\Gamma(N)}(x_I,x_i) - d_{\Gamma(N)}(x_I,p_i) - d_{\Gamma(N)}(x_i,p_i') \\ &\leq (\rho_I+\rho_i) - (\rho_I-\sigma'-\overline\mu) - (\rho_i-\tau_i) = \sigma'+\tau_i+\overline\mu.
\end{align*}
Therefore,
\[
d_{\Gamma_0}(r',p_i') \leq d_{\Gamma_0}(r',p_i)+d_{\Gamma_0}(p_i,p_i') \leq \overline\mu + (\sigma'+\tau_i+\overline\mu) \leq \sigma' + \tau_i + \frac{5\overline\mu}{2},
\]
as required.

\item[$B_{\tau_i+\frac{5\overline\mu}{2}}(p_i';\Gamma_0) \cap B_{\tau_{i'}+\frac{5\overline\mu}{2}}(p_{i'}';\Gamma_0) \neq \varnothing$ for all $i,i' \in \mathcal{J}_P$] As shown above, $\tau_i,\tau_{i'} \geq -2\overline\mu$, and so both $\tau_i+\frac{5\overline\mu}{2}$ and $\tau_{i'}+\frac{5\overline\mu}{2}$ are non-negative. It is thus enough to show that we have $d_{\Gamma_0}(p_i',p_{i'}') \leq \tau_i + \tau_{i'} + 5\overline\mu$. Let $S \subseteq \Gamma(N)$ be a geodesic with $S_- = x_i$ and $S_+ = x_{i'}$. We will apply Proposition~\ref{prop:BCPtriangles} to the triangle $\widehat{P_i} \widehat{S} \overline{\widehat{P_{i'}}} \subseteq \CayG$, which is a non-backtracking $(\lambda,c)$-quasi-geodesic triangle by Theorem~\ref{thm:g->qg}.

Suppose first that $S$ does not penetrate $\Gamma_0$; see Figure~\ref{sfig:czin-Snpen}. It then follows from Proposition~\ref{prop:BCPtriangles}\ref{it:bcpt-sim2} (cf the point \ref{it:jp} above) that $d_{\Gamma_0}(p_i',p_{i'}') \leq \overline\mu$. In particular, since $\tau_i,\tau_{i'} \geq -2\overline\mu$ we have
\[
d_{\Gamma_0}(p_i',p_{i'}') \leq \overline\mu \leq \overline\mu + (\tau_i+2\overline\mu) + (\tau_{i'}+2\overline\mu) = \tau_i + \tau_{i'} + 5\overline\mu,
\]
as required.

Suppose now that $S$ penetrates $\Gamma_0$, and let $s$ and $s'$ be the vertices at which $S$ enters and leaves $\Gamma_0$, respectively; see Figure~\ref{sfig:czin-Spen}. It follows from Proposition~\ref{prop:BCPtriangles}\ref{it:bcpt-sim3} (cf the point \ref{it:jpq} above) that $d_{\Gamma(N)}(s,p_i') \leq d_{\Gamma_0}(s,p_i') \leq \overline\mu$ and $d_{\Gamma(N)}(s',p_{i'}') \leq d_{\Gamma_0}(s',p_{i'}') \leq \overline\mu$. In particular, we have
\[
d_{\Gamma(N)}(x_i,s) \geq d_{\Gamma(N)}(x_i,p_i') - d_{\Gamma(N)}(s,p_i') \geq \rho_i-\tau_i-\overline\mu,
\]
and similarly $d_{\Gamma(N)}(x_{i'},s') \geq \rho_{i'}-\tau_{i'}-\overline\mu$. Since $S$ is a geodesic and since $d_{\Gamma(N)}(x_i,x_{i'}) \leq \rho_i+\rho_{i'}$, it follows that
\begin{align*}
d_{\Gamma_0}(s,s') &= d_{\Gamma(N)}(s,s') = d_{\Gamma(N)}(x_i,x_{i'}) - d_{\Gamma(N)}(x_i,s) - d_{\Gamma(N)}(x_{i'},s') \\ &\leq (\rho_i+\rho_{i'}) - (\rho_i-\tau_i-\overline\mu) - (\rho_{i'}-\tau_{i'}-\overline\mu) = \tau_i+\tau_{i'}+2\overline\mu.
\end{align*}
Therefore,
\begin{align*}
d_{\Gamma_0}(p_i',p_{i'}') &\leq d_{\Gamma_0}(p_i',s) + d_{\Gamma_0}(s,s') + d_{\Gamma_0}(s',p_{i'}') \\ &\leq \overline\mu + (\tau_i+\tau_{i'}+2\overline\mu) + \overline\mu \leq \tau_i+\tau_{i'}+5\overline\mu,
\end{align*}
as required.

\end{description}

We have thus shown that $\overline{\mathcal{B}}$ is a collection of pairwise intersecting balls in $\Gamma_0$. But since $\Gamma_0$ is isomorphic to $\Gamma_{j,N}$, it follows from Lemma~\ref{lem:GammajN-chelly} and the choice of $\xi_0$ that $\Gamma_0$ is $\lceil \xi_0/N \rceil$-coarsely Helly. Therefore, there exists a vertex $\overline{z} \in V(\Gamma_0)$ such that $d_{\Gamma_0}(r',\overline{z}) \leq \sigma'+\lceil \xi_0/N \rceil$ and $d_{\Gamma_0}(p_i',\overline{z}) \leq \tau_i+\frac{5\overline\mu}{2}+\lceil \xi_0/N \rceil$ for all $i \in \mathcal{J}_P$.
We claim that $d_{\Gamma(N)}(x_i,\overline{z}) \leq \rho_i+\frac{5\overline\mu}{2}+\left\lceil \frac{\xi_0}{N} \right\rceil$ for all $i \in \mathcal{I}$: this will establish the $\left( \frac{5\overline\mu}{2} + \left\lceil \frac{\xi_0}{N} \right\rceil \right)$-coarse Helly property for $\mathcal{B}$.

Suppose first that $i \in \mathcal{J}_P$. We then have
\[
d_{\Gamma(N)}(x_i,\overline{z}) \leq d_{\Gamma(N)}(x_i,p_i') + d_{\Gamma(N)}(p_i',\overline{z}) \leq (\rho_i-\tau_i) + \left( \tau_i+\frac{5\overline\mu}{2} + \left\lceil \frac{\xi_0}{N} \right\rceil \right) = \rho_i+\frac{5\overline\mu}{2}+\left\lceil \frac{\xi_0}{N} \right\rceil,
\]
as claimed.

Suppose now that $i \notin \mathcal{J}_P$; since $\mathcal{I} = \mathcal{J}_P \cup \mathcal{J}_Q$, it follows that $i \in \mathcal{J}_Q$ (see Figure~\ref{sfig:czin-jq}). By the point \ref{it:jq} above, we have $d_{\Gamma(N)}(r',q_i) \leq \overline\mu$. Therefore,
\begin{align*}
d_{\Gamma(N)}(y,q_i) &\geq d_{\Gamma(N)}(y,x_I) - d_{\Gamma(N)}(x_I,r') - d_{\Gamma(N)}(r',q_i) \geq d_{\Gamma(N)}(y,x_I) - \rho_I + \sigma' - \overline\mu \\ &\geq d_{\Gamma(N)}(y,x_i) - \rho_i + \sigma' - \overline\mu,
\end{align*} 
where the last inequality follows by the choice of $I$. As $Q_i$ is a geodesic, this implies that
\[
d_{\Gamma(N)}(x_i,q_i) = d_{\Gamma(N)}(y,x_i)-d_{\Gamma(N)}(y,q_i) \leq \rho_i-\sigma'+\overline\mu,
\]
and hence
\begin{align*}
d_{\Gamma(N)}(x_i,\overline{z}) &\leq d_{\Gamma(N)}(x_i,q_i) + d_{\Gamma(N)}(q_i,r') + d_{\Gamma(N)}(r',\overline{z}) \leq (\rho_i-\sigma'+\overline\mu) + \overline\mu + \left( \sigma'+\left\lceil \frac{\xi_0}{N} \right\rceil \right) \\ &\leq \rho_i + \frac{5\overline\mu}{2} + \left\lceil \frac{\xi_0}{N} \right\rceil,
\end{align*}
as claimed.

We thus have $d_{\Gamma(N)}(x_i,\overline{z}) \leq \rho_i+\frac{5\overline\mu}{2}+\left\lceil \frac{\xi_0}{N} \right\rceil$ for all $i \in \mathcal{I}$. Hence $\bigcap_{i \in \mathcal{I}} B_{\rho_i+\frac{5\overline\mu}{2}+\left\lceil \frac{\xi_0}{N} \right\rceil}(x_i;\Gamma(N))$ contains $\overline{z}$ and so is non-empty, establishing the $\left( \frac{5\overline\mu}{2}+\left\lceil \frac{\xi_0}{N} \right\rceil \right)$-coarse Helly property for $\mathcal{B}$.
\end{proof}

Finally, we deduce the conclusions of Proposition~\ref{prop:helly-coarse} and Theorem~\ref{thm:helly}.

\begin{proof}[Proof of Proposition~\ref{prop:helly-coarse}]
By the choice of $\xi$ in \eqref{eq:defxi}, it follows from Lemmas \ref{lem:chelly-zout} and \ref{lem:chelly-zin} that $\mathcal{B}$ satisfies the $(\xi-4)$-coarse Helly property, and so $\mathcal{B}'$ satisfies the $\xi$-coarse Helly property. As $\mathcal{B}'$ was an arbitrary collection of balls in $\Gamma(N)$, the conclusion follows.
\end{proof}

\begin{proof}[Proof of Theorem~\ref{thm:helly}]
Suppose that the graphs $\Gamma_1,\ldots,\Gamma_m$ as above are Helly. Then they are clearly coarsely Helly; moreover, by \cite[Lemma 6.5]{ccgho}, each $\Gamma_j$ has $1$-stable intervals. Thus, the Theorem follows immediately from Propositions \ref{prop:helly-stint} and \ref{prop:helly-coarse} together with Theorem~\ref{thm:ccgho-helly}.
\end{proof}

\section{Quasiconvex subgroups of Helly groups}
\label{sec:quasiconvex}

In this section, we prove Theorem~\ref{thm:quasiconvex} from the Introduction, that is, we show that if a subgroup $H$ of a (coarsely) Helly group $G$ is, in a certain sense, quasiconvex in $G$, then $H$ is (coarsely) Helly.

\begin{defn} \label{defn:quasiconvex}
Let $\Gamma$ be a graph.
\begin{enumerate}[label=(\roman*)]
\item Given $\lambda \geq 1$ and $c \geq 0$, we say a subset $W \subseteq V(\Gamma)$ is \emph{$(\lambda,c)$-quasiconvex} if there exists $k = k(\lambda,c) \geq 0$ such that every $(\lambda,c)$-quasigeodesic in $\Gamma$ with endpoints in $W$ is in the $k$-neighbourhood of $W$.
\item Let $G$ be a group acting on $\Gamma$ geometrically. We say a subgroup $H \leq G$ is \emph{strongly quasiconvex} (respectively, \emph{semi-strongly quasiconvex}) \emph{with respect to $\Gamma$} if some $H$-orbit in $\Gamma$ is $(\lambda,c)$-quasiconvex (respectively, $(1,c)$-quasiconvex) for any $\lambda \geq 1$ and $c \geq 0$.
\end{enumerate}
\end{defn}

Throughout this section, we adopt the following terminology. Given a graph $\Gamma$ and $k \geq 1$, we construct a (Vietoris--Rips) graph $\Gamma_k$ with $V(\Gamma_k) = V(\Gamma)$, such that $v,w \in V(\Gamma)$ are adjacent in $\Gamma_k$ if and only if $d_\Gamma(v,w) \leq k$. Thus $\Gamma_1 = \Gamma$.

Since the collection of balls in $\Gamma_k$ is just a subcollection of the collection of balls in $\Gamma$ (for any $k \geq 1$), the following result is immediate.

\begin{lem} \label{lem:Gammak-Helly}
If a graph $\Gamma$ is Helly, then so is $\Gamma_k$ for any $k \geq 1$. \qed
\end{lem}

In order to prove Theorem~\ref{thm:quasiconvex}, we let $G$ act on a coarsely Helly graph $\Gamma$ geometrically, and consider the full subgraph $\Delta \subseteq \Gamma_k$ spanned by some $H$-orbit of vertices and their neighbours. We show in Lemma~\ref{lem:qconvex=>chelly} that (for $k$ large enough) the graph $\Delta$ is coarsely Helly. Since the $H$-action on $\Delta$ can be shown to be geometric, it follows that $H$ is coarsely Helly.

Moreover, if $\Gamma$ is Helly, then we consider the \emph{Hellyfication} $\Theta = \operatorname{Helly}(\Delta)$ of $\Delta$ (see \cite{ccgho}) -- the `smallest' Helly graph containing $\Delta$ as an isometrically embedded subgraph. We show in Lemma~\ref{lem:qconvex=>isometric} that (for $k$ large enough) $\Delta$ is an isometric subgraph of $\Gamma_k$, implying that $\Theta$ is a subgraph of $\Gamma_k$ as well (and so $\Theta$ is locally finite), whereas the fact that $\Delta$ is coarsely Helly implies that the induced action of $H$ on $\Theta$ is cobounded, and consequently geometric. Thus $H$ is Helly.

Thus, our proof of Theorem~\ref{thm:quasiconvex} is based on the following two Lemmas.

\begin{lem} \label{lem:qconvex=>chelly}
Let $\xi \geq 0$, let $\Gamma$ be a $\xi$-coarsely Helly graph, and $W \subseteq V(\Gamma)$ a $(1,2\xi)$-quasiconvex subset. Let $k \geq 1$ be such that every $(1,2\xi)$-quasigeodesic in $\Gamma$ with endpoints in $W$ is contained in the $k$-neighbourhood of $W$. Then the full subgraph $\Delta$ of $\Gamma_k$ spanned by $\bigcup_{w \in W} B_1(w;\Gamma_k)$ is $(3+\lceil\xi/k\rceil)$-coarsely Helly.
\end{lem}

\begin{proof}
Let $\mathcal{B} = \{ B_{\rho_i}(x_i;\Delta) \mid i \in \mathcal{I} \}$ be a collection of pairwise intersecting balls in $\Delta$. By the construction of $\Delta$, for each $i \in \mathcal{I}$ there exists $x_i' \in W$ such that $d_\Delta(x_i,x_i') \leq 1$, and in particular $B_{\rho_i}(x_i;\Delta) \subseteq B_{\rho_i+1}(x_i';\Delta) \subseteq B_{\rho_i+1}(x_i';\Gamma_k) = B_{(\rho_i+1)k}(x_i';\Gamma)$. It follows that $\mathcal{B}' = \{ B_{(\rho_i+1)k}(x_i';\Gamma) \mid i \in \mathcal{I} \}$ is a collection of pairwise intersecting balls.

Now fix any $j \in \mathcal{I}$, and consider $D = \{ d_\Gamma(x_i',x_j')-(\rho_i+1)k \mid i \in \mathcal{I} \} \subseteq \mathbb{Z}$; let $\delta = \sup D$. Since the balls in $\mathcal{B}'$ have pairwise non-empty intersections, we have $\delta \leq (\rho_j+1)k$ (and in particular $\delta = \max D$). If $\delta \leq 0$, it then follows that $d_\Gamma(x_i',x_j') \leq (\rho_i+1)k$ for all $i \in \mathcal{I}$; therefore, since $x_i',x_j' \in W$ and in particular any geodesic in $\Gamma$ between $x_i'$ and $x_j'$ is a path in $\Delta$, $d_\Delta(x_i',x_j') = \lceil d_\Gamma(x_i',x_j')/k \rceil \leq \rho_i+1$ for all $i \in \mathcal{I}$. Hence, $x_j' \in \bigcap_{i \in \mathcal{I}} B_{\rho_i+2}(x_i;\Delta)$, implying that $\mathcal{B}$ satisfies the $2$-coarse Helly property. Thus, we may assume that $\delta > 0$.

By the choice of $\delta$, we know that $\mathcal{B}'' = \{ B_{(\rho_i+1)k}(x_i';\Gamma) \mid i \in \mathcal{I} \setminus \{j\} \} \cup \{ B_\delta(x_j';\Gamma) \}$ is a collection of pairwise intersecting balls. It follows by the $\xi$-coarse Helly property that there exists $y \in V(\Gamma)$ such that $d_\Gamma(y,x_i') \leq (\rho_i+1)k+\xi$ for all $i \in \mathcal{I} \setminus \{j\}$ and $d_\Gamma(y,x_j') \leq \delta+\xi$.

Moreover, by the choice of $\delta$, there exists $\ell \in \mathcal{I}$ such that $d_\Gamma(x_j',x_\ell') = (\rho_\ell+1)k+\delta$. Therefore,
\begin{align*}
&d_\Gamma(y,x_j')+d_\Gamma(y,x_\ell')-d_\Gamma(x_j',x_\ell') \\ &\qquad \leq (\delta+\xi) + [(\rho_\ell+1)k+\xi] - [(\rho_\ell+1)k+\delta] = 2\xi,
\end{align*}
and so $y$ lies on a $(1,2\xi)$-quasigeodesic in $\Gamma$ from $x_j'$ to $x_\ell'$; thus, $y \in V(\Delta)$, implying that $d_\Gamma(y,y') \leq k$ for some $y' \in W$.

By the choice of $\Delta$, it follows that all geodesics in $\Gamma$ between $y'$ and $x_i'$ are paths in $\Delta$, and so $d_\Delta(y',x_i') = \lceil d_\Gamma(y',x_i')/k \rceil$ for each $i \in \mathcal{I}$. Moreover, since $\delta \leq (\rho_j+1)k$, we have $d_\Gamma(y,x_i') \leq (\rho_i+1)k+\xi$ for each $i \in \mathcal{I}$, and so
\[
d_\Delta(y',x_i') \leq \left\lceil \frac{d_\Gamma(y',y)+d_\Gamma(y,x_i')}{k} \right\rceil \leq \left\lceil \frac{k+[(\rho_i+1)k+\xi]}{k} \right\rceil = \rho_i+2+\left\lceil\frac{\xi}{k}\right\rceil
\]
for all $i \in \mathcal{I}$. It follows that
\[
d_\Delta(y',x_i) \leq d_\Delta(y',x_i') + d_\Delta(x_i',x_i) \leq \rho_i+3+\left\lceil\frac{\xi}{k}\right\rceil,
\]
which proves the $(3+\lceil\xi/k\rceil)$-coarse Helly property for $\mathcal{B}$.
\end{proof}

In the next Lemma, we say a graph $\Gamma$ is \emph{pseudo-modular} if for every triple $w_1,w_2,w_3 \in V(\Gamma)$ there exist geodesics $P_1,P_2,P_3,e_{1,2},e_{2,3},e_{3,1}$ in $\Gamma$, with $|e_{1,2}| = |e_{2,3}| = |e_{3,1}| \leq 1$, such that $(P_i)_- = w_i$, $(P_i)_+ = (e_{i,i+1})_- = (e_{i-1,i})_+$, and such that $P_ie_{i,i+1}\overline{P_{i+1}}$ is a geodesic in $\Gamma$ for $i \in \{1,2,3\}$ (with indices taken modulo $3$). It is well-known (see \cite[Proposition 4]{bandelt-mulder}) that a connected graph $\Gamma$ is pseudo-modular if and only if every triple $\{ B_1',B_2',B_3' \}$ of pairwise intersecting balls in $\Gamma$ has non-empty intersection (that is, satisfies the $0$-coarse Helly property). In particular, every Helly graph is pseudo-modular.

\begin{lem} \label{lem:qconvex=>isometric}
Let $\Gamma$ be a pseudo-modular graph, and $W \subseteq V(\Gamma)$ a $(5,0)$-quasiconvex subset. Let $k \geq 1$ be such that every $(5,0)$-quasigeodesic in $\Gamma$ with endpoints in $W$ is contained in the $k$-neighbourhood of $W$. Then the full subgraph $\Delta$ of $\Gamma_k$ spanned by $\bigcup_{w \in W} B_1(w;\Gamma_k)$ is isometrically embedded in $\Gamma_k$.
\end{lem}

\begin{figure}[ht]
\begin{tikzpicture}[very thick,decoration={markings,mark=at position 0.6 with {\arrow{>}}}]
\draw [postaction=decorate] (0,0) -- (1.5,1.5) node [midway,right] {$P_v'$};
\draw [postaction=decorate] (0,3.5) -- (1.5,2) node [midway,left] {$P_v$};
\draw [postaction=decorate] (1.5,1.5) -- (1.5,2) node [midway,left] {$e_1$};
\draw [postaction=decorate] (1.5,1.5) -- (2,1.75) node [midway,below] {$e_2$};
\draw [postaction=decorate] (1.5,2) -- (2,1.75) node [midway,above] {$e_3$};
\draw [postaction=decorate] (2,1.75) -- (4,1.75) node [midway,below] {$Q$};
\draw [postaction=decorate] (4.5,1.5) -- (6,0) node [midway,left] {$P_w'$};
\draw [postaction=decorate] (4.5,2) -- (6,3.5) node [midway,right] {$P_w$};
\draw [postaction=decorate] (4.5,1.5) -- (4.5,2) node [midway,right] {$f_1$};
\draw [postaction=decorate] (4,1.75) -- (4.5,1.5) node [midway,below] {$f_2$};
\draw [postaction=decorate] (4,1.75) -- (4.5,2) node [midway,above] {$f_3$};
\fill (0,0) circle (2pt) node [below] {$v'$};
\fill (0,3.5) circle (2pt) node [above] {$v$};
\fill (6,0) circle (2pt) node [below] {$w'$};
\fill (6,3.5) circle (2pt) node [above] {$w$};
\end{tikzpicture}
\caption{The proof of Lemma~\ref{lem:qconvex=>isometric}.}
\label{fig:isometric}
\end{figure}
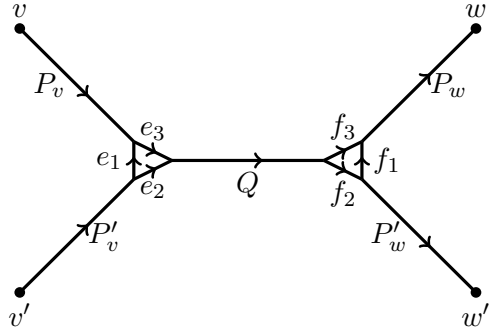

\begin{proof}
Suppose for contradiction that $\Delta$ is not isometrically embedded in $\Gamma_k$. Thus, there exist $v,w \in V(\Delta)$ such that $d_\Delta(v,w) > d_{\Gamma_k}(v,w)$; without loss of generality, assume that $v$ and $w$ are chosen in such a way that $d_{\Gamma_k}(v,w)$ is as small as possible. By the definition of $\Delta$, there exist $v',w' \in W$ such that $d_\Gamma(v,v') \leq k$ and $d_\Gamma(w,w') \leq k$.

Since $\Gamma$ is pseudo-modular, there exist geodesics $P_v'e_1\overline{P_v}$, $P_v'e_2Q'$ and $P_ve_3Q'$ in $\Gamma$ from $v'$ to $v$, from $v'$ to $w$ and from $v$ to $w$, respectively, such that $|e_1| = |e_2| = |e_3| \leq 1$. Similarly, there exist geodesics $\overline{P_w'}f_1P_w$, $Qf_2P_w'$ and $Qf_3P_w$ in $\Gamma$ from $w'$ to $w$, from $Q'_-$ to $w'$ and from $Q'_-$ to $w$, respectively, such that $|f_1| = |f_2| = |f_3| \leq 1$: see Figure~\ref{fig:isometric}. Note that $Q'$ is a geodesic in $\Gamma$ with the same endpoints as the geodesic $Qf_3P_w$: therefore, since $P_v'e_2Q'$ and $P_ve_3Q'$ are geodesics in $\Gamma$, so are $P_v'e_2Qf_3P_w$ and $P_ve_3Qf_3P_w$. In particular, the paths $P_v'e_1\overline{P_v}$, $\overline{P_w'}f_1P_w$, $P_v'e_2Q$, $Qf_2P_w'$ and $P_ve_3Qf_3P_w$ are all geodesics in $\Gamma$.

Let $p = d_\Gamma(v,w) - k \left[ d_{\Gamma_k}(v,w)-1 \right]$, so that $1 \leq p \leq k$. We claim that $|P_ve_3| < |P_v'e_2| - k + p$. Indeed, let $u \in P_ve_3Qf_3P_w$ be the vertex with $d_\Gamma(v,u) = p$; it follows from the choice of $p$ and the fact that $P_ve_3Qf_3P_w$ is a geodesic in $\Gamma$ that $u$ lies on a geodesic in $\Gamma_k$ between $v$ and $w$. Since $u$ is adjacent to $v$ in $\Gamma_k$ and since $\Delta$ is a full subgraph, it follows by the minimality of $d_{\Gamma_k}(v,w)$ that $u \notin V(\Delta)$; therefore, $d_\Gamma(v',u) > k$. In particular, since $d_\Gamma(v',v) \leq k$ and since $P_v' e_1 \overline{P_v}$ is a geodesic in $\Gamma$, this implies that $u \notin P_v$ and so $u \in Qf_3P_w$. We thus have
\[
k < d_\Gamma(v',u) \leq |P_v'e_2|+d_\Gamma(Q_-,u) = |P_v'e_2| + (p-|P_ve_3|),
\]
and so $|P_ve_3| < |P_v'e_2| - k + p$, as claimed. Similarly, $|f_3P_w| < |f_2P_w'| - k + p$.

We now claim that $|Q| \geq \frac{k}{2}$. Indeed, since $p \leq k$ and $|e_2| \leq 1$, the previous paragraph implies that $|P_ve_3| \leq |P_v'|$, and hence
\[
2|P_ve_3| \leq |P_ve_3| + |P_v'| = |P_v| + |e_1| + |P_v'| = d_\Gamma(v',v) \leq k
\]
since $P_v'e_1\overline{P_v}$ is a geodesic in $\Gamma$; similarly, $2|f_3P_w| \leq k$. Therefore, if $d_\Gamma(v,w) \geq \frac{3k}{2}$ then we have
\[
|Q| = d_\Gamma(v,w) - |P_ve_3| - |f_3P_w| \geq \frac{3k}{2} - \frac{k}{2} - \frac{k}{2} = \frac{k}{2},
\]
as claimed. On the other hand, suppose that $d_\Gamma(v,w) < \frac{3k}{2}$. It then follows that $d_{\Gamma_k}(v,w) = 2$ (as $v$ and $w$ cannot be adjacent in $\Gamma_k$), and so $p < \frac{k}{2}$ and $d_\Gamma(v,w) = k+p$. Then, again by the previous paragraph, we have
\begin{equation} \label{eq:lenQk2}
\begin{aligned}
|Q| &= d_\Gamma(v,w) - |P_ve_3| - |f_3P_w| \\ &> (k+p) - (|P_v'e_2|-k+p) - (|f_2P_w'|-k+p) \\ &= 3k-p-|P_v'e_2|-|f_2P_w'|.
\end{aligned}
\end{equation}
Furthermore, we have $|P_v'e_2| = |P_v'|+|e_1| \leq |P_v'e_1\overline{P_v}| = d_\Gamma(v',v) \leq k$ since $P_v'e_1\overline{P_v}$ is a geodesic in $\Gamma$, and similarly $|f_2P_w'| \leq k$. Therefore, \eqref{eq:lenQk2} implies that $|Q| > k-p$. As $p < \frac{k}{2}$, it follows that $|Q| \geq \frac{k}{2}$, as claimed.

Finally, we claim that $P_v'e_2Qf_2P_w'$ is a $(5,0)$-quasigeodesic in $\Gamma$. Indeed, let $R \subseteq P_v'e_2Qf_2P_w'$ be a subpath: we thus claim that $|R| \leq 5d_\Gamma(R_-,R_+)$. Since $P_v'e_2Q$ and $Qf_2P_w'$ are geodesics in $\Gamma$, we may assume, without loss of generality, that $R_+ \in f_2P_w'$ and $R_- \in P_v'e_2$. Moreover, assume that $d_\Gamma(R_+,Q_+) \leq d_\Gamma(R_-,Q_-)$: the other case is analogous. We then have
\begin{align*}
d_\Gamma(R_-,R_+) &\geq d_\Gamma(R_-,Q_+)-d_\Gamma(R_+,Q_+) \\ &= d_\Gamma(R_-,Q_-) + |Q| - d_\Gamma(R_+,Q_+) \geq |Q|
\end{align*}
since $P_v'e_2Q$ is a geodesic in $\Gamma$. On the other hand, since $P_v'e_1\overline{P_v}$ is a geodesic in $\Gamma$ of length $\leq k$, we have $|P_v'e_2| = |P_v'| + |e_1| \leq k$, and similarly $|f_2P_w'| \leq k$. It follows that $|R| \leq |P_v'e_2Qf_2P_w'| = |P_v'e_2| + |Q| + |f_2P_w'| \leq |Q|+2k$, and so, since $|Q| \geq \frac{k}{2}$, we have $|R| \leq |Q|+2k \leq 5|Q| \leq 5d_\Gamma(R_-,R_+)$, as claimed.

But now, by the choice of $k$, it follows that $P_v'e_2Qf_2P_w'$ is in the $k$-neighbourhood of $W$ in $\Gamma$, and so, in particular, all the vertices of $Q$ belong to $\Delta$. Since $P_v'e_1\overline{P_v}$ and $\overline{P_w'}f_1P_w$ are geodesics in $\Gamma$ of length $\leq k$, it also follows that all vertices of $P_v$ and of $P_w$ belong to $\Delta$. Therefore, all the vertices of $P_ve_3Qf_3P_w$ belong to $\Delta$. But as $P_ve_3Qf_3P_w$ is a geodesic in $\Gamma$, there exists a geodesic in $\Gamma_k$ from $v$ to $w$ all of whose vertices are also vertices of $P_ve_3Qf_3P_w$, and so of $\Delta$. This contradicts the choice of $v$ and $w$.

Therefore, $\Delta$ must be an isometrically embedded subgraph of $\Gamma_k$, as required.
\end{proof}

\begin{proof}[Proof of Theorem~\ref{thm:quasiconvex}]
We first make the following elementary observation.

\begin{cm}
Let $\Theta$ be a locally finite graph, and $H$ a group acting on $\Theta$. Suppose that there exists $x \in V(\Theta)$ such that $|\Stab_H(x)| < \infty$ and the $H$-orbit $x \cdot H$ is finite Hausdorff distance away from $V(\Theta)$. Then the $H$-action on $\Theta$ is geometric.
\end{cm}

\begin{cmproof}
Since $\Theta$ is locally finite, it is enough to show that the $H$-action is cocompact and $|\Stab_H(y)| < \infty$ for all $y \in V(\Theta)$.

Let $k < \infty$ be the Hausdorff distance from $x \cdot H$ to $V(\Theta)$. If $e \subseteq \Theta$ is an edge, then there exists $h \in H$ such that $d_\Theta(e'_-,x) \leq k$, and so $d_\Theta(e'_\pm,x) \leq k+1$, where $e' = e \cdot h$. But since $\Theta$ is locally finite, there exist only finitely many edges $e' \subseteq \Theta$ with $d_\Theta(e'_\pm,x) \leq k+1$, and so there are only finitely many of orbits of edges under the $H$-action on $\Theta$. Thus the action is cocompact, as required.

Now let $y \in V(\Theta)$. Then there exists $h \in H$ such that $d_\Theta(y',x) \leq k$ where $y' = y \cdot h$; in particular, if $g \in \Stab_H(y')$ then $d_\Theta(y',x \cdot g) = d_\Theta(y',x) \leq k$, and so $x \cdot g \in B_k(y';\Theta)$. Since $\Theta$ is locally finite, this implies that
\[
|\Stab_H(y)| = |h\Stab_H(y')h^{-1}| = |\Stab_H(y')| \leq |B_k(y';\Theta)| \times |\Stab_H(x)| < \infty,
\]
as required.
\end{cmproof}

We now prove parts \ref{it:qconvex-helly} and \ref{it:qconvex-coarse} of the Theorem.

\begin{enumerate}[label=(\roman*)]

\item By assumption, there exists $x \in V(\Gamma)$ such that the orbit $x \cdot H \subseteq V(\Gamma)$ is $(5,0)$-quasiconvex. Let $k \geq 1$ be a constant such that every $(5,0)$-quasigeodesic with endpoints in $x \cdot H$ is in the $k$-neighbourhood of $x \cdot H$, and let $\Delta$ be the full subgraph of $\Gamma_k$ spanned by $\bigcup_{h \in H} B_1(x \cdot h; \Gamma_k)$.

Let $\Theta = \operatorname{Helly}(\Delta)$ be the Hellyfication of $\Delta$, as defined in \cite[\S 4.2]{ccgho}. By \cite[Theorem 4.4]{ccgho}, $\Theta$ is a Helly graph contained as a subgraph in every Helly graph containing $\Delta$ as an isometrically embedded subgraph. As $\Gamma$ is Helly, it is also pseudo-modular, and so by Lemma~\ref{lem:qconvex=>isometric}, $\Delta$ is isometrically embedded in $\Gamma_k$; moreover, by Lemma~\ref{lem:Gammak-Helly}, $\Gamma_k$ is Helly. It follows that $\Theta$ is (isomorphic to) a subgraph of $\Gamma_k$. But since $\Gamma$ is locally finite, so is $\Gamma_k$ and therefore so is $\Theta$. Thus, $\Theta$ is a locally finite Helly graph, and so it is enough to show that $H$ acts on $\Theta$ geometrically.

Now the $G$-action on $\Gamma$ induces a $G$-action on $\Gamma_k$, with respect to which $\Delta$ is clearly $H$-invariant. Furthermore, the $H$-action on $\Delta$ extends to an $H$-action on $\Theta$: see \cite{ccgho}. By Lemma~\ref{lem:qconvex=>chelly} (applied in this case with $\xi = 0$), $\Delta$ is coarsely Helly. It follows by \cite[Proposition~3.12]{ccgho} that $V(\Theta)$ is Hausdorff distance $\ell < \infty$ away from $V(\Delta)$, and so Hausdorff distance $\leq \ell+1$ away from $x \cdot H$ (in $\Theta$). Moreover, since the action of $G$ on $\Gamma$, and so on $\Gamma_k$, is geometric, we have $|\Stab_H(x)| \leq |\Stab_G(x)| < \infty$. It follows from the Claim that the action of $H$ on $\Theta$ is geometric, as required.

\item Let $\xi \geq 0$ be such that $\Gamma$ is $\xi$-coarsely Helly. By assumption, there exists $x \in V(\Gamma)$ such that the orbit $W := x \cdot H \subseteq V(\Gamma)$ is $(1,2\xi)$-quasiconvex. Let $k \geq 1$ be the constant and $\Delta \subseteq \Gamma_k$ the subgraph given by Lemma~\ref{lem:qconvex=>chelly}. Since $\Gamma$ is locally finite, so is $\Gamma_k$; moreover, the $G$-action on $\Gamma$ induces a $G$-action on $\Gamma_k$. By construction, the subgraph $\Delta \subseteq \Gamma_k$ is $H$-invariant and is Hausdorff distance $\leq 1$ away from the $H$-orbit $W = x \cdot H$. As a subgraph of $\Gamma_k$, $\Delta$ is also locally finite; furthermore, $|\Stab_H(x)| \leq |\Stab_G(x)| < \infty$. It follows from the Claim that the $H$-action on $\Delta$ is geometric. But by Lemma~\ref{lem:qconvex=>chelly}, $\Delta$ is coarsely Helly, as required. \qedhere

\end{enumerate}
\end{proof}

\bibliographystyle{amsalpha}
\bibliography{ref}

\end{document}